\documentclass[twoside]{article}

\usepackage[accepted]{aistats2024}
\usepackage[round]{natbib}
\usepackage{amssymb}
\usepackage{amsmath}
\usepackage{amsthm}
\usepackage{mathtools,xparse}

\usepackage{pdfpages}

\usepackage{nicefrac}
\usepackage{thmtools}

\usepackage{graphicx}
\usepackage[T1]{fontenc}
\usepackage{amsfonts}
\usepackage[version=4]{mhchem}
\usepackage{stmaryrd}
\usepackage[export]{adjustbox}
\graphicspath{ {./images/} }

\usepackage{hyperref}
\usepackage{cleveref}
\hypersetup{colorlinks,linkcolor={blue},citecolor={niceblue},urlcolor={blue}}
\usepackage{bbold}
\usepackage{algpseudocode}
\usepackage{algorithm}

\usepackage[flushleft]{threeparttable} % http://ctan.org/pkg/threeparttable
\usepackage{caption}
\usepackage{multirow}
\usepackage{colortbl}
\usepackage{makecell}
\usepackage{subcaption}
\usepackage[toc,page]{appendix}
\usepackage{xcolor}

\newcommand{\cB}{{\cal B}}
\newcommand{\cC}{{\cal C}}

\newcommand{\cG}{{\cal G}}

\newcommand{\cL}{{\cal L}}

\newcommand{\cO}{{\cal O}}

\newcommand{\cQ}{{\cal Q}}

\newcommand{\Prob}[1]{\mathbb{P}\left( #1 \right)}
\newcommand{\E}{\mathbb{E}}
\newcommand{\EE}[1]{\mathbb{E}\left[ #1 \right]}

\newcommand{\EEC}[1]{\mathbb{E}_{\cC}\left[ #1 \right]}

\newcommand{\EEh}[1]{\mathbb{E}_{h}\left[ #1 \right]}

\newcommand{\algname}[1]{{\color{black}\small \sf #1}}

\newcommand{\R}{\mathbb{R}}
\newcommand{\eqdef}{\stackrel{\mathrm{def}}{=}}
\newcommand{\hx}{\widehat{x}}

\newcommand{\hdelta}{\widehat{\Delta}}
\newcommand{\parens}[1]{\left( #1 \right)}
\newcommand{\brac}[1]{\left\{ #1 \right\}}
\newcommand{\norm}[1]{\left\| #1 \right\|}

\newcommand{\vertiii}[1]{{\left\vert\kern-0.25ex\left\vert\kern-0.25ex\left\vert #1 
    \right\vert\kern-0.25ex\right\vert\kern-0.25ex\right\vert}}

\newcommand{\Z}{\mathbb{Z}}

\usepackage{pifont}
\definecolor{PineGreen}{RGB}{0,110,51}
\definecolor{BrickRed}{RGB}{143,20,2}
\definecolor{bgcolor}{rgb}{0.8,1,1}
\definecolor{bgcolor2}{rgb}{0.8,1,0.8}
\definecolor{niceblue}{rgb}{0.0,0.19,0.56}
\definecolor{mydarkgreen}{RGB}{39,130,67}

\definecolor{mydarkred}{RGB}{192,47,25}

\newcommand{\argmin}{\mathop{\arg\!\min}}

\newtheorem{lemma}{Lemma}[section]
\newtheorem{theorem}{Theorem}[section]
\newtheorem{definition}{Definition}[section]

\newtheorem{assumption}{Assumption}[section]
\newtheorem{corollary}{Corollary}[section]

\allowdisplaybreaks

\begin{document}

\runningtitle{Communication Compression for Byzantine Robust Learning}


\twocolumn[

\aistatstitle{Communication Compression for Byzantine Robust Learning: \\ New Efficient Algorithms and Improved Rates}

\aistatsauthor{Ahmad Rammal\And Kaja Gruntkowska \And  Nikita Fedin \And Eduard Gorbunov \And Peter Richt\'arik }

\aistatsaddress{KAUST \\ \'Ecole Polytechnique \And  KAUST \And MIPT \And MBZUAI \And KAUST} ]

\begin{abstract}
    Byzantine robustness is an essential feature of algorithms for certain distributed optimization problems, typically encountered in collaborative/federated learning.  These problems are usually huge-scale, implying that communication compression is also imperative for their resolution. These factors have spurred recent algorithmic and theoretical developments in the literature of Byzantine-robust learning with compression. In this paper, we contribute to this research area in two main directions. First, we propose a new Byzantine-robust method with compression -- \algname{Byz-DASHA-PAGE} -- and prove that the new method has better convergence rate (for non-convex and Polyak-{\L}ojasiewicz smooth optimization problems), smaller neighborhood size in the heterogeneous case, and tolerates more Byzantine workers under over-parametrization than the previous method with SOTA theoretical convergence guarantees (\algname{Byz-VR-MARINA}). Secondly, we develop the first Byzantine-robust method with communication compression and error feedback -- \algname{Byz-EF21} -- along with its bidirectional compression version -- \algname{Byz-EF21-BC} -- and derive the convergence rates for these methods for non-convex and Polyak-{\L}ojasiewicz smooth case. We test the proposed methods and illustrate our theoretical findings in the numerical experiments.

    % Byzantine robustness is an essential feature of algorithms addressing specific distributed optimization problems, especially in collaborative/federated learning contexts. Such problems are usually huge-scale, making communication compression equally crucial for their efficient resolution. These factors have spurred recent algorithmic and theoretical developments in the field of Byzantine-robust learning with compression. In this paper, we contribute to this research area in two main directions. First, we propose a new method for Byzantine-robust learning with compression -- \algname{Byz-DASHA-PAGE} -- and prove that it offers superior convergence rates for non-convex and Polyak-{\L}ojasiewicz smooth optimization problems, converges to a smaller neighborhood in the heterogeneous case, and tolerates more Byzantine workers under over-parametrization than the previous method with SOTA theoretical convergence guarantees (\algname{Byz-VR-MARINA}). Secondly, we develop the first Byzantine-robust method with biased communication compression and error feedback -- \algname{Byz-EF21} -- along with its bidirectional compression version -- \algname{Byz-EF21-BC} -- and derive the convergence rates of these methods in non-convex and Polyak-{\L}ojasiewicz smooth case. We test the proposed algorithms and illustrate our theoretical findings in numerical experiments.
\end{abstract}

\section{INTRODUCTION}

Contemporary machine learning and deep learning pose a number of challenges, with the ability to train increasingly complex models using datasets of enormous sizes becoming  one of the most pressing issues \citep{gpt4}. Training such models on a single machine within a reasonable timeframe is no longer feasible \citep{gpt3costlambda}. In response to this challenge, distributed algorithms have emerged as indispensable tools, effectively sharing the computational load across multiple machines, and hence significantly speeding up the training process. Such methods also prove invaluable when data is inherently distributed across multiple sources or locations. When this is the case, the adoption of distributed methods is not only a natural choice, but often an imperative one \citep{FEDLEARN,kairouz2021advances}.

While distributed learning comes with a number of benefits, it also introduces some risks. In collaborative and federated learning scenarios, these include the potential presence of \emph{Byzantine workers}\footnote{Following the standard terminology \citep{lamport1982byzantine, su2016fault}, we call a worker \emph{Byzantine} if it can (maliciously or not) send incorrect information to other workers/server. Such workers are assumed to be omniscient, i.e., they have access to the vectors that other workers send, know the aggregation rule on the server, and can coordinate their actions with one another.}. Standard methods such as Parallel Stochastic Gradient Descent (\algname{SGD}) \citep{zinkevich2010parallelized}, which rely on averaging vectors received from workers, are highly vulnerable to Byzantine attacks. Consequently, a critical need has emerged for the development and investigation of specialized methods designed to demonstrate robustness when Byzantine participants are involved, giving Byzantine-robustness significant attention in recent years \citep{lyu2020privacy}.

Another important aspect of distributed learning is managing communication costs. Indeed, communication between the nodes typically constitutes a significant portion of the time and resource consumption of the training process. Consequently, as the field of machine learning continues to leverage larger and more complex models trained on extensive datasets, the need to efficiently exchange information between nodes becomes paramount. Communication compression techniques such as quantization \citep{alistarh2017qsgd} and sparsification \citep{suresh2017distributed, stich2018sparsified}, play a pivotal role in addressing this challenge.

While both Byzantine robustness and communication compression are individually highly significant topics, their simultaneous exploration is relatively rare in the exisitng literature. To date, only five papers have tackled both of these challenges concurrently: \citet{bernstein2018signsgd} study \algname{signSGD} with majority vote, \citet{ghosh2020distributed, ghosh2021communication} propose methods utilising aggregation rules that select the update vectors based on their norms, and \citet{zhu2021broadcast, gorbunov2023variance} develop variance-reduced methods. The current state-of-the-art theoretical results in this area are derived by \citet{gorbunov2023variance}, who propose \algname{Byz-VR-MARINA} -- an algorithm with provably robust aggregation \citep{karimireddy2021learning, karimireddy2020byzantine} based on the \algname{SARAH}-type variance reduction \citep{nguyen2017sarah}, employing unbiased compression of the stochastic gradient differences \citep{gorbunov2021marina}.

The existing results have certain limitations. In particular, although \algname{Byz-VR-MARINA} achieves state-of-the-art convergence rates, it requires occasional communication of uncompressed messages. Further, it has inferior theoretical guarantees for optimization error in the heterogeneous case, and tolerates less Byzantine workers in the heterogeneous over-parameterized regime compared to some other existing methods, such as Byzantine-Robust \algname{SGD} with momentum (\algname{BR-SGDm}) \citep{karimireddy2020byzantine}. Moreover, existing Byzantine-robust methods only support the use of unbiased compressors. Meanwhile, it is known that employing (typically biased) contractive compressors combined with error feedback -- a powerful technique introduced in the communication compression literature \citep{seide20141, richtarik2021ef21} -- often yields superior empirical performance. \emph{Our work comprehensively addresses all these limitations}.

\subsection{Technical Preliminaries}

In this paper, we consider a standard distributed optimization problem in which both the objective function $f$ and the functions $f_i$ stored on the nodes have a finite-sum structure:
\begin{align}\label{eq:main_problem}
	\small &\min\limits_{x\in \R^d} \left\{ f(x) = \frac{1}{G}\sum\limits_{i\in \cG} f_i(x)\right\},\\
    &\text{where}\quad f_i(x) = \frac{1}{m}\sum\limits_{j=1}^m f_{i,j}(x) \quad \forall i \in \cG, \nonumber
\end{align}
where $\cG$ is the set of \emph{good/regular/non-Byzantine} clients, $|\cG| = G$, $f_i(x)$ corresponds to the loss of the model $x$ on the data stored on worker $i$, and $f_{i,j}(x)$ is the loss on the $j$-th example from the local dataset of worker $i$. In addition to regular workers, there is a set of \emph{bad/malicious/Byzantine} workers, denoted as $\cB$, also participating in the training process. For notational convenience, we assume that $\cG \cup \cB = [n] = \{1,2,\ldots, n\}$. We refrain from making any assumptions on the behaviour of the Byzantine workers, but we do impose a constraint on their number, requiring them to constitute less than half of all clients.

\paragraph{Robust aggregation.} One of the key ingredients of our methods is a $(\delta, c)$-Robust Aggregator, initially introduced by \citet{karimireddy2021learning, karimireddy2020byzantine}. We use the generalized version from \citet{gorbunov2023variance}.

\begin{definition}[$(\delta, c)$-Robust Aggregator]\label{def:RAgg_def}
	Assume that $\{x_1,x_2,\ldots,x_n\}$ is such that there exists a subset $\cG \subseteq [n]$ of size $|\cG| = G \geq (1-\delta)n$ with $\delta < 0.5$ and there exists $\sigma \geq 0$ such that $\frac{1}{G(G-1)}\sum_{i,l\in \cG}\EE{\norm{x_i-x_l}^2} \leq \sigma^2$, where the expectation is taken w.r.t.\ the randomness of $\{x_i\}_{i\in \cG}$. We say that the quantity $\hx$ is a $(\delta, c)$-Robust Aggregator ($(\delta, c)$-\texttt{RAgg}) and write $\hx = \texttt{RAgg}(x_1,\ldots,x_n)$ for some $c> 0$, if the following inequality holds:
	\begin{equation}
		\EE{\|\hx - \overline{x}\|^2} \leq c\delta \sigma^2, \label{eq:RAgg_def}
	\end{equation}
	where $\overline{x} = \tfrac{1}{|\cG|}\sum_{i\in \cG} x_i$. If additionally $\hx$ is computed without the knowledge of $\sigma^2$, we say that $\hx$ is a $(\delta, c)$-Agnostic Robust Aggregator ($(\delta, c)$-\texttt{ARAgg}) and write $\hx = \texttt{ARAgg}(x_1,\ldots,x_n)$.
\end{definition}

In essence, an aggregator is regarded as \emph{robust} if it closely approximates the average of regular vectors. Specifically, the upper bound on the expected squared distance between the two quantities should be proportional to the pairwise variance of the non-malicious vectors, and the upper bound on the proportion of Byzantine workers. In terms of this criterion, this definition is tight \citep{karimireddy2021learning}. Some examples of $(\delta, c)$-robust aggregators are provided in Appendix~\ref{appendix:robust_aggr_and_compr}.

\paragraph{Communication compression.} We focus on two main classes of compression operators: \emph{unbiased} and \emph{biased} (contractive) compressors.

\begin{definition}[Unbiased compressor]\label{def:quantization}
	A stochastic mapping $\cQ:\R^d \to \R^d$ is called an \emph{unbiased compressor/compression operator} if there exists  $\omega \geq 0$ such that for any $x\in\R^d$
	\begin{equation}
		\EE{\cQ(x)} = x,\quad \EE{\|\cQ(x) - x\|^2} \leq \omega\|x\|^2. \label{eq:quantization_def}
	\end{equation}
	% For the given unbiased compressor $\cQ(x)$, the expected density is $\zeta_{\cQ} = \sup_{x\in\R^d}\EE{\left\|\cQ(x)\right\|_0},$ where $\|y\|_0$ is the number of non-zero components of $y\in\R^d$.
\end{definition}

This definition encompasses a wide range of well-known compression techniques, including Rand$K$ sparsification \citep{stich2018sparsified}, random dithering \citep{goodall1951television, roberts1962picture} and natural compression \citep{horvath2019natural}. However, it does not cover another important class of compression operators, called contractive compressors, which are usually biased.

\begin{definition}[Contractive compressor]\label{def:contractive_compr}
	A stochastic mapping $\cC:\R^d \to \R^d$ is called a \emph{contractive compressor/compression operator} if there exists  $\alpha \in [0,1)$ such that for any $x\in\R^d$
	\begin{equation}
		\EE{\|\cC(x) - x\|^2} \leq (1-\alpha)\|x\|^2. \label{eq:contractive_compr}
	\end{equation}
\end{definition}

One of the most popular examples of contractive compressors is the Top$K$ sparsification \citep{alistarh2018convergence}.

We shall denote the families of compressors satisfying Definitions \ref{def:quantization} and \ref{def:contractive_compr} by $\mathbb{U}(\omega)$ and $\mathbb{B}(\alpha)$ respectively. Notably, it can easily be verified that if $\cC \in \mathbb{U}(\omega),$ then $(\omega + 1)^{-1} \cC \in \mathbb{B}\left((\omega + 1)^{-1}\right)$ \citep{beznosikov2020biased}, so the family of biased compressors is wider.

Some examples of such mappings are provided in Appendix~\ref{appendix:robust_aggr_and_compr}. For a comprehensive overview of biased and unbiased compressors, we refer to the summary in \citep{beznosikov2020biased}.

\paragraph{Assumptions.} We start with formulating the standard smoothness assumption.

\begin{assumption}\label{as:smoothness}
The function $f:\R^d \to \R$ is $L$-smooth, meaning that for any $x,y \in \R^d$ it satisfies $\|\nabla f(x) - \nabla f(y)\| \leq L\|x - y\|$. In addition, we assume that $f_* = \inf_{x\in\R^d}f(x) > -\infty$.
\end{assumption}

Our analysis also relies on a special notion of smoothness of loss functions stored on regular workers.

\begin{assumption}[Global Hessian variance \citep{szlendak2021permutation}]\label{as:hessian_variance}
	There exists $L_{\pm} \ge 0$ such that for all $x,y \in \R^d$
	\begin{align} \label{eq:hessian_variance}
		\small \frac{1}{G}\sum\limits_{i\in \cG} \|\nabla f_i(x) - \nabla f_i(y)\|^2 -& \|\nabla f(x) - \nabla f(y)\|^2 \nonumber \\
        &\leq L_{\pm}^2\|x - y\|^2.
	\end{align}
\end{assumption}
It can be verified that the above assumption is always valid for some $L_{\pm} \geq 0$ whenever $f_i$ is $L_i$-smooth for all $i\in \cG$. More precisely, if this is the case, then there exists $L_{\pm}$ satisfying the above assumption such that $L_{\text{avg}}^2 - L^2 \leq L_{\pm}^2 \leq L_{\text{avg}}^2$, where $L_{\text{avg}}^2 = \frac{1}{G}\sum_{i\in \cG} L_i^2$ \citep{szlendak2021permutation}. Moreover, there exist problems with heterogeneous data such that \eqref{eq:hessian_variance} holds with $L_{\pm} = 0$, while $L_{\text{avg}} > 0$ \citep{szlendak2021permutation}.

The vast majority of existing works on Byzantine-robustness focus solely on the standard uniform sampling. To be able to consider a wider range of samplings, we also employ an assumption on the (expected) Lipschitzness for samplings of stochastic gradients.

\begin{assumption}[Local Hessian variance \citep{gorbunov2023variance}]\label{as:hessian_variance_local}
	There exists $\cL_{\pm} \ge 0$ such that for all $x,y \in \R^d$ the unbiased mini-batched estimator $\widehat{\Delta}_i(x,y)$ of $\Delta_i(x,y) = \nabla f_i(x) - \nabla f_i(y)$ with batch size $b$ satisfies
	\begin{equation}
		\frac{1}{G}\sum\limits_{i\in \cG} \EE{\|\widehat{\Delta}_i(x,y) - \Delta_i(x,y)\|^2} \leq \frac{\cL_{\pm}^2}{b}\|x - y\|^2. \label{eq:hessian_variance_local}
	\end{equation}
\end{assumption}

As shown in \citep[Appendix E.1]{gorbunov2023variance}, the above assumption is quite general, encompassing scenarios such as uniform and importance sampling of stochastic gradients.

Finally, to ensure provable Byzantine-robustness, it is necessary to introduce an assumption on heterogeneity among regular workers. Otherwise, Byzantine workers can transmit arbitrary vectors, pretending to have access to non-representative data, thus becoming undetectable.

\begin{assumption}[$(B,\zeta^2)$-heterogeneity]\label{as:bounded_heterogeneity}
    There exist $B \geq 0$ and $\zeta \geq 0$ such that for all $x \in \R^d$ the local loss functions of the good workers satisfy
	\begin{equation}
		\frac{1}{G}\sum\limits_{i\in \cG} \|\nabla f_i(x) - \nabla f(x)\|^2 \leq B\|\nabla f(x)\|^2 + \zeta^2.
    \label{eq:bounded_heterogeneity}	
	\end{equation}
\end{assumption}
This assumption is the most general one of its kind in the literature on Byzantine-robustness. In particular, when $B = 0$, it reduces to the standard bounded gradient dissimilarity assumption, and when $\zeta = 0$, it implies that all stationary points of $f$ are also stationary points of each function $f_i$. The latter typically occurs in over-parameterized models \citep{vaswani2019fast}. Finally, it is worth mentioning that even the homogeneous case ($B = 0$, $\zeta = 0$) where all workers $i\in \cG$ have access to the same local function $f_i = f$ is relevant, especially in collaborative learning scenarios, when data is openly shared and the aim is to speed up the training process \citep{diskin2021distributed, kijsipongse2018hybrid}.

\subsection{Our Contributions}
\begin{table*}[t]
    \centering
    \caption{\small Summary of the derived complexity bounds in the general non-convex case and a comparison with the complexity of \algname{Byz-VR-MARINA}. Columns: ``Rounds'' = the number of communication rounds required to find $x$ such that $\EE{\norm{\nabla f(x)}^2} \leq \varepsilon^2$; ``$\varepsilon \leq$'' = the lower bound for the best achievable accuracy $\varepsilon$; ``$\delta <$'' = the maximal ratio of Byzantine workers that the method can provably tolerate. Only dependencies with respect to the following variables are shown: $\omega =$ unbiased compression parameter, $n =$ number of workers, $m =$ number of local functions, $b =$ batch size, $G =$ number of regular workers, $c =$ aggregation constant, $\delta =$ percentage of Byzantine workers, $\alpha_D, \alpha_P =$ contractive compression parameters for uplink and downlink compression, respectively. The results derived in this paper are highlighted in light blue color; red color indicates terms in the complexity bound/lower bound for $\varepsilon$/upper bound for $\delta$ for \algname{Byz-VR-MARINA} that we improve in our work.}
    \label{tab:comparison_of_rates_minimization}
    \begin{threeparttable}
    \small
        \begin{tabular}{|c c c c|}
        \hline
        Method & Rounds & $\varepsilon \leq $ & $\delta < $\\
        \hline\hline
        \begin{tabular}{c}
              \algname{Byz-VR-MARINA}\tnote{\color{blue}(1)}\\
              \citep{gorbunov2023variance} 
        \end{tabular}
        & $\frac{1}{\varepsilon^2}\!\parens{\!1+{\color{BrickRed}\sqrt{\max \{(1+\omega)^2, \frac{m(1+\omega)}{b}\}}}\!\parens{\!\sqrt{\frac{1}{G}}\!+\!{\sqrt{c\delta\color{BrickRed}\max \{\omega, \frac{m}{b}\}}}}\!}$ & $\frac{c\delta\zeta^2}{{\color{BrickRed}p} - c\delta B}$ & $\frac{{\color{BrickRed}p}}{cB}$ \\
        \cellcolor{bgcolor} \begin{tabular}{c}
             \algname{Byz-VR-MARINA 2.0} \tnote{\color{blue}(1)}
        \end{tabular}  & \cellcolor{bgcolor}$\frac{1}{\varepsilon^2}\parens{1+{\color{BrickRed}\sqrt{\max \{(1+\omega)^2, \frac{m(1+\omega)}{b}\}}}\parens{\sqrt{\frac{1}{G}} + \sqrt{c\delta}}}$ &\cellcolor{bgcolor} $\frac{c\delta\zeta^2}{1 - c\delta B}$&\cellcolor{bgcolor} $\frac{1}{(c+\sqrt{c})B}$ \\
        \cellcolor{bgcolor} \begin{tabular}{c}
             \algname{Byz-DASHA-PAGE} \tnote{\color{blue}(1)}
        \end{tabular}  & \cellcolor{bgcolor}$\frac{1}{\varepsilon^2}\parens{1+\parens{\omega + \frac{\sqrt{m}}{b}}\parens{\sqrt{\frac{1}{G}} + \sqrt{c\delta}}}$ &\cellcolor{bgcolor} $\frac{c\delta\zeta^2}{1 - c\delta B}$&\cellcolor{bgcolor} $\frac{1}{(c+\sqrt{c})B}$ \\
        \hline\hline
        \cellcolor{bgcolor} \begin{tabular}{c}
             \algname{Byz-EF21} \tnote{\color{blue}(2)}
        \end{tabular}  & \cellcolor{bgcolor}$\frac{(1+\sqrt{c\delta})}{\alpha_D\varepsilon^2}$ &\cellcolor{bgcolor} $\frac{(c\delta+\sqrt{c\delta})\zeta^2}{1 - B(c\delta+\sqrt{c\delta})}$&\cellcolor{bgcolor} $\frac{1}{c(B+B^2)}$ \\
        \cellcolor{bgcolor} \begin{tabular}{c}
             \algname{Byz-EF21-BC} \tnote{\color{blue}(2)}
        \end{tabular}  & \cellcolor{bgcolor}$\frac{(1+\sqrt{c\delta})}{\alpha_D\alpha_P\varepsilon^2}$ &\cellcolor{bgcolor} $\frac{(c\delta+\sqrt{c\delta})\zeta^2}{1 - B(c\delta+\sqrt{c\delta})}$&\cellcolor{bgcolor} $\frac{1}{c(B+B^2)}$ \\
        \hline
    \end{tabular}
    \begin{tablenotes}
        \item[{\color{blue} (1)}] These methods use unbiased compression and work with stochastic gradients. For \algname{Byz-VR-MARINA} and \algname{Byz-VR-MARINA 2.0} $p = \min\{\nicefrac{1}{1+\omega}, \nicefrac{b}{m}\}$, for \algname{Byz-DASHA-PAGE} $p = \nicefrac{b}{m}$.
        \item[{\color{blue} (2)}] These methods compute full gradients on regular workers and use (biased) contractive compression. To enable easier comparison with the methods employing unbiased compression, one can assume that the biased compressors arise from unbiased ones through scaling and the following relations hold: $\nicefrac{1}{\alpha_P} = \omega_P + 1$ and $\nicefrac{1}{\alpha_D} = \omega_D + 1$.
    \end{tablenotes}     
    \end{threeparttable}
    % \vspace{-5mm}
\end{table*}

Below we summarize our main contributions.

\textbf{$\diamond$ Improved complexity bounds.} We propose two new Byzantine-robust methods incorporating \emph{unbiased} compression: \algname{Byz-VR-MARINA 2.0} and \algname{Byz-DASHA-PAGE}, and prove their complexity bounds for general smooth non-convex functions under quite general assumptions about sampling and stochasticity of the gradients.
The derived complexity bounds for \algname{Byz-VR-MARINA 2.0} and \algname{Byz-DASHA-PAGE} improve the existing theoretical results of the current state-of-the-art \algname{Byz-VR-MARINA} method, outperforming it by factors of $\sqrt{\max\{1+\omega, \nicefrac{m}{b}\}}$ and $\sqrt{\max\{(1+\omega)^3, \nicefrac{m^2(1+\omega)}{b^2}\}}$ in the leading term, respectively, as shown in Table~\ref{tab:comparison_of_rates_minimization}. These are significant improvements since $\omega$ (compression parameter) and $m$ (size of the local dataset) are usually large. Moreover, we prove that both algorithms converge linearly under the Polyak-{\L}ojasiewicz condition (see Appendix~\ref{appendix:PL}).

\textbf{$\diamond$ Smaller size of the neighborhood.} Under the $(B,\zeta^2)$-heterogeneity assumption, \algname{Byz-VR-MARINA 2.0} and \algname{Byz-DASHA-PAGE} converge to a smaller neighborhood of the solution than their competitors. Furthermore, when $B=0$, our methods can achieve $\mathbb{E}[\|\nabla f(x)\|^2] = \cO(c\delta)$, matching the lower bound by \citet{karimireddy2020byzantine}, while for \algname{Byz-VR-MARINA} one can only prove $\mathbb{E}[\|\nabla f(x)\|^2] = \cO(\nicefrac{c\delta}{p})$, which is worse by a large factor of $\nicefrac{1}{p} \sim \max\{\omega, \nicefrac{m}{b}\}$.

\textbf{$\diamond$ Higher tolerance to Byzantine workers.} When the $(B,\zeta^2)$-heterogeneity assumption holds with $B > 0$, the results derived for \algname{Byz-VR-MARINA 2.0} and \algname{Byz-DASHA-PAGE} guarantee convergence in the presence of $\nicefrac{1}{p}$ times more Byzantine workers than in the case of \algname{Byz-VR-MARINA}.

\textbf{$\diamond$ The first Byzantine-robust methods with error feedback.} Finally, we propose and analyze two new Byzantine-robust methods employing any (in general, \emph{biased}) contractive compressors -- \algname{Byz-EF21} and \algname{Byz-EF21-BC}. Both are based on modern error feedback -- the \algname{EF21} algorithm of \citet{richtarik2021ef21}, and are the first provably Byzantine-robust methods utilising error feedback. The \algname{Byz-EF21-BC} algorithm, in addition to workers-to-server compression, also compresses messages sent from the server to workers, hence being the first provably Byzantine-robust algorithm using bidirectional compression.

\subsection{Related Work on Byzantine Robustness}

The first approaches to designing Byzantine-robust distributed optimization methods\footnote{We defer the discussion of related work on communication compression to Appendix~\ref{appendix:extra_related_work}.} primarily concentrate on the aggregation aspect, employing conventional \algname{Parallel-SGD} as the algorithm's foundation \citep{blanchard2017machine, chen2017distributed, yin2018byzantine, damaskinos2019aggregathor, guerraoui2018hidden, pillutla2022robust}. Nonetheless, it has become evident that these classical approaches are susceptible to specialized attacks \citep{baruch2019little, xie2020fall}. To address this issue, \citet{karimireddy2021learning} develop a formal definition of robust aggregation and propose a provably Byzantine-robust algorithm based on the aggregation of client momentums. \citet{karimireddy2020byzantine} further extend these results to the heterogeneous setup. An alternative formalism of robust aggregation, along with new examples of robust aggregation rules and Byzantine-robust methods, are proposed and analyzed by \citet{allouah2023fixing}.
The application of variance reduction to achieve Byzantine robustness is first explored by \citet{wu2020federated}, who use \algname{SAGA}-type variance reduction \citep{defazio2014saga} on regular workers. Subsequently, \citet{zhu2021broadcast} enhance this approach by incorporating unbiased communication compression. \citet{gorbunov2023variance} improve the results from \citet{wu2020federated, zhu2021broadcast} and derive the current theoretical state-of-the-art convergence results. Numerous other approaches have also been proposed, including the banning of Byzantine workers \citep{alistarh2018byzantine, allen2020byzantine}, random checks of computations \citep{gorbunov2021secure}, computation redundancy \citep{chen2018draco, rajput2019detox}, and reputation scores \citep{rodriguez2020dynamic, regatti2020bygars, xu2020towards}. We refer to \cite{lyu2020privacy} for a comprehensive survey.

\section{METHODS WITH UNBIASED COMPRESSION}

In this section, we introduce our main results on methods employing unbiased compression.

\subsection{Warm-up: \algname{Byz-VR-MARINA 2.0}}

\begin{algorithm}[t]
    \caption{\algname{Byz-VR-MARINA 2.0}}\label{alg:marina2}
    \begin{algorithmic}[1]
    \State \textbf{Input:} starting point $x_0 \in \mathbb{R}^d$, stepsize $\gamma > 0$, probability $p \in (0, 1]$ , number of iterations $T \geq 1$, a collection of unbiased compressors $\{\cQ_i\}_{i\in\cG}$
%    \State Initialize $g^0 = \nabla f(x^0)$ on the server, {\color{green}$g_i^0 = \nabla f_i(x^0)$ for the workers.}
    \Statex
    \For{$t=0, 1, \dots, T-1$}
    \State Let $c^{t+1} = \begin{cases}
          1,& \textnormal{with probability $p$} \\
          0, & \textnormal{with probability $1 - p$} 
        \end{cases}$
    \State Broadcast $g^{t}$ to all nodes
    \For{$i\in \cG$ in parallel}
    \State $x^{t+1} = x^t - \gamma g^t$
    \If{$c^{t+1}=1$}
    \State $g_i^{t+1} =\nabla f_i(x^{t+1})$
    \State Send $\nabla f_i(x^{t+1})$ to the server.
    \Else
    \State $m_i^{t+1} = \cQ_i(\hdelta_{i}(x^{t+1}, x^t))$
    \State\label{line:marina_change} $g_i^{t+1} = g_i^t + m_i^{t+1}$
    \State Send $m_i^{t+1}$ to the server
    \EndIf
    \EndFor
    \State $g^{t+1} = \text{ARAgg}\parens{g_1^{t+1};\ldots;g_n^{t+1}}$
    \EndFor
%    \State \textbf{Return:} $\hat{x}^T$ chosen uniformly at random from $\{x_t\}_{t=0}^{T-1}$
    \end{algorithmic}
\end{algorithm}
    
\begin{algorithm}[t]
    \caption{\algname{Byz-DASHA-PAGE}}
    \label{alg:main_algorithm}
    \begin{algorithmic}[1]
    \State \textbf{Input:} starting point $x^0 \in \R^d$, stepsize $\gamma > 0$, momentum $a \in (0, 1]$, probability $p \in (0, 1]$ , number of iterations $T \geq 1$, a collection of unbiased compressors $\{\cQ_i\}_{i\in\cG}$
    
%    \State Initialize $g^0_i\in \R^d$, $h^0_i\in \R^d$ on the nodes and  $g^0 = \frac{1}{n}\sum_{i=1}^n g^0_i$ on the server
    \For{$t = 0, 1, \dots, T - 1$} 
    \State Let $c^{t+1} = \begin{cases}
      1,& \textnormal{with probability $p$} \\
      0, & \textnormal{with probability $1 - p$} 
    \end{cases}$
    \State Broadcast $g^t$ to all nodes
    \For{$i \in \cG$ in parallel}
    \State $x^{t+1} = x^t - \gamma g^t$
    \If{$c^{t+1}=1$}
    \State $h^{t+1}_i = \nabla f_{i}(x^{t+1})$ 
    \Else
    \State $h^{t+1}_i = h^{t}_i + \hdelta_i(x^{t+1}, x^{t})$
    \EndIf
    \State $m^{t+1}_i = \cQ_i\!(h^{t+1}_i - h^{t}_i - a \parens{g^t_i - h^{t}_i}\!)$ 
    \State $g^{t+1}_i = g_i^t+ m^{t+1}_i$
    \State Send $m_i^{t+1}$ to the server
    \EndFor
    \State $g^{t+1} = \text{ARAgg}(g_1^{t+1};\ldots;g_n^{t+1})$
    \EndFor
%    \State \textbf{Return:} $\hat{x}^T$ chosen uniformly at random from $\{x^t\}_{t=0}^{T-1}$
    \end{algorithmic}
\end{algorithm}

We begin by presenting \algname{Byz-VR-MARINA 2.0} (Algorithm~\ref{alg:marina2}) -- a modification of \algname{Byz-VR-MARINA} by \cite{gorbunov2023variance} that uses local vectors $g_i^t$ instead of $g^t$ in the update of $g_i^{t+1}$. In other words, line~\ref{line:marina_change} of the original \algname{Byz-VR-MARINA} method is $g_i^{t+1} = g^t + m_i^{t+1}$. The key idea behind both versions remains the same: to adapt \algname{VR-MARINA} \citep{gorbunov2021marina} by replacing the standard averaging of gradient estimators with a $(\delta,c)$-agnostic robust aggregator. 
It is worth mentioning that even without this modification and with no Byzantine workers, the gradient estimator in \algname{VR-MARINA} is \emph{conditionally biased}, i.e., $\E[g_i^{t+1}\mid x^{t+1}, x^t] \neq \nabla f_i(x^{t+1})$. With this insight in mind, and noting that in the homogeneous scenario, the variance of vectors received from regular workers converges to zero, robust aggregation naturally integrates with the algorithm, enabling provable tolerance to Byzantine workers.

While the algorithmic difference between \algname{Byz-VR-MARINA} and \algname{Byz-VR-MARINA 2.0} is almost negligible, it provides significant theoretical improvements, as demonstrated in Table~\ref{tab:comparison_of_rates_minimization}. This advancement has an intuitive explanation: $g^t$ contains the information received from Byzantine workers in the previous steps, so the vectors $\{g_i^{t+1}\}_{i\in\cG}$ in \algname{Byz-VR-MARINA} are more influenced by malicious messages than the vectors $\{g_i^{t+1}\}_{i\in \cG}$ in \algname{Byz-VR-MARINA 2.0}. Moreover, it can also be shown that in \algname{Byz-VR-MARINA 2.0} the local gradient estimators are unbiased, so that $\E[g_i^{t+1}] = \E[\nabla f_i(x^{t+1})]$ for all $i\in\cG$. This is not the case in \algname{Byz-VR-MARINA}. Despite the apparent similarity of our algorithm to Byz-VR-MARINA, the underlying intuition significantly differ. We refer to \Cref{appendix:proof_sketch} for further details.

The following theorem formalizes the main convergence result for \algname{Byz-VR-MARINA 2.0}.

% \begin{theorem}\label{thm:main_result_Byz-VR-MARINA_2_main_text}
% 	Let Assumptions~\ref{as:smoothness}, \ref{as:hessian_variance}, \ref{as:hessian_variance_local}, and \ref{as:bounded_heterogeneity} hold. Assume that $0 < \gamma \leq \frac{1}{L + \sqrt{R}}$ and $\delta < \frac{1}{(8c + 4\sqrt{c})B}$, where
% %	\begin{align*}
% %		0 < \gamma &\leq \frac{1}{L + \sqrt{R}}\\
% %        0<\delta &< \frac{1}{32cB} \min\left\{1, \frac{n}{B} \right\}\\
% %        g_i^0 &= \nabla f_i(x^0) \text{ } \forall i \in \cG
% %	\end{align*}
% 	$R = \frac{1-p}{p} \parens{\omega\parens{\frac{\cL_{\pm}^2}{b}+ L_{\pm}^2 + L^2} +\frac{\cL_{\pm}^2}{b}}\parens{\sqrt{\frac{1}{G}} + \sqrt{8c\delta}}^2$ and $g_i^0 = \nabla f_i(x^0)$ for all $i\in \cG$. Then for all $T \geq 0$ the iterates produced by \algname{Byz-VR-MARINA 2.0} satisfy
%  \begin{equation*}
%     \EE{\norm{\nabla f(\hx^T)}^2}\leq \frac{1}{1 - \parens{8c\delta+\sqrt{\frac{8c\delta}{G}}}B}\parens{\frac{2(f(x^0) - f^*)}{\gamma T} + \parens{8c\delta + \sqrt{\frac{8c\delta}{G}}}\zeta^2},
% \end{equation*}
% 	where $\hx^T$ is chosen uniformly at random from $x^0, x^1, \ldots, x^{T-1}$.
% \end{theorem}

\begin{restatable}{theorem}{BYZVRM}\label{thm:main_result_Byz-VR-MARINA_2}
% \begin{theorem}\label{thm:main_result_Byz-VR-MARINA_2}
	Let Assumptions~\ref{as:smoothness}, \ref{as:hessian_variance}, \ref{as:hessian_variance_local} and \ref{as:bounded_heterogeneity} hold. Assume that $0 < \gamma \leq (L + \sqrt{\eta})^{-1}$, $\delta < \parens{(8c + 4\sqrt{c})B}^{-1}$ and initialize $g_i^0 = \nabla f_i(x^0)$ for all $i\in \cG$, where
	$\eta = \frac{1-p}{p} \parens{\omega\parens{\frac{\cL_{\pm}^2}{b}+ L_{\pm}^2 + L^2} +\frac{\cL_{\pm}^2}{b}}\parens{\sqrt{\frac{1}{G}} + \sqrt{8c\delta}}^2$. Then for all $T \geq 0$ the iterates produced by \algname{Byz-VR-MARINA 2.0} satisfy
    % \begin{align*}
    %     \EE{\norm{\nabla f(\hx^T)}^2}\leq \frac{1}{1 - \parens{8c\delta+\sqrt{\frac{8c\delta}{G}}}B}\parens{\frac{2(f(x^0) - f^*)}{\gamma T} + \parens{8c\delta + \sqrt{\frac{8c\delta}{G}}}\zeta^2},
    % \end{align*}
    \begin{align*}
        \EE{\norm{\nabla f(\hx^T)}^2}\leq \frac{1}{A}\parens{\frac{2\delta^0}{\gamma T} + \parens{8c\delta + \sqrt{\frac{8c\delta}{G}}}\zeta^2},
    \end{align*}
	where $\delta^0=f(x^0) - f^*$, $A = 1 - \parens{8c\delta+\sqrt{\nicefrac{8c\delta}{G}}}B$ and $\hx^T$ is chosen uniformly at random from $x^0, x^1, \ldots, x^{T-1}$.
% \end{theorem}
\end{restatable}

The above upper bound consists of two terms: the first one decreases to zero at a rate $\cO(\nicefrac{1}{T})$, which is optimal for approximating first-order stationary points in the setup we consider \citep{fang2018spider, arjevani2023lower}, and the second one, corresponding to the size of the neighbourhood of the solution to which the method converges, is constant. This neighbourhood disappears when either $\zeta=0$ or $\delta=0$ (no malicious workers).
Furthermore, when $B = 0$, the second term is $\cO(\delta\zeta^2)$ (due to the fact that $\nicefrac{1}{G} = \cO(\delta)$ when $\delta > 0$), matching the lower bound form \citet{karimireddy2020byzantine} up to a numerical factor. Meanwhile, the corresponding term in the results derived for \algname{Byz-VR-MARINA} is $\cO(\nicefrac{\delta\zeta^2}{p})$, which is usually much larger than $\cO(\delta\zeta^2)$ since $p$ is typically small. When $B > 0$, the results for \algname{Byz-VR-MARINA 2.0} are valid when $\delta = \cO(\nicefrac{1}{B})$, while the existing guarantees for \algname{Byz-VR-MARINA} require $\delta = \cO(\nicefrac{p}{B})$, which is significantly smaller. Finally, as shown in Table~\ref{tab:comparison_of_rates_minimization}, the rate of convergence of \algname{Byz-VR-MARINA 2.0} is better than that of \algname{Byz-VR-MARINA} by a potentially large factor of $\sqrt{\max\{\omega, \nicefrac{m}{b}\}}$.

\subsection{\algname{Byz-DASHA-PAGE}}

Although \algname{Byz-VR-MARINA 2.0} enjoys notable theoretical improvements, it inherits some limitations of \algname{VR-MARINA}. The first one is of purely algorithmic nature: with probability $p$, regular workers send uncompressed vectors. These synchronization steps are necessary for the algorithm to converge, as they correct the error coming from compression and stochasticity in the gradients. However, their large communication cost can render the algorithm impractical.
To make the computation and communication cost of one round equal (on average, up to a constant factor) to the cost per round when workers compute stochastic gradients and send compressed vectors, the probability $p$ should be approximately $\min\{(1+\omega)^{-1}, \nicefrac{b}{m}\}$. If this is the case, a factor of $\sqrt{1+\omega}\sqrt{\max\{1+\omega, \nicefrac{m}{b}\}}$ appears in the complexity bounds of \algname{(Byz-)VR-MARINA (2.0)}, and the effects of stochasticity and communication compression are coupled.

To address the issues arising in \algname{VR-MARINA}, \citet{tyurin2022dasha} propose \algname{DASHA-PAGE}, which uses momentum variance reduction mechanisms \citep{cutkosky2019momentum, tran2022hybrid, liu2020optimal} to handle the noise resulting from communication compression (and a \emph{separate} \algname{Geom-SARAH}/\algname{PAGE}-type variance reduction technique to manage the stochasticity in the gradients). The derived complexity of \algname{DASHA-PAGE} matches that of \algname{VR-MARINA}, but with a better dependence on $\omega, m, b$: the factor $\sqrt{1+\omega}\sqrt{\max\{1+\omega, \nicefrac{m}{b}\}}$ appearing in the complexity bounds of \algname{VR-MARINA} is replaced with $\omega + \frac{\sqrt{m}}{b}$, which is always not greater than $\sqrt{1+\omega}\sqrt{\max\{1+\omega, \nicefrac{m}{b}\}}$ and strictly smaller than $\sqrt{1+\omega}\sqrt{\max\{1+\omega, \nicefrac{m}{b}\}}$ when $0<\omega < \nicefrac{m}{b}$.

Motivated by these developments, we introduce a Byzantine-robust variant of \algname{DASHA-PAGE} called \algname{Byz-DASHA-PAGE} (Algorithm~\ref{alg:main_algorithm}), aiming to enhance the convergence rates achieved by \algname{Byz-VR-MARINA 2.0}. We find that $(\delta,c)$-robust aggregation integrates seamlessly with \algname{DASHA-PAGE}, leading to the following result.

% \begin{theorem}\label{thm:main_result_dp_main_text}
% 	Let Assumptions~\ref{as:smoothness}, \ref{as:hessian_variance}, \ref{as:hessian_variance_local}, and \ref{as:bounded_heterogeneity} hold. Assume that $0 < \gamma \leq \frac{1}{L + \sqrt{R}}$ and $\delta < \frac{1}{(8c + 4\sqrt{c})B}$, where
% %	\begin{align*}
% %		0 < \gamma &\leq \frac{1}{L + \sqrt{R}}\label{eq:gamma_condition_app}\\
% %        0<\delta &< \frac{1}{32cB} \min\left\{1, \frac{n}{B} \right\}\\
% %        a &= \frac{1}{2\omega +1} \\
% %        g_i^0 = h_i^0 &= \nabla f_i(x^0) \text{ } \forall i \in \cG
% %	\end{align*}
% 	$R = \parens{8\omega(2\omega +1) \parens{L_{\pm}^2 + L^2} + \frac{1-p}{b} \parens{12\omega\parens{2\omega+1} + \frac{2}{p}}\cL_{\pm}^{2}}\parens{\sqrt{\frac{1}{G}} + \sqrt{8c\delta}}^2$ and $g_i^0 = \nabla f_i(x^0)$ for all $i\in \cG$. Then for all $T \geq 0$ the iterates produced by \algname{Byz-DASHA-PAGE} satisfy
%  \begin{equation*}
%     \EE{\norm{\nabla f(\hx^T)}^2}\leq \frac{1}{\parens{1 - \parens{8c\delta+\sqrt{\frac{8c\delta}{G}}}B}}\parens{\frac{2(f(x^0) - f^*)}{\gamma T} + \parens{8c\delta+\sqrt{\frac{8c\delta}{G}}}\zeta^2},
% \end{equation*}
% 	where $\hx^T$ is chosen uniformly at random from $x^0, x^1, \ldots, x^{T-1}$.
% \end{theorem}

\begin{restatable}{theorem}{BYZDP}\label{thm:main_result_dp}
% \begin{theorem}\label{thm:main_result_dp}
    Let Assumptions~\ref{as:smoothness}, \ref{as:hessian_variance}, \ref{as:hessian_variance_local} and \ref{as:bounded_heterogeneity} hold. Assume that $0 < \gamma \leq (L + \sqrt{\eta})^{-1}$, $\delta < \parens{(8c + 4\sqrt{c})B}^{-1}$ and initialize $g_i^0 = \nabla f_i(x^0)$ for all $i\in \cG$, where
	$\eta = \parens{8\omega(2\omega +1) \parens{L_{\pm}^2 + L^2} + \frac{1-p}{b} \parens{12\omega\parens{2\omega+1} + \frac{2}{p}}\cL_{\pm}^{2}} \allowbreak \times \parens{\sqrt{\frac{1}{G}} + \sqrt{8c\delta}}^2$. Then for all $T \geq 0$ the iterates produced by \algname{Byz-DASHA-PAGE} satisfy
    % \begin{equation*}
    %     \EE{\norm{\nabla f(\hx^T)}^2}\leq \frac{1}{1 - \parens{8c\delta+\sqrt{\frac{8c\delta}{G}}}B}\parens{\frac{2(f(x^0) - f^*)}{\gamma T} + \parens{8c\delta+\sqrt{\frac{8c\delta}{G}}}\zeta^2},
    % \end{equation*}
    \begin{equation*}
        \EE{\norm{\nabla f(\hx^T)}^2}\leq \frac{1}{A}\parens{\frac{2\delta^0}{\gamma T} + \parens{8c\delta+\sqrt{\frac{8c\delta}{G}}}\zeta^2},
    \end{equation*}
    % \begin{align*}
    %     \EE{\norm{\nabla f(\hx^T)}^2} \leq \frac{2(f(x^0) - f^*)}{A \gamma T}
    %     + \frac{\parens{8c\delta+\sqrt{\frac{8c\delta}{G}}}\zeta^2}{A},
    % \end{align*}
	where $\delta^0=f(x^0) - f^*$, $A = 1 - \parens{8c\delta+\sqrt{\nicefrac{8c\delta}{G}}}B$ and $\hx^T$ is chosen uniformly at random from $x^0, x^1, \ldots, x^{T-1}$.
% \end{theorem}
\end{restatable}

In terms of the size of the neighborhood and maximum number of Byzantine workers, the above result aligns closely with what we obtain for \algname{Byz-VR-MARINA 2.0}, thereby being superior to the existing guarantees for \algname{Byz-VR-MARINA}. However, the upper bound on the stepsize of \algname{Byz-DASHA-PAGE} has a better joint dependence on $\omega$ and $p$ compared to that in \algname{Byz-VR-MARINA 2.0}. The difference can be seen in the expression for $\eta$: in Theorem~\ref{thm:main_result_dp}, $\nicefrac{1}{p}$ and $\omega$ are decoupled, whereas $\eta$ from Theorem~\ref{thm:main_result_Byz-VR-MARINA_2} has a term proportional to $\nicefrac{1+\omega}{p}$, ultimately leading to the factor of $\sqrt{1+\omega}\sqrt{\max\{1+\omega, \nicefrac{m}{b}\}}$ in the complexity bounds of \algname{Byz-VR-MARINA 2.0}. A comparison of complexity results is given in Table~\ref{tab:comparison_of_rates_minimization}.

\section{METHODS WITH BIASED COMPRESSION AND ERROR FEEDBACK}

\begin{algorithm}[t]
%\footnotesize
    \caption{\algname{Byz-EF21}}
    \label{algorithm:byzef21}
    \begin{algorithmic}[1]
    \State \textbf{Input:} starting point $x^0 \in \R^d$, stepsize $\gamma > 0$, number of iterations $T \geq 1$, a collection of biased compressors $\{\cC_i\}_{i\in\cG}$
    \For{$t = 0, 1, \dots, T - 1$}      
        \State $x^{t+1} = x^t - \gamma g^t$ %\hfill {\scriptsize \color{gray} Take  gradient-type step}
        \Statex
        \Statex
        \State {Broadcast $x^{t+1}$ to all workers}
        \For{$i \in \cG$ {\bf in parallel}}
            \Statex
            \State $c_i^{t} = \cC_i(\nabla f_i(x^{t+1}) - g_i^t)$ %\hfill {\scriptsize \color{gray} Compress shifted gradient via $\cC_i^t \in \mathbb{B}(\alpha)$}
            \State $g_i^{t+1} = g_i^t + c_i^{t}$
            \State {Send message $c_i^{t}$ to the server}
        \EndFor
        \State $g^{t+1} = \texttt{ARAgg}(g_1^{t+1},\ldots,g_n^{t+1})$ %\hfill {\scriptsize \color{gray} Compute gradient estimator}
    \EndFor
    \end{algorithmic}
\end{algorithm}

In this section, we transition from a discussion of methods using unbiased compressors to algorithms utilizing (typically biased) contractive compressors, which typically have better empirical performance than the unbiased alternatives \citep{seide20141}. Such compressors are usually employed together with error feedback mechanisms, since their naive use in distributed Gradient Descent can lead to divergence of the algorithm \citep{beznosikov2020biased}. In this work, we focus on the modern \algname{EF21} error feedback mechanism proposed by \citet{richtarik2021ef21}, as it offers better convergence guarantees compared to standard error feedback. The method is based on the idea of each worker compressing the difference between the current gradient and its estimate $g_i^t$ and using this compressed message to update the local gradient estimate in the next round. Since both $\nabla f_i(x^t)$ and $g_i^t$ converge to $\nabla f_i(x^*)$ (where $x^*$ is a stationary point to which the method converges), $\nabla f_i(x^t) - g_i^t$ tends to zero. Given that the compressor is contractive, it must be the case that the inaccuracy due to the compression of this difference also converges to zero. Importantly, from an algorithmic point of view, \algname{EF21} resembles \algname{MARINA}, which is known to work well with robust aggregation.

Motivated by the above considerations, we propose two new Byzantine-robust methods -- \algname{Byz-EF21} and \algname{Byz-EF21-BC} (Algorithms~\ref{algorithm:byzef21} and \ref{algorithm:byzef21bc}). \algname{Byz-EF21} is a modification of the \algname{EF21} mechanism employing $(\delta,c)$-robust aggregation to ensure Byzantine-robustness. \algname{Byz-EF21-BC} further enhances the method by adding bidirectional compression: following \citet{gruntkowska2023ef21}, we additionally apply the \algname{EF21} mechanism on the server's side to compress messages broadcast to workers. Note that when $\cC^{P}$ is the identity operator, the latter algorithm reduces to \algname{Byz-EF21}.
The intuition behind both algorithms mirrors that of \algname{Byz-VR-MARINA (2.0)}: the variance of the estimates $\{g_i^t\}_{i\in\cG}$ approaches $\cO(\zeta^2)$, leaving Byzantine workers progressively less room to ``hide in the noise''.

\begin{algorithm}[t]
%\footnotesize
    \caption{\algname{Byz-EF21-BC}}
    \label{algorithm:byzef21bc}
    \begin{algorithmic}[1]
    \State \textbf{Input:} starting point $x^0 \in \R^d$, stepsize $\gamma > 0$, number of iterations $T \geq 1$, a collection of biased compressors $\{\cC_i^D\}_{i\in\cG}$, $\cC^P$ 
    \For{$t = 0, 1, \dots, T - 1$}
        \State $x^{t+1} = x^t - \gamma g^t$ %\hfill {\scriptsize \color{gray} Take  gradient-type step}
        \State $s^{t+1} = \cC^{P}\left(x^{t+1} - w^t\right)$ % \hfill {\scriptsize \color{gray} Compress shifted model on the server via $\cC^{P} \in \mathbb{B}\left(\alpha_P\right)$}
        \State $w^{t+1} = w^t + s^{t+1}$ %\hfill {\scriptsize \color{gray} Update model shift}
        \State {Broadcast $s^{t+1}$ to all workers}
        \For{$i \in \cG$ {\bf in parallel}}
            \State $w^{t+1} = w^{t} + s^{t+1}$ %\hfill  {\scriptsize \color{gray} Update model shift}
            \State $c_i^{t} = \cC_i^D(\nabla f_i(w^{t+1}) - g_i^t)$ % \hfill {\scriptsize \color{gray} Compress shifted gradient via $\cC_i^{D} \in \mathbb{B}(\alpha_D)$}
            \State $g_i^{t+1} = g_i^t + c_i^{t}$
            \State {Send message $c_i^{t}$ to the server}
        \EndFor
        \State $g^{t+1} = \text{ARAgg}(g_1^{t+1},\ldots,g_n^{t+1})$ %\hfill {\scriptsize \color{gray} Compute gradient estimator}
    \EndFor
    \end{algorithmic}
\end{algorithm}

The following theorem presents the main result for \algname{Byz-EF21} and \algname{Byz-EF21-BC} in a unified manner.

\begin{restatable}{theorem}{BYZEFTW}\label{thm:Byz_EF21_BC}
% \begin{theorem}\label{thm:Byz_EF21_BC}
    Let Assumptions~\ref{as:smoothness}, \ref{as:hessian_variance}, and \ref{as:bounded_heterogeneity} hold. Assume that $\cC_i^D \in \mathbb{B}(\alpha_D)$, $\cC^P \in \mathbb{B}(\alpha_P)$, $0 < \gamma \leq (L + \sqrt{\eta})^{-1}$ and $\delta < (8c(\sqrt{B}+B)^2)^{-1}$, where $\eta = \frac{32}{\alpha_D^2} \parens{1 + \frac{5}{\alpha_P^2}} \parens{1 + \sqrt{8c\delta}}^2 \parens{L_{\pm}^2 + L^2}$. Initialize $w^0 = x^0$, and $g_i^0 = \nabla f_i(x^0)$ for all $i\in\cG$. Then for all $T \geq 0$, the iterates produced by \algname{Byz-EF21}/\algname{Byz-EF21-BC} satisfy
    \begin{equation*}
        \EE{\norm{\nabla f(\hx^T)}^2}\leq \frac{1}{A} \left(\frac{2 \delta^0}{\gamma T}
            + \left(8 c\delta + \sqrt{8 c\delta} \right) \zeta^2\right),
    \end{equation*}
    where $\delta^0=f(x^0) - f^*$, $A = 1 - B\left(8c\delta + \sqrt{8c\delta}\right)$ and $\hx^T$ is chosen uniformly at random from $x^0, x^1, \ldots, x^{T-1}$.
% \end{theorem}
\end{restatable}

% \begin{theorem}\label{thm:Byz_EF21_BC}
% 	Let Assumptions~\ref{as:smoothness}, \ref{as:hessian_variance}, and \ref{as:bounded_heterogeneity} hold. Assume that $0 < \gamma \leq \frac{1}{L + \sqrt{R}}$ and $\delta < \frac{1}{8c(\sqrt{B}+B)^2}$, where $R = \frac{32}{\alpha_D^2} \parens{1 + \frac{5}{\alpha_P^2}} \parens{1 + \sqrt{8c\delta}}^2 \parens{L_{\pm}^2 + L^2}$, $1-\alpha_D$, $1-\alpha_P$ are contraction parameters of $\cC_i^D$ and $\cC^P$ respectively, $w^0 = x^0$, and $g_i^0 = \nabla f_i(x^0)$ for all $i\in\cG$. Then for all $T \geq 0$, the iterates produced by \algname{Byz-EF21}/\algname{Byz-EF21-BC} satisfy
%  \begin{equation*}
%     \EE{\norm{\nabla f(\hx^T)}^2}\leq \frac{1}{1 - B\left(8c\delta + \sqrt{8c\delta}\right)} \left(\frac{2 (f(x^0) - f^*)}{\gamma T}
%         + \left(8 c\delta + \sqrt{8 c\delta} \right) \zeta^2\right),
% \end{equation*}
% where $\hx^T$ is chosen uniformly at random from $x^0, x^1, \ldots, x^{T-1}$.
% \end{theorem}

Similar to the bounds we derived for \algname{Byz-VR-MARINA 2.0} and \algname{Byz-DASHA-PAGE}, the above upper bound for \algname{Byz-EF21}/\algname{Byz-EF21-BC} has two terms: the first decreases at a rate $\cO(\nicefrac{1}{T})$, and the second remains constant. The neighbourhood again vanishes when either $\zeta=0$ or $\delta=0$. In the scenario where $B = 0$ and $\sqrt{c\delta} = \cO(c\delta)$ (which occurs when there are many Byzantine workers), the second term is $\cO(\delta\zeta^2)$, matching the lower bound from \citet{karimireddy2020byzantine}. Furthermore, when $\delta = 0$ (no Byzantines), the rate aligns with the result for \algname{EF21-BC} \citep{fatkhullin2021ef21}, and when additionally $\alpha_P = 1$ (no downlink compression), the rate matches the one of \algname{EF21} \citep{richtarik2021ef21}. Finally, when $B > 0$, the above result holds whenever $\delta < \nicefrac{1}{8c(\sqrt{B}+B)^2}$. While this bound is worse than the one we derive for \algname{Byz-VR-MARINA 2.0} and \algname{Byz-DASHA-PAGE}, it is better than that of \algname{Byz-VR-MARINA} when $B < \nicefrac{1}{p}$ (which occurs, for example, when $\omega$ is large and $B$ is small).

\section{NUMERICAL EXPERIMENTS}

We conduct an empirical comparison\footnote{Our codes are available online: \url{https://github.com/Nikosimus/CC-for-BR-Learning}.} of \algname{Byz-VR-MARINA}, \algname{Byz-VR-MARINA 2.0} and \algname{Byz-DASHA-PAGE} in both homogeneous and heterogeneous settings. More details and additional experiments, including those on error feedback methods, are provided in Appendix~\ref{appendix:extra_exp_details}.

We solve a binary logistic regression problem with non-convex regularizer, using the \texttt{phishing} dataset from \texttt{LibSVM} \citep{chang2011libsvm}. The data is divided among $n = 16$ workers, out of which $3$ are Byzantine. As the aggregation rule, we use the Coordinate-wise Median (CM) aggregator \citep{yin2018byzantine} with bucketing \citep{karimireddy2020byzantine} (see Appendix~\ref{appendix:robust_aggr_and_compr}). We consider four different attacks performed by the Byzantine clients: \emph{Bit Flipping} (BF): change the sign of the update, \emph{Label Flipping} (LF): change the labels, i.e., $y_{i,j} \mapsto -y_{i,j}$, \emph{Inner Product Manipulation} (IPM): send $-\frac{z}{G} \sum_{i\in\cG} \nabla f_i(x)$, and \emph{A Little Is Enough} (ALIE): estimate the mean $\mu_{\cG}$ and standard deviation $\sigma_{\cG}$ of the regular updates and send $\mu_{\cG} - z \sigma_{\cG}$, where $z$ is a constant controlling the strength of the attacks.

\begin{figure}[t]
\centering
\begin{subfigure}{1.6in}%{.23\textwidth}
  \centering
  \includegraphics[width=\textwidth]{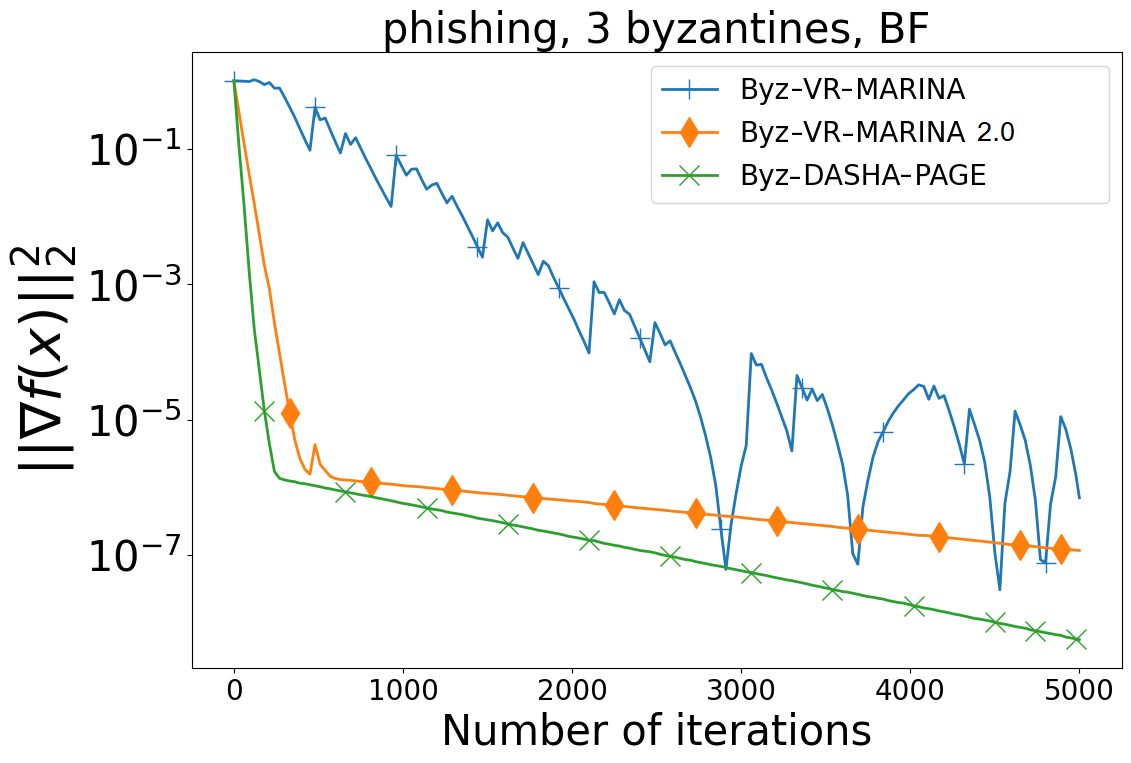}
  \caption{BF attack}
\end{subfigure}%
\begin{subfigure}{1.6in}%{.23\textwidth}
  \centering
  \includegraphics[width=\textwidth]{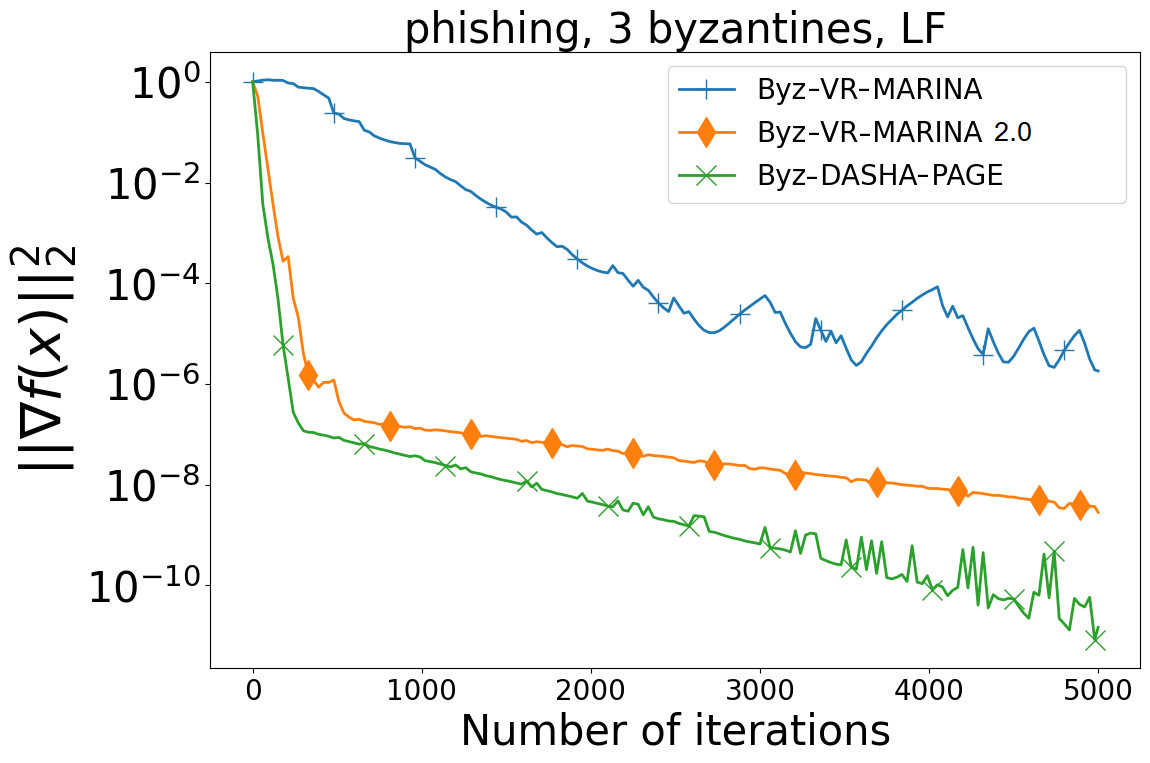}
  \caption{LF attack}
\end{subfigure}
\begin{subfigure}{1.6in}%{.23\textwidth}
  \centering
  \includegraphics[width=\textwidth]{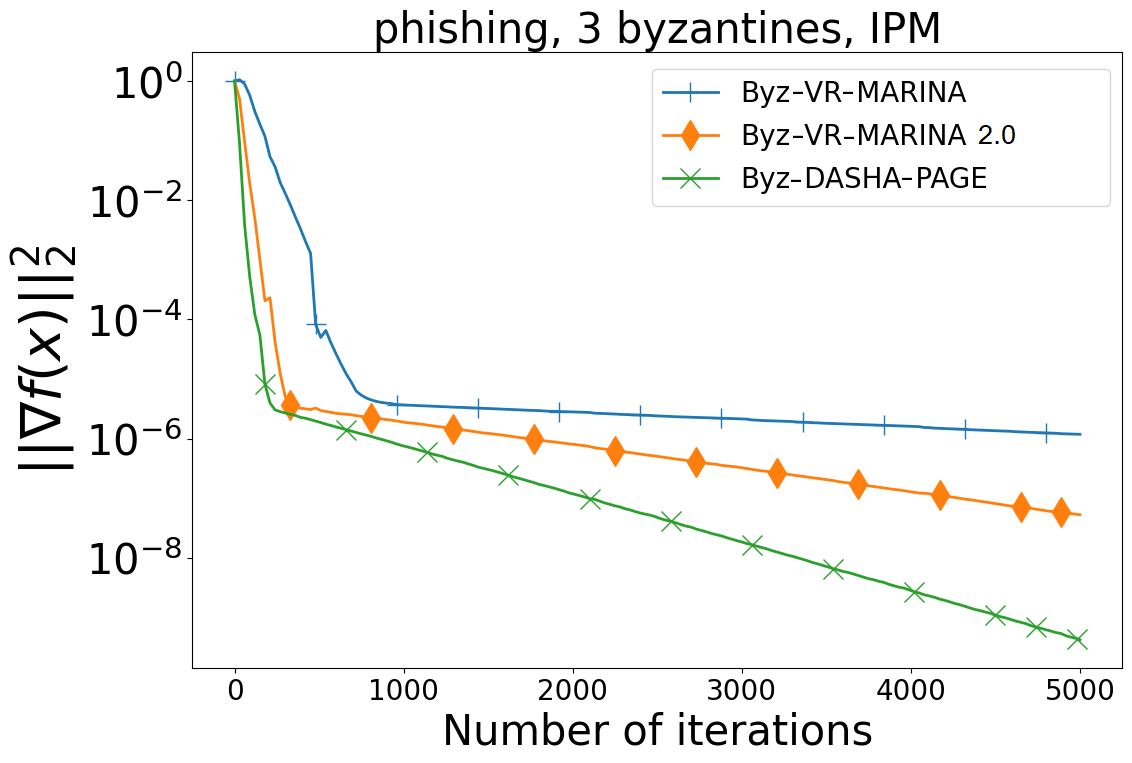}
  \caption{IPM attack}
\end{subfigure}
\begin{subfigure}{1.6in}%{.23\textwidth}
  \centering
  \includegraphics[width=\textwidth]{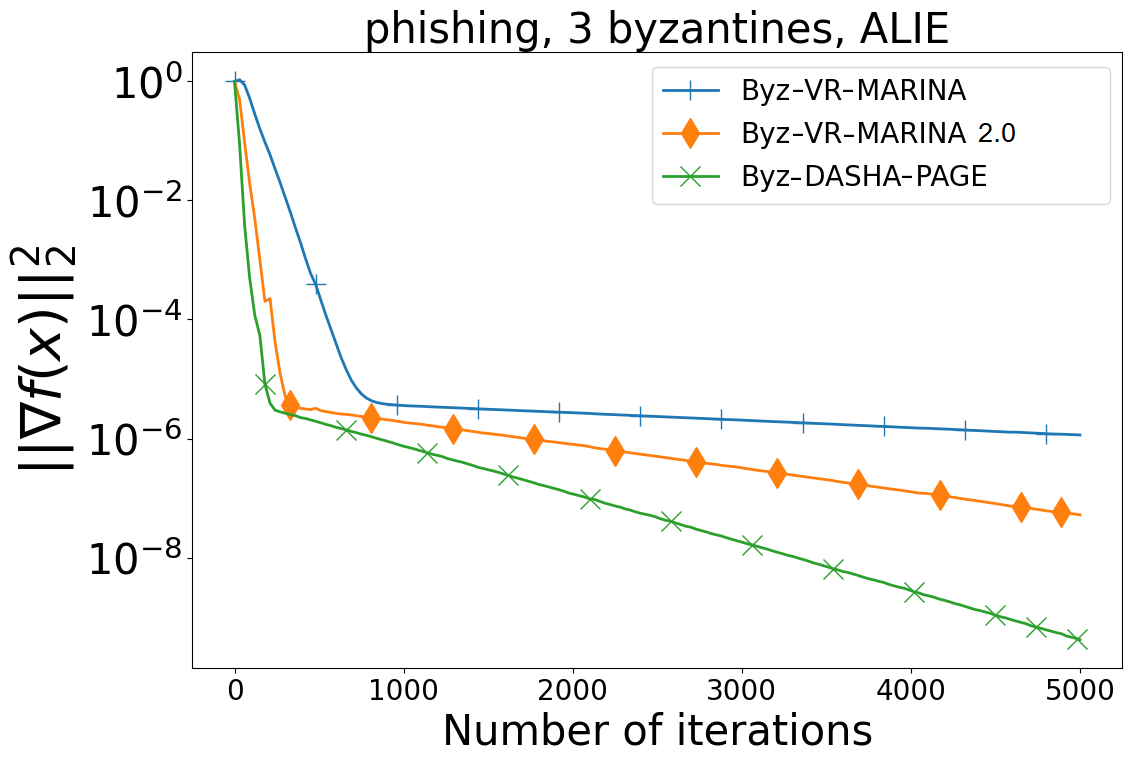}
  \caption{ALIE attack}
\end{subfigure}
\caption{Convergence in terms of the number of iterations in the homogeneous non-convex setting.}
\label{fig:NC_homogeneous}
\end{figure}

\begin{figure}[t]
\centering
\begin{subfigure}{1.6in}%{.23\textwidth}
  \centering
  \includegraphics[width=\textwidth]{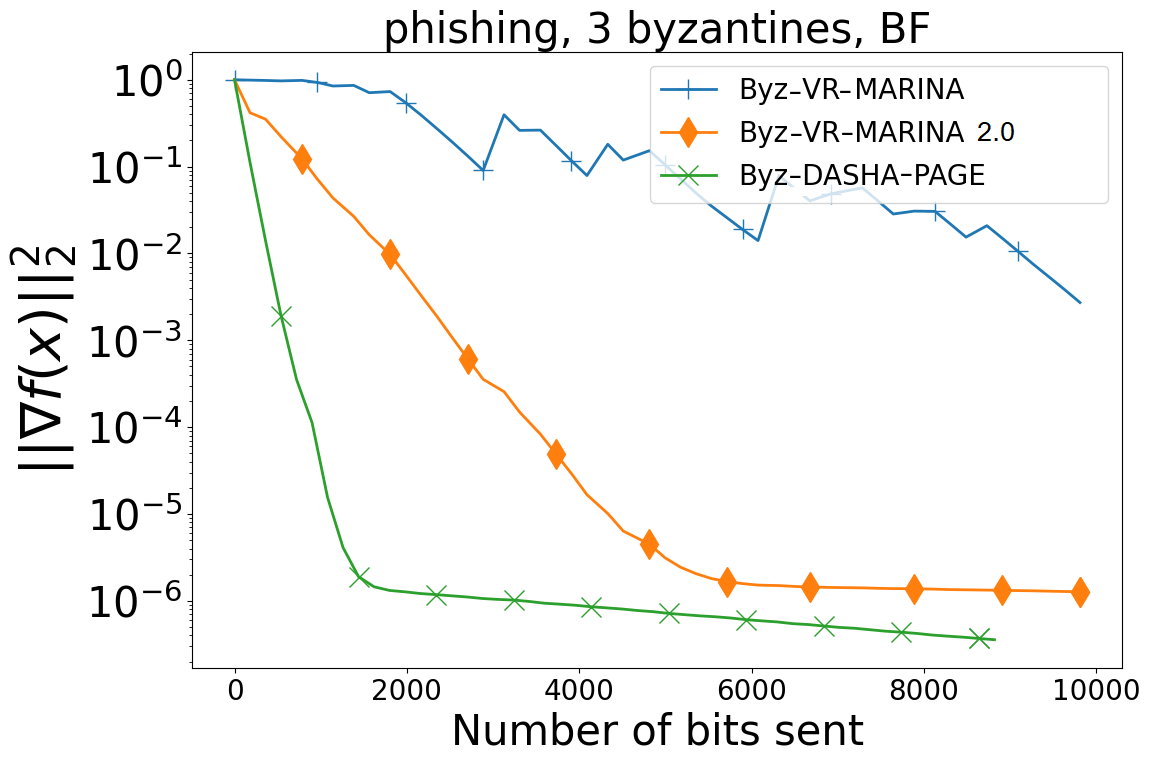}
  \caption{BF attack}
\end{subfigure}
\begin{subfigure}{1.6in}%{.23\textwidth}
  \centering
  \includegraphics[width=\textwidth]{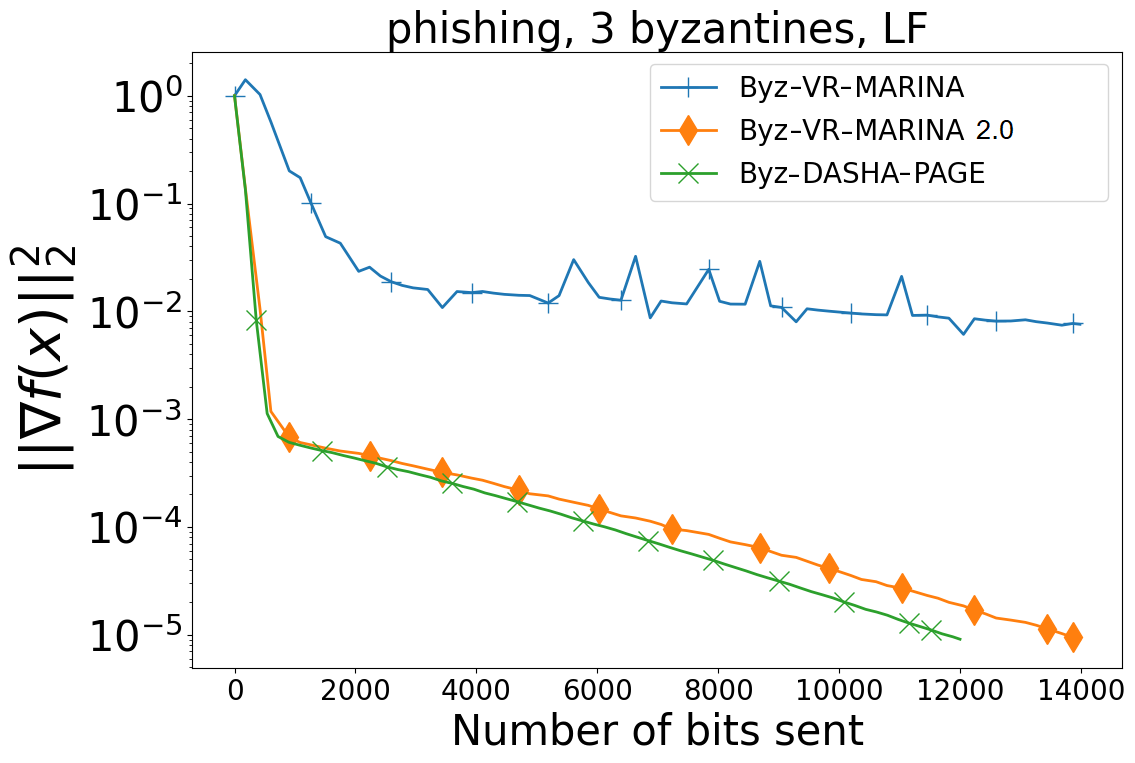}
  \caption{LF attack}
\end{subfigure}
\begin{subfigure}{1.6in}%{.23\textwidth}
  \centering
  \includegraphics[width=\textwidth]{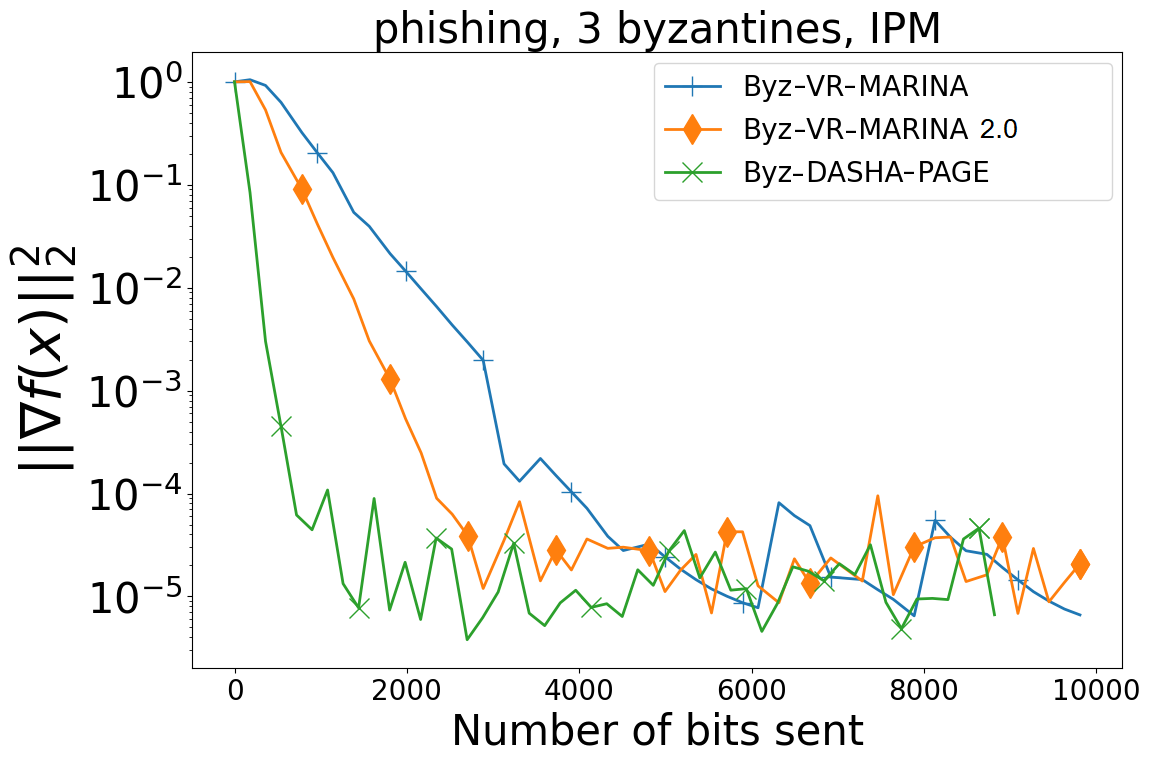}
  \caption{IPM attack}
\end{subfigure}
\begin{subfigure}{1.6in}%{.23\textwidth}
  \centering
  \includegraphics[width=\textwidth]{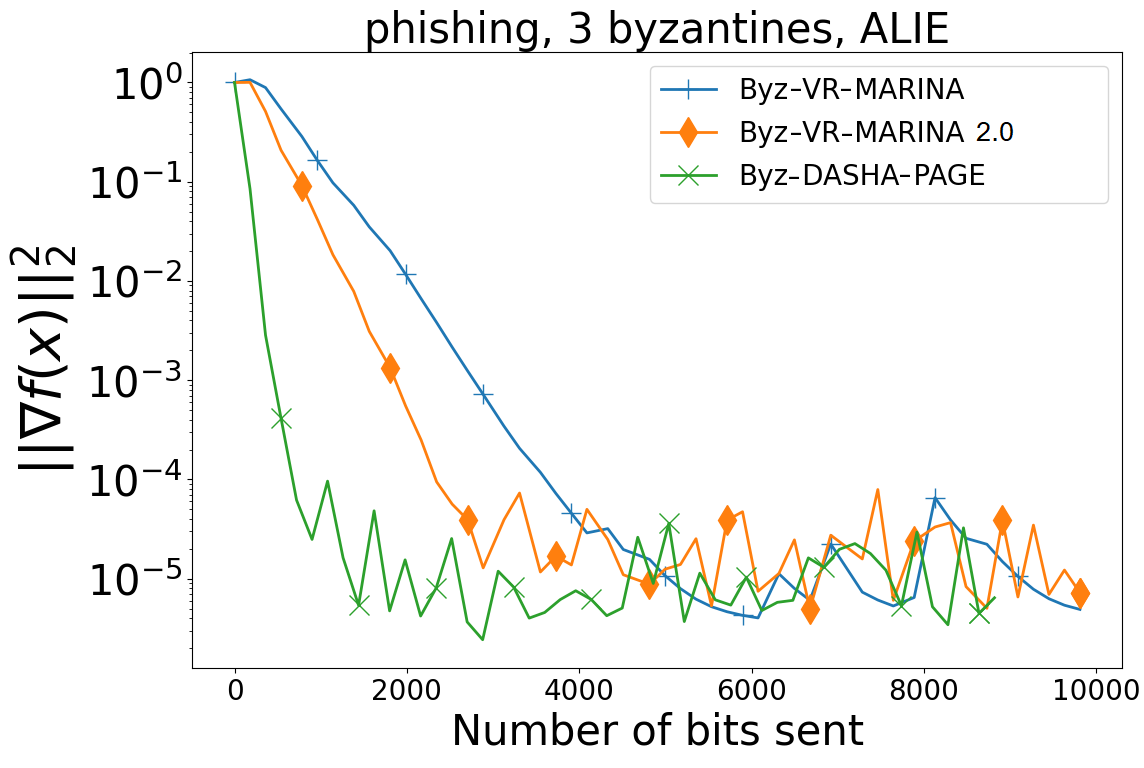}
  \caption{ALIE attack}
\end{subfigure}
\caption{Convergence in terms of the number of bits sent in the heterogeneous non-convex setting.}
\label{fig:NC_heterogeneous}
\end{figure}

\paragraph{Non-convex homogeneous setting.} In the first set of experiments (Figure~\ref{fig:NC_homogeneous}), we showcase the efficiency of our methods in the homogeneous setting ($B=\zeta=0$ in Assumption~\ref{as:bounded_heterogeneity}). %Theorems \ref{thm:main_result_Byz-VR-MARINA_2} and \ref{thm:main_result_dp} indicate convergence up to arbitrary accuracy.
Regardless of the attack type, \algname{Byz-VR-MARINA 2.0} and \algname{Byz-DASHA-PAGE} significantly outperform the \algname{Byz-VR-MARINA} baseline, both in terms of convergence speed and achieved accuracy, with \algname{Byz-DASHA-PAGE} taking the lead.

\paragraph{Non-convex heterogeneous setting.} As in the homogeneous scenario, both \algname{Byz-VR-MARINA 2.0} and \algname{Byz-DASHA-PAGE} converge faster than \algname{Byz-VR-MARINA} (Figure~\ref{fig:NC_heterogeneous}). The empirical superiority of our methods is especially apparent in the case of BF and LF attacks. Importantly, as suggested by theory, in the heterogeneous setting ($B,\zeta > 0$ in Assumption~\ref{as:bounded_heterogeneity}), the algorithms converge only to a certain neighborhood of the solution. The radius of this neighborhood is larger in the case of \algname{Byz-VR-MARINA} compared to our proposed alternatives. A noteworthy observation is the inherent stability of \algname{Byz-VR-MARINA 2.0} and \algname{Byz-DASHA-PAGE}, as they consistently exhibit less fluctuations compared to \algname{Byz-VR-MARINA}.

\subsubsection*{Acknowledgements}

The work of A. Rammal was performed during a research internship at KAUST led by P. Richt\'arik. A.
Rammal is a student at \'Ecole Polytechnique, France. The work of N. Fedin was supported by a grant for research centers in the field of artificial intelligence, provided by the Analytical Center for the Government of the Russian Federation in accordance with the subsidy agreement (agreement identifier 000000D730321P5Q0002) and the agreement with the Moscow Institute of Physics and Technology dated November 1, 2021 No. 70-2021-00138.

\bibliography{ref}

\begin{thebibliography}{}

\bibitem[Alistarh et~al., 2018a]{alistarh2018byzantine}
Alistarh, D., Allen-Zhu, Z., and Li, J. (2018a).
\newblock Byzantine stochastic gradient descent.
\newblock In {\em Proceedings of the 32nd International Conference on Neural Information Processing Systems}, pages 4618--4628.

\bibitem[Alistarh et~al., 2017]{alistarh2017qsgd}
Alistarh, D., Grubic, D., Li, J., Tomioka, R., and Vojnovic, M. (2017).
\newblock {Q}{S}{G}{D}: Communication-efficient sgd via gradient quantization and encoding.
\newblock {\em Advances in Neural Information Processing Systems}, 30.

\bibitem[Alistarh et~al., 2018b]{alistarh2018convergence}
Alistarh, D., Hoefler, T., Johansson, M., Konstantinov, N., Khirirat, S., and Renggli, C. (2018b).
\newblock The convergence of sparsified gradient methods.
\newblock {\em Advances in Neural Information Processing Systems}, 31.

\bibitem[Allen-Zhu et~al., 2021]{allen2020byzantine}
Allen-Zhu, Z., Ebrahimian, F., Li, J., and Alistarh, D. (2021).
\newblock Byzantine-resilient non-convex stochastic gradient descent.
\newblock In {\em International Conference on Learning Representations}.

\bibitem[Allouah et~al., 2023]{allouah2023fixing}
Allouah, Y., Farhadkhani, S., Guerraoui, R., Gupta, N., Pinot, R., and Stephan, J. (2023).
\newblock Fixing by mixing: A recipe for optimal byzantine ml under heterogeneity.
\newblock In {\em International Conference on Artificial Intelligence and Statistics}, pages 1232--1300. PMLR.

\bibitem[Arjevani et~al., 2023]{arjevani2023lower}
Arjevani, Y., Carmon, Y., Duchi, J.~C., Foster, D.~J., Srebro, N., and Woodworth, B. (2023).
\newblock Lower bounds for non-convex stochastic optimization.
\newblock {\em Mathematical Programming}, 199(1-2):165--214.

\bibitem[Baruch et~al., 2019]{baruch2019little}
Baruch, G., Baruch, M., and Goldberg, Y. (2019).
\newblock A little is enough: Circumventing defenses for distributed learning.
\newblock {\em Advances in Neural Information Processing Systems}, 32.

\bibitem[Basu et~al., 2019]{basu2019qsparse}
Basu, D., Data, D., Karakus, C., and Diggavi, S. (2019).
\newblock Qsparse-local-sgd: Distributed sgd with quantization, sparsification and local computations.
\newblock {\em Advances in Neural Information Processing Systems}, 32.

\bibitem[Bernstein et~al., 2018]{bernstein2018signsgd}
Bernstein, J., Zhao, J., Azizzadenesheli, K., and Anandkumar, A. (2018).
\newblock signsgd with majority vote is communication efficient and fault tolerant.
\newblock {\em arXiv preprint arXiv:1810.05291}.

\bibitem[Beznosikov et~al., 2022]{beznosikov2022stochastic}
Beznosikov, A., Gorbunov, E., Berard, H., and Loizou, N. (2022).
\newblock Stochastic gradient descent-ascent: Unified theory and new efficient methods.
\newblock {\em arXiv preprint arXiv:2202.07262}.

\bibitem[Beznosikov et~al., 2020]{beznosikov2020biased}
Beznosikov, A., Horv{\'a}th, S., Richt{\'a}rik, P., and Safaryan, M. (2020).
\newblock On biased compression for distributed learning.
\newblock {\em arXiv preprint arXiv:2002.12410}.

\bibitem[Beznosikov et~al., 2021]{beznosikov2021distributed}
Beznosikov, A., Richt{\'a}rik, P., Diskin, M., Ryabinin, M., and Gasnikov, A. (2021).
\newblock Distributed methods with compressed communication for solving variational inequalities, with theoretical guarantees.
\newblock {\em arXiv preprint arXiv:2110.03313}.

\bibitem[Blanchard et~al., 2017a]{blanchard2017machine}
Blanchard, P., El~Mhamdi, E.~M., Guerraoui, R., and Stainer, J. (2017a).
\newblock Machine learning with adversaries: Byzantine tolerant gradient descent.
\newblock {\em Advances in Neural Information Processing Systems}, 30.

\bibitem[Blanchard et~al., 2017b]{blanchard2017byzantine}
Blanchard, P., El~Mhamdi, E.~M., Guerraoui, R., and Stainer, J. (2017b).
\newblock Machine learning with adversaries: {B}yzantine tolerant gradient descent.
\newblock In {\em Advances in Neural Information Processing Systems}, volume~30. Curran Associates, Inc.

\bibitem[Chang and Lin, 2011]{chang2011libsvm}
Chang, C.-C. and Lin, C.-J. (2011).
\newblock Libsvm: a library for support vector machines.
\newblock {\em ACM transactions on intelligent systems and technology (TIST)}, 2(3):1--27.

\bibitem[Chen et~al., 2018]{chen2018draco}
Chen, L., Wang, H., Charles, Z., and Papailiopoulos, D. (2018).
\newblock Draco: Byzantine-resilient distributed training via redundant gradients.
\newblock In {\em International Conference on Machine Learning}, pages 903--912. PMLR.

\bibitem[Chen et~al., 2017]{chen2017distributed}
Chen, Y., Su, L., and Xu, J. (2017).
\newblock Distributed statistical machine learning in adversarial settings: Byzantine gradient descent.
\newblock {\em Proceedings of the ACM on Measurement and Analysis of Computing Systems}, 1(2):1--25.

\bibitem[Cutkosky and Orabona, 2019]{cutkosky2019momentum}
Cutkosky, A. and Orabona, F. (2019).
\newblock Momentum-based variance reduction in non-convex sgd.
\newblock {\em Advances in neural information processing systems}, 32.

\bibitem[Damaskinos et~al., 2019]{damaskinos2019aggregathor}
Damaskinos, G., El-Mhamdi, E.-M., Guerraoui, R., Guirguis, A., and Rouault, S. (2019).
\newblock Aggregathor: Byzantine machine learning via robust gradient aggregation.
\newblock {\em Proceedings of Machine Learning and Systems}, 1:81--106.

\bibitem[Defazio et~al., 2014]{defazio2014saga}
Defazio, A., Bach, F., and Lacoste-Julien, S. (2014).
\newblock Saga: A fast incremental gradient method with support for non-strongly convex composite objectives.
\newblock {\em Advances in neural information processing systems}, 27.

\bibitem[Diskin et~al., 2021]{diskin2021distributed}
Diskin, M., Bukhtiyarov, A., Ryabinin, M., Saulnier, L., Sinitsin, A., Popov, D., Pyrkin, D.~V., Kashirin, M., Borzunov, A., Villanova~del Moral, A., et~al. (2021).
\newblock Distributed deep learning in open collaborations.
\newblock {\em Advances in Neural Information Processing Systems}, 34:7879--7897.

\bibitem[Faghri et~al., 2020]{faghri2020adaptive}
Faghri, F., Tabrizian, I., Markov, I., Alistarh, D., Roy, D.~M., and Ramezani-Kebrya, A. (2020).
\newblock Adaptive gradient quantization for data-parallel sgd.
\newblock {\em Advances in neural information processing systems}, 33:3174--3185.

\bibitem[Fang et~al., 2018]{fang2018spider}
Fang, C., Li, C.~J., Lin, Z., and Zhang, T. (2018).
\newblock Spider: Near-optimal non-convex optimization via stochastic path-integrated differential estimator.
\newblock {\em Advances in Neural Information Processing Systems}, 31.

\bibitem[Fang et~al., 2020]{fang2020local}
Fang, M., Cao, X., Jia, J., and Gong, N. (2020).
\newblock Local model poisoning attacks to $\{$Byzantine-Robust$\}$ federated learning.
\newblock In {\em 29th USENIX security symposium (USENIX Security 20)}, pages 1605--1622.

\bibitem[Fatkhullin et~al., 2021]{fatkhullin2021ef21}
Fatkhullin, I., Sokolov, I., Gorbunov, E., Li, Z., and Richt{\'a}rik, P. (2021).
\newblock {EF21} with bells \& whistles: Practical algorithmic extensions of modern error feedback.
\newblock {\em arXiv preprint arXiv:2110.03294}.

\bibitem[Ghosh et~al., 2021]{ghosh2021communication}
Ghosh, A., Maity, R.~K., Kadhe, S., Mazumdar, A., and Ramchandran, K. (2021).
\newblock Communication-efficient and byzantine-robust distributed learning with error feedback.
\newblock {\em IEEE Journal on Selected Areas in Information Theory}, 2(3):942--953.

\bibitem[Ghosh et~al., 2020]{ghosh2020distributed}
Ghosh, A., Maity, R.~K., and Mazumdar, A. (2020).
\newblock Distributed newton can communicate less and resist byzantine workers.
\newblock {\em Advances in Neural Information Processing Systems}, 33:18028--18038.

\bibitem[Goodall, 1951]{goodall1951television}
Goodall, W. (1951).
\newblock Television by pulse code modulation.
\newblock {\em Bell System Technical Journal}, 30(1):33--49.

\bibitem[Gorbunov et~al., 2021a]{gorbunov2021secure}
Gorbunov, E., Borzunov, A., Diskin, M., and Ryabinin, M. (2021a).
\newblock Secure distributed training at scale.
\newblock {\em arXiv preprint arXiv:2106.11257}.

\bibitem[Gorbunov et~al., 2021b]{gorbunov2021marina}
Gorbunov, E., Burlachenko, K.~P., Li, Z., and Richt{\'a}rik, P. (2021b).
\newblock {M}{A}{R}{I}{N}{A}: {F}aster non-convex distributed learning with compression.
\newblock In {\em International Conference on Machine Learning}, pages 3788--3798. PMLR.

\bibitem[Gorbunov et~al., 2023]{gorbunov2023variance}
Gorbunov, E., Horv{\'a}th, S., Richt{\'a}rik, P., and Gidel, G. (2023).
\newblock Variance reduction is an antidote to byzantines: Better rates, weaker assumptions and communication compression as a cherry on the top.

\bibitem[Gorbunov et~al., 2020]{gorbunov2020linearly}
Gorbunov, E., Kovalev, D., Makarenko, D., and Richt{\'a}rik, P. (2020).
\newblock Linearly converging error compensated sgd.
\newblock {\em Advances in Neural Information Processing Systems}, 33:20889--20900.

\bibitem[Gruntkowska et~al., 2023]{gruntkowska2023ef21}
Gruntkowska, K., Tyurin, A., and Richt{\'a}rik, P. (2023).
\newblock Ef21-p and friends: Improved theoretical communication complexity for distributed optimization with bidirectional compression.
\newblock In {\em International Conference on Machine Learning}, pages 11761--11807. PMLR.

\bibitem[Guerraoui et~al., 2018]{guerraoui2018hidden}
Guerraoui, R., Rouault, S., et~al. (2018).
\newblock The hidden vulnerability of distributed learning in byzantium.
\newblock In {\em International Conference on Machine Learning}, pages 3521--3530. PMLR.

\bibitem[Haddadpour et~al., 2021]{haddadpour2021federated}
Haddadpour, F., Kamani, M.~M., Mokhtari, A., and Mahdavi, M. (2021).
\newblock Federated learning with compression: Unified analysis and sharp guarantees.
\newblock In {\em International Conference on Artificial Intelligence and Statistics}, pages 2350--2358. PMLR.

\bibitem[Horv{\'a}th et~al., 2019a]{horvath2019natural}
Horv{\'a}th, S., Ho, C.-Y., Horvath, L., Sahu, A.~N., Canini, M., and Richt{\'a}rik, P. (2019a).
\newblock Natural compression for distributed deep learning.
\newblock {\em arXiv preprint arXiv:1905.10988}.

\bibitem[Horv{\'a}th et~al., 2019b]{horvath2019stochastic}
Horv{\'a}th, S., Kovalev, D., Mishchenko, K., Stich, S., and Richt{\'a}rik, P. (2019b).
\newblock Stochastic distributed learning with gradient quantization and variance reduction.
\newblock {\em arXiv preprint arXiv:1904.05115}.

\bibitem[Islamov et~al., 2021]{islamov2021distributed}
Islamov, R., Qian, X., and Richt{\'a}rik, P. (2021).
\newblock Distributed second order methods with fast rates and compressed communication.
\newblock In {\em International Conference on Machine Learning}, pages 4617--4628. PMLR.

\bibitem[Kairouz et~al., 2021]{kairouz2021advances}
Kairouz, P., McMahan, H.~B., Avent, B., Bellet, A., Bennis, M., Bhagoji, A.~N., Bonawitz, K., Charles, Z., Cormode, G., Cummings, R., et~al. (2021).
\newblock Advances and open problems in federated learning.
\newblock {\em Foundations and Trends{\textregistered} in Machine Learning}, 14(1--2):1--210.

\bibitem[Karimireddy et~al., 2021]{karimireddy2021learning}
Karimireddy, S.~P., He, L., and Jaggi, M. (2021).
\newblock Learning from history for byzantine robust optimization.
\newblock In {\em International Conference on Machine Learning}, pages 5311--5319. PMLR.

\bibitem[Karimireddy et~al., 2022]{karimireddy2020byzantine}
Karimireddy, S.~P., He, L., and Jaggi, M. (2022).
\newblock Byzantine-robust learning on heterogeneous datasets via bucketing.
\newblock {\em International Conference on Learning Representations}.

\bibitem[Karimireddy et~al., 2019]{karimireddy2019error}
Karimireddy, S.~P., Rebjock, Q., Stich, S., and Jaggi, M. (2019).
\newblock Error feedback fixes signsgd and other gradient compression schemes.
\newblock In {\em International Conference on Machine Learning}, pages 3252--3261. PMLR.

\bibitem[Khirirat et~al., 2018]{khirirat2018distributed}
Khirirat, S., Feyzmahdavian, H.~R., and Johansson, M. (2018).
\newblock Distributed learning with compressed gradients.
\newblock {\em arXiv preprint arXiv:1806.06573}.

\bibitem[Kijsipongse et~al., 2018]{kijsipongse2018hybrid}
Kijsipongse, E., Piyatumrong, A., et~al. (2018).
\newblock A hybrid gpu cluster and volunteer computing platform for scalable deep learning.
\newblock {\em The Journal of Supercomputing}, 74(7):3236--3263.

\bibitem[Koloskova et~al., 2019]{koloskova2019decentralized}
Koloskova, A., Stich, S., and Jaggi, M. (2019).
\newblock Decentralized stochastic optimization and gossip algorithms with compressed communication.
\newblock In {\em International Conference on Machine Learning}, pages 3478--3487. PMLR.

\bibitem[Kone\v{c}n\'{y} et~al., 2016]{FEDLEARN}
Kone\v{c}n\'{y}, J., McMahan, H.~B., Yu, F., Richt\'{a}rik, P., Suresh, A.~T., and Bacon, D. (2016).
\newblock Federated learning: strategies for improving communication efficiency.
\newblock In {\em NIPS Private Multi-Party Machine Learning Workshop}.

\bibitem[Kovalev et~al., 2021]{kovalev2021linearly}
Kovalev, D., Koloskova, A., Jaggi, M., Richtarik, P., and Stich, S. (2021).
\newblock A linearly convergent algorithm for decentralized optimization: Sending less bits for free!
\newblock In {\em International Conference on Artificial Intelligence and Statistics}, pages 4087--4095. PMLR.

\bibitem[Lamport et~al., 1982]{lamport1982byzantine}
Lamport, L., Shostak, R., and Pease, M. (1982).
\newblock The byzantine generals problem.
\newblock {\em ACM Transactions on Programming Languages and Systems}, 4(3):382--401.

\bibitem[Li, 2020]{gpt3costlambda}
Li, C. (2020).
\newblock Demystifying gpt-3 language model: A technical overview.

\bibitem[Li et~al., 2023]{li2023experimental}
Li, S., Ngai, E. C.-H., and Voigt, T. (2023).
\newblock An experimental study of {B}yzantine-robust aggregation schemes in federated learning.
\newblock {\em IEEE Transactions on Big Data}.

\bibitem[Li et~al., 2021]{li2021page}
Li, Z., Bao, H., Zhang, X., and Richt{\'a}rik, P. (2021).
\newblock {PAGE}: A simple and optimal probabilistic gradient estimator for nonconvex optimization.
\newblock In {\em International Conference on Machine Learning}, pages 6286--6295. PMLR.

\bibitem[Li et~al., 2020]{li2020acceleration}
Li, Z., Kovalev, D., Qian, X., and Richtarik, P. (2020).
\newblock Acceleration for compressed gradient descent in distributed and federated optimization.
\newblock In {\em International Conference on Machine Learning}, pages 5895--5904. PMLR.

\bibitem[Li and Richt{\'a}rik, 2021]{li2021canita}
Li, Z. and Richt{\'a}rik, P. (2021).
\newblock Canita: Faster rates for distributed convex optimization with communication compression.
\newblock {\em Advances in Neural Information Processing Systems}, 34.

\bibitem[Liu et~al., 2020]{liu2020optimal}
Liu, D., Nguyen, L.~M., and Tran-Dinh, Q. (2020).
\newblock An optimal hybrid variance-reduced algorithm for stochastic composite nonconvex optimization.
\newblock {\em arXiv preprint arXiv:2008.09055}.

\bibitem[Lyu et~al., 2020]{lyu2020privacy}
Lyu, L., Yu, H., Ma, X., Sun, L., Zhao, J., Yang, Q., and Yu, P.~S. (2020).
\newblock Privacy and robustness in federated learning: Attacks and defenses.
\newblock {\em arXiv preprint arXiv:2012.06337}.

\bibitem[Mishchenko et~al., 2019]{mishchenko2019distributed}
Mishchenko, K., Gorbunov, E., Tak{\'a}{\v{c}}, M., and Richt{\'a}rik, P. (2019).
\newblock Distributed learning with compressed gradient differences.
\newblock {\em arXiv preprint arXiv:1901.09269}.

\bibitem[Nguyen et~al., 2017]{nguyen2017sarah}
Nguyen, L.~M., Liu, J., Scheinberg, K., and Tak{\'a}{\v{c}}, M. (2017).
\newblock Sarah: A novel method for machine learning problems using stochastic recursive gradient.
\newblock In {\em International Conference on Machine Learning}, pages 2613--2621. PMLR.

\bibitem[OpenAI, 2023]{gpt4}
OpenAI (2023).
\newblock {GPT}-4 technical report.

\bibitem[{\"O}zfatura et~al., 2023]{ozfatura2023byzantines}
{\"O}zfatura, K., {\"O}zfatura, E., K{\"u}p{\c{c}}{\"u}, A., and G{\"u}nd{\"u}z, D. (2023).
\newblock Byzantines can also learn from history: Fall of centered clipping in federated learning.
\newblock {\em IEEE Transactions on Information Forensics and Security}.

\bibitem[Philippenko and Dieuleveut, 2020]{philippenko2020bidirectional}
Philippenko, C. and Dieuleveut, A. (2020).
\newblock Bidirectional compression in heterogeneous settings for distributed or federated learning with partial participation: tight convergence guarantees.
\newblock {\em arXiv preprint arXiv:2006.14591}.

\bibitem[Philippenko and Dieuleveut, 2021]{philippenko2021preserved}
Philippenko, C. and Dieuleveut, A. (2021).
\newblock Preserved central model for faster bidirectional compression in distributed settings.
\newblock {\em Advances in Neural Information Processing Systems}, 34.

\bibitem[Pillutla et~al., 2022]{pillutla2022robust}
Pillutla, K., Kakade, S.~M., and Harchaoui, Z. (2022).
\newblock Robust aggregation for federated learning.
\newblock {\em IEEE Transactions on Signal Processing}, 70:1142--1154.

\bibitem[Qian et~al., 2021]{qian2021error}
Qian, X., Richt{\'a}rik, P., and Zhang, T. (2021).
\newblock Error compensated distributed sgd can be accelerated.
\newblock {\em Advances in Neural Information Processing Systems}, 34.

\bibitem[Rajput et~al., 2019]{rajput2019detox}
Rajput, S., Wang, H., Charles, Z., and Papailiopoulos, D. (2019).
\newblock Detox: A redundancy-based framework for faster and more robust gradient aggregation.
\newblock {\em Advances in Neural Information Processing Systems}, 32.

\bibitem[Regatti et~al., 2020]{regatti2020bygars}
Regatti, J., Chen, H., and Gupta, A. (2020).
\newblock {ByGARS}: Byzantine {SGD} with arbitrary number of attackers.
\newblock {\em arXiv preprint arXiv:2006.13421}.

\bibitem[Richt{\'a}rik et~al., 2021]{richtarik2021ef21}
Richt{\'a}rik, P., Sokolov, I., and Fatkhullin, I. (2021).
\newblock {E}{F}21: A new, simpler, theoretically better, and practically faster error feedback.
\newblock {\em Advances in Neural Information Processing Systems}, 34.

\bibitem[Roberts, 1962]{roberts1962picture}
Roberts, L. (1962).
\newblock Picture coding using pseudo-random noise.
\newblock {\em IRE Transactions on Information Theory}, 8(2):145--154.

\bibitem[Rodr{\'\i}guez-Barroso et~al., 2020]{rodriguez2020dynamic}
Rodr{\'\i}guez-Barroso, N., Mart{\'\i}nez-C{\'a}mara, E., Luz{\'o}n, M., Seco, G.~G., Veganzones, M.~{\'A}., and Herrera, F. (2020).
\newblock Dynamic federated learning model for identifying adversarial clients.
\newblock {\em arXiv preprint arXiv:2007.15030}.

\bibitem[Sadiev et~al., 2022]{sadiev2022federated}
Sadiev, A., Malinovsky, G., Gorbunov, E., Sokolov, I., Khaled, A., Burlachenko, K., and Richt{\'a}rik, P. (2022).
\newblock Federated optimization algorithms with random reshuffling and gradient compression.
\newblock {\em arXiv preprint arXiv:2206.07021}.

\bibitem[Safaryan et~al., 2021]{safaryan2021fednl}
Safaryan, M., Islamov, R., Qian, X., and Richt{\'a}rik, P. (2021).
\newblock Fednl: Making newton-type methods applicable to federated learning.
\newblock {\em arXiv preprint arXiv:2106.02969}.

\bibitem[Sahu et~al., 2021]{sahu2021rethinking}
Sahu, A., Dutta, A., M~Abdelmoniem, A., Banerjee, T., Canini, M., and Kalnis, P. (2021).
\newblock Rethinking gradient sparsification as total error minimization.
\newblock {\em Advances in Neural Information Processing Systems}, 34.

\bibitem[Seide et~al., 2014]{seide20141}
Seide, F., Fu, H., Droppo, J., Li, G., and Yu, D. (2014).
\newblock 1-bit stochastic gradient descent and its application to data-parallel distributed training of speech dnns.
\newblock In {\em Fifteenth Annual Conference of the International Speech Communication Association}. Citeseer.

\bibitem[Stich et~al., 2018]{stich2018sparsified}
Stich, S.~U., Cordonnier, J.-B., and Jaggi, M. (2018).
\newblock Sparsified sgd with memory.
\newblock {\em Advances in Neural Information Processing Systems}, 31.

\bibitem[Su and Vaidya, 2016]{su2016fault}
Su, L. and Vaidya, N.~H. (2016).
\newblock Fault-tolerant multi-agent optimization: optimal iterative distributed algorithms.
\newblock In {\em Proceedings of the 2016 ACM symposium on principles of distributed computing}, pages 425--434.

\bibitem[Suresh et~al., 2017]{suresh2017distributed}
Suresh, A.~T., Felix, X.~Y., Kumar, S., and McMahan, H.~B. (2017).
\newblock Distributed mean estimation with limited communication.
\newblock In {\em International Conference on Machine Learning}, pages 3329--3337. PMLR.

\bibitem[Szlendak et~al., 2021]{szlendak2021permutation}
Szlendak, R., Tyurin, A., and Richt{\'a}rik, P. (2021).
\newblock Permutation compressors for provably faster distributed nonconvex optimization.
\newblock {\em arXiv preprint arXiv:2110.03300}.

\bibitem[Tang et~al., 2019]{tang2019doublesqueeze}
Tang, H., Yu, C., Lian, X., Zhang, T., and Liu, J. (2019).
\newblock {D}ouble{S}queeze: Parallel stochastic gradient descent with double-pass error-compensated compression.
\newblock In {\em International Conference on Machine Learning}, pages 6155--6165.

\bibitem[Tao et~al., 2023]{tao2023byzantine}
Tao, Y., Cui, S., Xu, W., Yin, H., Yu, D., Liang, W., and Cheng, X. (2023).
\newblock Byzantine-resilient federated learning at edge.
\newblock {\em IEEE Transactions on Computers}.

\bibitem[Tran-Dinh et~al., 2022]{tran2022hybrid}
Tran-Dinh, Q., Pham, N.~H., Phan, D.~T., and Nguyen, L.~M. (2022).
\newblock A hybrid stochastic optimization framework for composite nonconvex optimization.
\newblock {\em Mathematical Programming}, 191(2):1005--1071.

\bibitem[Tyurin and Richt{\'a}rik, 2023a]{tyurin20232direction}
Tyurin, A. and Richt{\'a}rik, P. (2023a).
\newblock 2direction: Theoretically faster distributed training with bidirectional communication compression.
\newblock {\em arXiv preprint arXiv:2305.12379}.

\bibitem[Tyurin and Richt{\'a}rik, 2023b]{tyurin2022dasha}
Tyurin, A. and Richt{\'a}rik, P. (2023b).
\newblock {D}{A}{S}{H}{A}: Distributed nonconvex optimization with communication compression, optimal oracle complexity, and no client synchronization.
\newblock {\em International Conference on Learning Representations}.

\bibitem[Vaswani et~al., 2019]{vaswani2019fast}
Vaswani, S., Bach, F., and Schmidt, M. (2019).
\newblock Fast and faster convergence of sgd for over-parameterized models and an accelerated perceptron.
\newblock In {\em The 22nd international conference on artificial intelligence and statistics}, pages 1195--1204. PMLR.

\bibitem[Vogels et~al., 2019]{vogels2019powersgd}
Vogels, T., Karimireddy, S.~P., and Jaggi, M. (2019).
\newblock Powersgd: Practical low-rank gradient compression for distributed optimization.
\newblock {\em Advances in Neural Information Processing Systems}, 32.

\bibitem[Weiszfeld, 1937]{Weiszfeld1937}
Weiszfeld, E. (1937).
\newblock Sur le point pour lequel la somme des distances de n points donnés est minimum.
\newblock {\em Tohoku Mathematical Journal, First Series}, 43:355--386.

\bibitem[Wen et~al., 2017]{wen2017terngrad}
Wen, W., Xu, C., Yan, F., Wu, C., Wang, Y., Chen, Y., and Li, H. (2017).
\newblock Terngrad: Ternary gradients to reduce communication in distributed deep learning.
\newblock {\em Advances in neural information processing systems}, 30.

\bibitem[Wu et~al., 2020]{wu2020federated}
Wu, Z., Ling, Q., Chen, T., and Giannakis, G.~B. (2020).
\newblock Federated variance-reduced stochastic gradient descent with robustness to byzantine attacks.
\newblock {\em IEEE Transactions on Signal Processing}, 68:4583--4596.

\bibitem[Xie et~al., 2020]{xie2020fall}
Xie, C., Koyejo, O., and Gupta, I. (2020).
\newblock Fall of empires: Breaking byzantine-tolerant sgd by inner product manipulation.
\newblock In {\em Uncertainty in Artificial Intelligence}, pages 261--270. PMLR.

\bibitem[Xu and Lyu, 2020]{xu2020towards}
Xu, X. and Lyu, L. (2020).
\newblock Towards building a robust and fair federated learning system.
\newblock {\em arXiv preprint arXiv:2011.10464}.

\bibitem[Yin et~al., 2018]{yin2018byzantine}
Yin, D., Chen, Y., Kannan, R., and Bartlett, P. (2018).
\newblock Byzantine-robust distributed learning: Towards optimal statistical rates.
\newblock In {\em International Conference on Machine Learning}, pages 5650--5659. PMLR.

\bibitem[Zhu and Ling, 2021]{zhu2021broadcast}
Zhu, H. and Ling, Q. (2021).
\newblock Broadcast: Reducing both stochastic and compression noise to robustify communication-efficient federated learning.
\newblock {\em arXiv preprint arXiv:2104.06685}.

\bibitem[Zinkevich et~al., 2010]{zinkevich2010parallelized}
Zinkevich, M., Weimer, M., Li, L., and Smola, A. (2010).
\newblock Parallelized stochastic gradient descent.
\newblock {\em Advances in neural information processing systems}, 23.

\end{thebibliography}

\section*{Checklist}

 \begin{enumerate}

 \item For all models and algorithms presented, check if you include:
 \begin{enumerate}
   \item A clear description of the mathematical setting, assumptions, algorithm, and/or model. [Yes]
   \item An analysis of the properties and complexity (time, space, sample size) of any algorithm. [Yes]
   \item (Optional) Anonymized source code, with specification of all dependencies, including external libraries. [Yes]
 \end{enumerate}

 \item For any theoretical claim, check if you include:
 \begin{enumerate}
   \item Statements of the full set of assumptions of all theoretical results. [Yes]
   \item Complete proofs of all theoretical results. [Yes]
   \item Clear explanations of any assumptions. [Yes]     
 \end{enumerate}

 \item For all figures and tables that present empirical results, check if you include:
 \begin{enumerate}
   \item The code, data, and instructions needed to reproduce the main experimental results (either in the supplemental material or as a URL). [Yes]
   \item All the training details (e.g., data splits, hyperparameters, how they were chosen). [Yes]
         \item A clear definition of the specific measure or statistics and error bars (e.g., with respect to the random seed after running experiments multiple times). [Yes]
         \item A description of the computing infrastructure used. (e.g., type of GPUs, internal cluster, or cloud provider). [Yes]
 \end{enumerate}

 \item If you are using existing assets (e.g., code, data, models) or curating/releasing new assets, check if you include:
 \begin{enumerate}
   \item Citations of the creator If your work uses existing assets. [Yes]
   \item The license information of the assets, if applicable. [Yes]
   \item New assets either in the supplemental material or as a URL, if applicable. [Not Applicable]
   \item Information about consent from data providers/curators. [Not Applicable]
   \item Discussion of sensible content if applicable, e.g., personally identifiable information or offensive content. [Not Applicable]
 \end{enumerate}

 \item If you used crowdsourcing or conducted research with human subjects, check if you include:
 \begin{enumerate}
   \item The full text of instructions given to participants and screenshots. [Not Applicable]
   \item Descriptions of potential participant risks, with links to Institutional Review Board (IRB) approvals if applicable. [Not Applicable]
   \item The estimated hourly wage paid to participants and the total amount spent on participant compensation. [Not Applicable]
 \end{enumerate}

 \end{enumerate}

\clearpage

% \begin{thebibliography}{}
% \setlength{\itemindent}{-\leftmargin}
% \makeatletter\renewcommand{\@biblabel}[1]{}\makeatother
% \bibitem{} J.~Alspector, B.~Gupta, and R.~B.~Allen (1989).
%     \newblock Performance of a stochastic learning microchip.
%     \newblock In D. S. Touretzky (ed.),
%     \textit{Advances in Neural Information Processing Systems 1}, 748--760.
%     San Mateo, Calif.: Morgan Kaufmann.

% \bibitem{} F.~Rosenblatt (1962).
%     \newblock \textit{Principles of Neurodynamics.}
%     \newblock Washington, D.C.: Spartan Books.

% \bibitem{} G.~Tesauro (1989).
%     \newblock Neurogammon wins computer Olympiad.
%     \newblock \textit{Neural Computation} \textbf{1}(3):321--323.
% \end{thebibliography}

%%%%%%%%%%%%%%%%%%%%%%%%%%%%%%%%%%%%%%%%%%%%%%%%%%%%%%%%%%%%

\appendix
\onecolumn
\tableofcontents
\clearpage

\section{EXTRA RELATED WORK}\label{appendix:extra_related_work}

\paragraph{Unbiased compression.} The first convergence results for the methods using unbiased compression of (stochastic) gradients are established by \citet{alistarh2017qsgd, wen2017terngrad, khirirat2018distributed}. \citet{mishchenko2019distributed} propose an approach based on the compression of certain gradient differences, achieving convergence to arbitrary accuracy. In this method, workers compute full gradients and a constant stepsize is used. This approach is generalized and combined with variance reduction by \citet{horvath2019stochastic}. The methods achieving current state-of-the-art theoretical complexities in the non-convex case are developed by \citet{gorbunov2021marina} and \citet{tyurin2022dasha}. Recent advancements in the literature on methods utilising unbiased compression include adaptive compression~\citep{faghri2020adaptive}, acceleration~\citep{li2020acceleration, li2021canita, qian2021error}, decentralized communication~\citep{kovalev2021linearly}, local steps~\citep{basu2019qsparse, haddadpour2021federated, sadiev2022federated}, second-order methods~\citep{islamov2021distributed, safaryan2021fednl}, and methods for the saddle-point problems~\citep{beznosikov2021distributed, beznosikov2022stochastic}.

\paragraph{Biased compression and error feedback.} Methods with biased compression and error feedback are known to perform well in practice \citep{seide20141, vogels2019powersgd}. In the non-convex case, which is the primary focus of our work, standard error feedback is analyzed by \citet{karimireddy2019error, beznosikov2020biased, koloskova2019decentralized, sahu2021rethinking}. However, the existing complexity bounds for standard error feedback either have an explicit dependence on the heterogeneity parameter $\zeta^2$ or require boundedness of the gradients. These issues are resolved by \citet{richtarik2021ef21}, who propose a novel version of error feedback, named \algname{EF21}. This approach is further extended in various directions by \citet{fatkhullin2021ef21}.

\paragraph{Bidirectional compression.} In some applications, the communication cost of downloading information from the server is comparable to the cost of uploading it, as observed in data from Speedtest.net\footnote{\url{https://www.speedtest.net/global-index}}. Consequently, numerous studies not only address uplink (workers-to-server) compression costs, but also focus on downlink (server-to-workers) compression. \cite{philippenko2020bidirectional, gorbunov2020linearly} utilize unbiased compression in both uplink and downlink communication, and \citet{philippenko2021preserved} improve these results by employing the \algname{DIANA} mechanism \citep{mishchenko2019distributed} on both worker and server sides. There also exist several works considering biased compression. \citet{tang2019doublesqueeze} use standard error feedback, and \citet{fatkhullin2021ef21} apply \algname{EF21} on both workers and server sides. The latter approach is further refined by \cite{gruntkowska2023ef21, tyurin20232direction}.

\paragraph{Comparison with \citet{tao2023byzantine}.}
\citet{tao2023byzantine} consider a variant of \algname{Parallel SGD} with coordinate-wise Truncated Mean aggregation and coordinate-wise robust estimate of the gradient on regular workers to handle the heavy-tailed noise in the stochastic gradients. Additionally, the authors apply contractive compression to messages communicated by the clients. However, their approach assumes that all clients sample from the same distribution and have a finite number of samples, while we consider a setup with bounded heterogeneity without assuming statistical similarity between local datasets on each client. Moreover, \citet{tao2023byzantine} focus solely on strongly convex problems, deriving high-probability results. Thus, one cannot formally compare our theoretical guarantees. Importantly, \citet{tao2023byzantine} assume that the second moments of the stochastic gradients are bounded, which implies boundedness of the gradient of the objective function. This assumption is known to be quite restrictive and, for example, does not hold for globally strongly convex functions.  Finally, in the special case of homogeneous data, our analysis recovers convergence to any predefined accuracy, while theirs ensures convergence only to the (non-reducible by stepsize) neighborhood of the solution, with a radius proportional to the problem dimension, which can be huge.

\clearpage

\section{EXAMPLES OF ROBUST AGGREGATORS AND COMPRESSION OPERATORS} \label{appendix:robust_aggr_and_compr}

\subsection{Robust Aggregators}

In recent years, research efforts have given rise to various aggregation rules asserted to be Byzantine-robust under certain assumptions \citep{li2023experimental}. However, it turns out that these rules do not satisfy Definition \ref{def:RAgg_def}, and there exist practical scenarios where methods employing them fail to converge \citep{karimireddy2021learning}.

In this section, we present a few examples of such rules and consider a technique known as \textit{bucketing}, proposed by \cite{karimireddy2020byzantine}, which robustify them.
\\

\textbf{Geometric Median} \citep{chen2017distributed, pillutla2022robust}:
\texttt{GM}-estimator, also known as Robust Federated Averaging (\texttt{RFA}), is an aggregation rule based on geometric median:
\begin{align*}
    \texttt{GM}(x_1, \ldots, x_n)
        \eqdef \argmin_{x\in\R^d} \sum_{i=1}^n \norm{x-x_i}.
\end{align*}
An approximate solution to the above problem can be obtained through several iterations of the smoothed Weiszfeld algorithm, with $\cO(n)$ cost per iteration \citep{Weiszfeld1937, pillutla2022robust}.

\textbf{Coordinate-wise Median} \citep{yin2018byzantine}:
\texttt{CM}-estimator operates component-wise, assigning to each element in the output vector the median of the corresponding elements in the input vectors, i.e.,
\begin{align*}
    [\texttt{CM} (x_1, \ldots, x_n)]_j
    \eqdef \texttt{Median}([x_1]_j,\ldots,[x_n]_j),
        % = \argmin_{u\in\R} \sum_{i=1}^n |u-[x_i]_j|,
\end{align*}
where $[x]_j$ is $j$-th coordinate of vector $x\in\R^d$. The method incurs a cost of $\cO(n)$.

\textbf{Krum} \citep{blanchard2017byzantine}:
\texttt{Krum} outputs a vector $x_i$ which is closest to the mean of the input vectors when the most extreme inputs are excluded. In mathematical terms, let $S_i \subseteq \{x_1, \ldots, x_n\}$ represent the subset of $n-|\cB|-2$ vectors closest to $x_i$. Then
\begin{align*}
    \texttt{Krum}(x_1, \ldots, x_n)
    \eqdef \argmin_{x_i\in\{x_1, \ldots, x_n\}} \sum_{j\in S_i} \norm{x_j-x_i}^2.
\end{align*}
This estimator has a serious limitation, namely the computational cost. Since it requires calculating all pairwise distances between vectors $x_1, \ldots, x_n$, this cost is $\cO(n^2)$.

\paragraph{Bucketing \citep{karimireddy2020byzantine}:}
The $s$-bucketing trick (Algorithm \ref{algorithm:bucketing}) first divides the $n$ inputs into $\lceil n/s \rceil$ buckets, so that each bucket contains at most $s$ elements. The vectors in each bucket are then averaged and employed as input of some aggregator $\texttt{Aggr}$.

\begin{algorithm}[H]\label{algorithm:bucketing}
%\footnotesize
    \caption{Robust Aggregation with Bucketing \citep{karimireddy2020byzantine}}
    \label{algorithm:bucketing}
    \begin{algorithmic}[1]
    \State \textbf{Input:} $\{x_1,\ldots,x_n\}$, bucket size $s\in\mathbb{N}$, aggregation rule \texttt{Aggr}
    \State Sample a random permutation $\pi = \parens{\pi(1), \ldots, \pi(n)}$ of $[n]$
    \State Set $y_i = \frac{1}{s} \sum_{k=s(i-1)+1}^{\min\{si, n\}} x_{\pi(k)}$ for $i=1,\ldots,\lceil n/s \rceil$
    \State \textbf{Return:} $\hat{x} = \texttt{Aggr}(y_1,\ldots,y_{\lceil n/s \rceil})$
    \end{algorithmic}
\end{algorithm}
\cite{gorbunov2023variance} show that \texttt{Krum}, \texttt{RFA} and \texttt{CM} with bucketing satisfy Definition \ref{def:RAgg_def}.
\\

\subsection{Compression Operators}

Below we provide several examples of compression operators belonging to the classes of unbiased (Definition \ref{def:quantization}) and contractive (Definition \ref{def:contractive_compr}) compressors. For a more thorough overview, we refer the reader to \cite{beznosikov2020biased}.

\begin{enumerate}
    \item \textbf{Rand$K$ sparsification} \citep{stich2018sparsified}:
    The Rand-$K \in \mathbb{U}(\frac{d}{K} - 1)$ compressor retains $K\in[d]$ random values of the input vector and scales it by $\frac{d}{K}$:
    \[
    \mathcal{C}(x) = \frac{d}{K} \sum_{i\in S} x_ie_i,
    \]
    where $S$ is a random subset of $[d]$ and $e_1, \ldots, e_d \in \R^d$ are the standard unit basis vectors. 
    \item \textbf{Natural compression} \citep{horvath2019natural}: The compressor is defined component-wise via
    $(\mathcal{C}(x))_i = \mathcal{C}(x_i)$. For $x_i\in\R$, we set
    \begin{align*}
        \mathcal{C}(0) \eqdef 0
    \end{align*}
    and for $x_i\neq 0$
    \begin{align*}
        \mathcal{C}(x_i) \eqdef
        \begin{cases}
        \text{sign}(x_i) \cdot 2^{\lfloor \log_2|x_i| \rfloor} & \text{with probability } p(x_i),\\
        \text{sign}(x_i) \cdot 2^{\lceil \log_2|x_i| \rceil} & \text{with probability } 1-p(x_i),
        \end{cases}
    \end{align*}
    where
    \[p(x_i) \eqdef \frac{2^{\lceil \log_2|x_i| \rceil} - |x_i|}{2^{\lfloor \log_2|x_i| \rfloor}}.\]
    It can be shown that $\mathcal{C} \in \mathbb{U}(\frac{1}{8})$.
    \item \textbf{Top$K$ sparsification} \citep{alistarh2018convergence}:
    The Top-$K \in \mathbb{B}(\frac{K}{d})$ sparsifier keeps the $K\in[d]$ largest coordinates in magnitude (ordered so that $|x_{(1)}|\leq\ldots\leq|x_{(d)}|$):
    \[
    \mathcal{C}(x) = \sum_{i=d-K+1}^d x_{(i)}e_{(i)}.
    \]
    \item \textbf{Biased random sparsification} \citep{beznosikov2020biased}: Let $S\subseteq[d]$ and $p_i = \Prob{i\in S}>0$ for $i\in[d]$ and define
    \[
    \mathcal{C}(x) = \sum_{i\in S} x_ie_i.
    \]
    The biased random sparsification operator belongs to $\mathbb{B}(\min_i p_i)$.
    \item \textbf{General biased rounding} \citep{beznosikov2020biased}: Let
    \begin{align*}
        (\mathcal{C}(x))_i = \text{sign}(x_i) \arg\min_{t\in(a_k)} |t-|x_i||, \qquad i\in[d],
    \end{align*}
    where $(a_k)_{k\in\Z}$ is an increasing sequence of positive numbers such that $\inf a_k = 0$ and $\sup a_k = \infty$. This operator belongs to $\mathbb{B}(\alpha)$, where $\alpha^{-1} = \sup_{k\in\mathbb{Z}} \frac{(a_k+a_{k+1})^2}{4a_k a_{k+1}}$.
\end{enumerate}

\clearpage

\section{USEFUL IDENTITIES AND INEQUALITIES}

For all $x, y \in \R^d$, $s>0$ and $\alpha\in(0,1]$, we have:
\begin{align}
    \label{eq:young}
    \norm{x+y}^2 &\leq (1+s) \norm{x}^2 + (1+s^{-1}) \norm{y}^2, \\
    \label{eq:ineq1}
    \left( 1 - \alpha \right)\left( 1 + \frac{\alpha}{2} \right) &\leq 1 - \frac{\alpha}{2}, \\
    \label{eq:ineq2}
    \left( 1 - \alpha \right)\left( 1 + \frac{2}{\alpha} \right) &\leq \frac{2}{\alpha}.
\end{align}
\textbf{Jensen's inequality:} If $f$ is a convex function and $X$ is a random variable, then
\begin{align}
\label{eq:jensen}
    \EE{f(X)} \geq f\left(\EE{X}\right).
\end{align}
\textbf{Variance decomposition:} For any random vector $X\in\R^d$ and any non-random vector $c\in\R^d$, we have
\begin{align}\label{eq:vardecomp}
   \EE{\norm{X-c}^2} = \EE{\norm{X - \EE{X}}^2} + \norm{\EE{X}-c}^2.
\end{align}
\textbf{Tower property:} For any random variables $X$ and $Y$, we have
\begin{align}
\label{eq:tower}
    \EE{\EE{X\,|\,Y}} = \EE{X}.
\end{align}
\begin{lemma}[Lemma $2$ of \citep{li2021page}]
\label{lemma:page}
Suppose that function $f$ is $L$-smooth and let $x^{t+1} = x^t - \gamma g^t$. Then for any $g^t\in\R^d$ and $\gamma>0$, we have:
\begin{align*}
    f(x^{t+1}) \leq f(x^t) - \frac{\gamma}{2} \norm{\nabla f(x^t)}^2 - \left( \frac{1}{2\gamma} - \frac{L}{2} \right) \norm{x^{t+1} - x^t}^2 + \frac{\gamma}{2}\norm{g^t - \nabla f(x^t)}^2.
\end{align*}
\end{lemma}
\begin{lemma}[Lemma $5$ of \citep{richtarik2021ef21}]
   \label{lemma:step_lemma}
   Let $a,b>0$. If $0\leq\gamma\leq\frac{1}{\sqrt{a}+b}$, then $a\gamma^2+b\gamma\leq 1$. 
   Moreover, the bound is tight up to the factor of $2$ since $\frac{1}{\sqrt{a}+b} \leq \min\left\{\frac{1}{\sqrt{a}}, \frac{1}{b}\right\} \leq \frac{2}{\sqrt{a}+b}$.
\end{lemma}

\clearpage

\section{SKETCH OF THE PROOFS}\label{appendix:proof_sketch} 
The core intuition of our work centers around reducing the variance in stochastic gradients generated by reliable clients, a goal achieved through the following application of Young's inequality:

\[
\begin{aligned}
\frac{1}{G} \sum_{i \in \mathcal{G}} \mathbb{E}\left[\left\|g_i-\nabla f(x)\right\|^2\right] & \leq 2 \frac{1}{G} \sum_{i \in \mathcal{G}} \mathbb{E}\left[\left\|g_i-\nabla f_i(x)\right\|^2\right]+2 \frac{1}{G} \sum_{i \in \mathcal{G}} \mathbb{E}\left[\left\|\nabla f_i(x)-\nabla f(x)\right\|^2\right] \\
& \leq 2 \frac{1}{G} \sum_{i \in \mathcal{G}} \mathbb{E}\left[\left\|g_i-\nabla f_i(x)\right\|^2\right]+2 \zeta^2
\end{aligned}
\]

Effectively bounding the variance of stochastic gradients \(g_i\) requires constructing them in a way such that:

\[
\frac{1}{G} \sum_{i \in \mathcal{G}} \mathbb{E}\left[\left\|g_i-\nabla f_i(x)\right\|^2\right] \underset{x \rightarrow x^*}{\rightarrow} 0 \quad \text{where} \quad x^*=\arg \min _{x \in \mathbb{R}^d} f(x)
\]

Our algorithms are designed to ensure this convergence. In a neighborhood of \(x^*\), we establish:

\[
\frac{1}{G} \sum_{i \in \mathcal{G}} \mathbb{E}\left[\left\|g_i-\nabla f(x)\right\|^2\right] \leq 2 \zeta^2
\]

To achieve this, we introduce an unconventional Byzantine-related descent lemma for the first time in the literature of optimization with Byzantine workers (Lemmas~\ref{lemma:byz_BVMNS} and 
\ref{lemma:byz_descent_dp})
And then we build on the components of the desent lemma inequality, in order to get the convergence results. Full proofs are provided in the following sections.

\clearpage

\section{MISSING PROOFS FOR \algname{Byz-VR-MARINA 2.0}}\label{appendix:proofs_marina}

For the sake of clarity, we adopt the following notation: $\overline{g}^{t} \eqdef \tfrac{1}{G}\sum_{i\in\cG}g_i^{t}$, $R^t \eqdef \norm{x^{t+1} - x^t}^2$, $H_i^t \eqdef \norm{g_i^t - \nabla f_i(x^t)}^2$, $H^t \eqdef \frac{1}{G}\sum_{i\in \cG} H_i^t$.

\subsection{Technical Lemmas}

We first prove several lemmas needed to prove the main result. The first two of them are based on the definition of a robust aggregator and are not specific to \algname{Byz-VR-MARINA 2.0} only.

\begin{lemma}[Bound on the variance]
    \label{lemma:dissimilarity}
    Suppose that Assumption \ref{as:bounded_heterogeneity} holds. Then
    \begin{eqnarray*}
        \frac{1}{G(G-1)}\sum\limits_{i,l\in \cG}\EE{\norm{g_i^{t} - g_l^{t}}^2}
        \leq 8 \EE{H^t} + 8B\EE{\norm{\nabla f(x^{t})}^{2}} + 8\zeta^2.
    \end{eqnarray*}
\end{lemma}

\begin{proof}
    Denoting $u_i \eqdef g_i^t - \nabla f_i(x^t)$ and $v_i \eqdef \nabla f_i(x^t) - \nabla f(x^t)$, we have
    \begin{align*}
        \frac{1}{G(G-1)}&\sum\limits_{i,l\in \cG}\EE{\|g_i^{t} - g_l^{t}\|^2}
        = \frac{1}{G(G-1)}\sum_{\substack{i,l\in \cG \\ i\neq l}} \EE{\|(u_i + v_i) - (u_l + v_l)\|^2}\\
        &\stackrel{\eqref{eq:young}}{\leq} \frac{4}{G(G-1)}\sum_{\substack{i,l\in \cG \\ i\neq l}} \left( \EE{\norm{u_i}^2} + \EE{\norm{v_i}^2} + \EE{\norm{u_l}^2} + \EE{\norm{v_l}^2} \right)\\
        &= \frac{8}{G}\sum\limits_{i\in \cG} \left( \EE{\norm{u_i}^2} + \EE{\norm{v_i}^2} \right) \\
        &= 8 \EE{H^t} + \frac{8}{G}\sum\limits_{i\in \cG}\EE{\norm{\nabla f_i(x^t) - \nabla f(x^t) }^2}\\
        &\stackrel{\eqref{as:bounded_heterogeneity}}{\leq} 8 \EE{H^t} + 8B\EE{\norm{\nabla f(x^{t})}^{2}} + 8\zeta^2
    \end{align*}
    as needed.
\end{proof}

\begin{lemma}[Descent Lemma]\label{lemma:byz_BVMNS}
    Suppose that Assumptions \ref{as:smoothness} and \ref{as:bounded_heterogeneity} hold. Then for all $s>0$ we have
    \begin{eqnarray*}
        \EE{f(x^{t + 1})} &\leq& \EE{f(x^t)} - \frac{\gamma\kappa}{2} \EE{\norm{\nabla f(x^{t})}^{2}}
        - \left(\frac{1}{2\gamma} - \frac{L}{2}\right)
        \EE{R^t}
        + 4\gamma c\delta (1+s) \EE{H^t}\\
        &&\qquad+ \frac{\gamma}{2}(1 + s^{-1})\EE{\norm{\overline{g}^t - \nabla f(x^t)}^2}
        + 4\gamma c\delta (1+s)\zeta^2,
    \end{eqnarray*}
    where $\kappa = 1 - 8 B c\delta (1+s)$.
\end{lemma}

\begin{proof}
    First, by Young's inequality, for any $s >0$ we have
    \begin{eqnarray*}
        \EE{\norm{g^{t} - \nabla f(x^t)}^2}
        &\stackrel{\eqref{eq:young}}{\leq}& (1 + s)\EE{\norm{g^{t} - \overline{g}^t}^2} + (1 + s^{-1})\EE{\norm{\overline{g}^t - \nabla f(x^t)}^2} \\
        &\stackrel{\eqref{eq:RAgg_def}}{\leq}& (1+s) \frac{c\delta}{G(G-1)}\sum\limits_{i,l\in \cG}\EE{\|g_i^{t} - g_l^{t}\|^2}
        + (1 + s^{-1})\EE{\norm{\overline{g}^t - \nabla f(x^t)}^2} \\
        &\stackrel{\eqref{lemma:dissimilarity}}{\leq}&  8(1+s)c\delta \EE{H^t} + 8B(1+s)c\delta \EE{\norm{\nabla f(x^{t})}^{2}} \\
        &&\qquad+ 8(1+s)c\delta \zeta^2
        + (1 + s^{-1})\EE{\norm{\overline{g}^t - \nabla f(x^t)}^2}.
    \end{eqnarray*}
    Taking expectation in Lemma \ref{lemma:page} and applying the above bound, we obtain
    \begin{eqnarray*}
        \EE{f(x^{t + 1})} &\stackrel{\eqref{lemma:page}}{\leq}& \EE{f(x^t)} - \frac{\gamma}{2}\EE{\norm{\nabla f(x^t)}^2} - \left(\frac{1}{2\gamma} - \frac{L}{2}\right)
        \EE{R^t} + \frac{\gamma}{2}\EE{\norm{g^{t} - \nabla f(x^t)}^2} \\
        &\leq& \EE{f(x^t)} - \frac{\gamma}{2}\EE{\norm{\nabla f(x^t)}^2} - \left(\frac{1}{2\gamma} - \frac{L}{2}\right)
        \EE{R^t}\\
        &&\qquad+ 4\gamma(1+s)c\delta \EE{H^t}
        + 4\gamma B(1+s)c\delta \EE{\norm{\nabla f(x^{t})}^{2}} \\
        &&\qquad+ 4\gamma (1+s)c\delta \zeta^2
        + \frac{\gamma}{2} (1 + s^{-1}) \EE{\norm{\overline{g}^t - \nabla f(x^t)}^2} \\
        &=& \EE{f(x^t)} - \parens{\frac{\gamma}{2} - 4\gamma B(1+s)c\delta} \EE{\norm{\nabla f(x^{t})}^{2}}
        - \left(\frac{1}{2\gamma} - \frac{L}{2}\right) \EE{R^t}\\
        &&\qquad+ 4\gamma(1+s)c\delta \EE{H^t}
        + 4\gamma (1+s)c\delta \zeta^2
        + \frac{\gamma}{2} (1 + s^{-1}) \EE{\norm{\overline{g}^t - \nabla f(x^t)}^2}.
    \end{eqnarray*}
    Letting $\kappa = 1 - 8 B c\delta (1+s)$, we get the result.
\end{proof}

We next provide recursive bounds on $\EE{\norm{\overline{g}^t - \nabla f(x^t)}^2}$ and $H^t = \frac{1}{G}\sum_{i\in \cG} \norm{g_i^t - \nabla f_i(x^t)}^2$.

\begin{lemma}\label{lemma:prel_1_MARINA}
Suppose that Assumptions \ref{as:smoothness}, \ref{as:hessian_variance} and \ref{as:hessian_variance_local} hold. Then
\begin{eqnarray*}
    \EE{\norm{\overline{g}^{t+1} - \nabla f(x^{t+1})}^2} \leq \parens{1-p} \parens{\EE{\norm{\overline{g}^t - \nabla f(x^t)}^2} + \parens{\frac{\omega}{G}\parens{\frac{\cL_{\pm}^2}{b}+ L_{\pm}^2 + L^2} +\frac{\cL_{\pm}^2}{Gb}}\EE{R^t}}.
\end{eqnarray*}
\end{lemma}

\begin{proof}
First, we know that
\begin{align*}
    &\EE{\norm{\overline{g}^{t+1} - \nabla f(x^{t+1})}^{2}} \\
    & \qquad = \parens{1-p} \EE{\norm{\overline{g}^{t} + \frac{1}{G}\sum \limits_{i \in \cG} \parens{\cC_i\parens{\hdelta_i\parens{x^{t+1}, x^t}} - \nabla f_i\parens{x^{t+1}}}}^2} \\
    & \qquad = \parens{1-p}\EE{\norm{\overline{g}^{t} - \nabla f\parens{x^t} + \frac{1}{G}\sum \limits_{i \in \cG} \parens{\cC_i\parens{\hdelta_i\parens{x^{t+1}, x^t}} - \parens{\nabla f_i\parens{x^{t+1}} -\nabla f_i\parens{x^t}} }}^2}.
\end{align*}
Denoting by $\EEC{\cdot}$ the expectation taken with respect to the randomness of the compressor and using the tower property of conditional expectation, we have
\begin{eqnarray*}
\EE{\cC_i\parens{\hdelta_i\parens{x^{t+1}, x^t}}}
&\overset{\eqref{eq:tower}}{=}& \EE{\EEC{\cC_i\parens{\hdelta_i\parens{x^{t+1}, x^t}}}}\\
&=& \EE{\hdelta_i\parens{x^{t+1}, x^t}}\\
&=& \EE{\nabla f_i(x^{t+1}) - \nabla f_i(x^t)}.
\end{eqnarray*}
Hence, using the independence of compressors $\{\cC_i\}_{i \in \cG}$ and estimators $\{\hdelta_i\parens{x^{t+1}, x^t}\}_{i \in \cG}$ gives
\begin{align*}
    &\EE{\norm{\overline{g}^{t+1} - \nabla f(x^{t+1})}^{2}}\\
    &\overset{\eqref{eq:vardecomp}}{=} \parens{1-p} \parens{\EE{\norm{\overline{g}^t - \nabla f(x^t)}^2} + \frac{1}{G^2}\sum\limits_{i \in \cG} \EE{\norm{\cC_i\parens{\hdelta_i\parens{x^{t+1}, x^t}} - \parens{\nabla f_i\parens{x^{t+1}} -\nabla f_i\parens{x^t}}}^2}}\\
    &\overset{\eqref{eq:vardecomp}}{=} \parens{1-p} \left(\EE{\norm{\overline{g}^t - \nabla f(x^t)}^2} \right . \\
    &\quad\left .+ \frac{1}{G^2} \sum\limits_{i \in \cG} \parens{\EE{\norm{\cC_i\parens{\hdelta_i\parens{x^{t+1}, x^t}} - \hdelta_i\parens{x^{t+1}, x^t}}^2}  +\EE{\norm{\hdelta_i\parens{x^{t+1}, x^t} - \Delta_i\parens{x^{t+1}, x^t}}^2}}\right) \\
    &\overset{\eqref{as:hessian_variance_local}}{\leq} \parens{1-p} \parens{\EE{\norm{\overline{g}^t - \nabla f(x^t)}^2} + \frac{\omega}{G^2}\sum\limits_{i \in \cG} \EE{\norm{\hdelta_i\parens{x^{t+1}, x^t}}^2} +\frac{\cL_{\pm}^2}{Gb}\EE{\norm{x^{t+1}-x^t}^2}} \\ 
    &\overset{\eqref{eq:vardecomp}}{=} \parens{1-p} \EE{\norm{\overline{g}^t - \nabla f(x^t)}^2} \\
    &\quad+ \parens{1-p} \frac{\omega}{G^2}\sum\limits_{i \in \cG} \parens{\EE{\norm{\hdelta_i\parens{x^{t+1}, x^t} - \Delta_i\parens{x^{t+1}, x^t}}^2} + \EE{\norm{\Delta_i\parens{x^{t+1}, x^t}}^2}} \\
    &\quad + \parens{1-p} \frac{\cL_{\pm}^2}{Gb}\EE{\norm{x^{t+1}-x^t}^2} \\
    &\overset{\eqref{as:hessian_variance}, \eqref{as:hessian_variance_local}}{\leq} \parens{1-p} \EE{\norm{\overline{g}^t - \nabla f(x^t)}^2} + \parens{1-p} \frac{\omega}{G}\parens{\frac{\cL_{\pm}^2}{b}+ L_{\pm}^2 + L^2} \EE{\norm{x^{t+1}-x^t}^2} \\
    &\quad +\parens{1-p}\frac{\cL_{\pm}^2}{Gb}\EE{\norm{x^{t+1}-x^t}^2} \\
    &= \parens{1-p} \parens{\EE{\norm{\overline{g}^t - \nabla f(x^t)}^2} + \parens{\frac{\omega}{G}\parens{\frac{\cL_{\pm}^2}{b}+ L_{\pm}^2 + L^2} +\frac{\cL_{\pm}^2}{Gb}}\EE{\norm{x^{t+1}-x^t}^2}}. 
\end{align*}
\end{proof}

\begin{lemma}\label{lemma:prel_2_MARINA}
Suppose that Assumptions \ref{as:smoothness}, \ref{as:hessian_variance} and \ref{as:hessian_variance_local} hold. Then
\begin{eqnarray*}
    \EE{H^{t+1}} &\leq& \parens{1-p}\left(\EE{H^{t}} 
    + \parens{\omega\parens{\frac{\cL_{\pm}^2}{b}+ L_{\pm}^2 + L^2} +\frac{\cL_{\pm}^2}{b}}\EE{R^t}\right).
\end{eqnarray*}
\end{lemma}

\begin{proof}
Following the same reasoning as in the proof of the previous lemma gives
    \begin{align*}
        \EE{H_i^{t+1}} &=\EE{\norm{g_i^{t+1} - \nabla f_i(x^{t+1})}^2} \\
        &= \parens{1-p}\EE{\norm{g_i^t + \cC_i\parens{\hdelta_i\parens{x^{t+1}, x^t}} - \nabla f_i\parens{x^{t+1}}}^2} \\
        &= \parens{1-p}\EE{\norm{g_i^t - \nabla f_i(x^t) + \cC_i\parens{\hdelta_i\parens{x^{t+1}, x^t}} - \parens{\nabla f_i\parens{x^{t+1}}-\nabla f_i\parens{x^t}}}^2} \\
        &\overset{\eqref{eq:vardecomp}}{=} \parens{1-p}\parens{\EE{\norm{g_i^t - \nabla f_i(x^t)}^2} + \EE{\norm{\cC_i\parens{\hdelta_i\parens{x^{t+1}, x^t}} - \parens{\nabla f_i\parens{x^{t+1}}-\nabla f_i\parens{x^t}}}^2}}.
        % &\overset{\eqref{as:hessian_variance}, \eqref{as:hessian_variance_local}}{\leq}& \parens{1-p}\parens{\EE{\norm{g_i^t - \nabla f_i(x^t)}^2} + \parens{\omega\parens{\frac{\cL_{\pm}^2}{b}+ L_{\pm}^2 + L^2} +\frac{\cL_{\pm}^2}{b}}\EE{\norm{x^{t+1}-x^t}^2}}
    \end{align*}
Averaging over all the good workers $i \in \cG$ and using Assumptions \ref{as:hessian_variance} and \ref{as:hessian_variance_local} gives the result.
\end{proof}

\subsection{Proof of Theorem~\ref{thm:main_result_Byz-VR-MARINA_2}}

For the readers' convenience, we repeat the statement of the theorem.

\BYZVRM*

\begin{proof}
Let $C = \frac{1+s^{-1}}{p}$, $D = \frac{8c\delta\parens{1+s}}{p}$ for some $s>0$ and define
\begin{align*}
\psi_t &= C\norm{\overline{g}^t - \nabla f(x^{t})}^2 + D H^t, \\
\Phi_t &= f(x^t) - f^* + \frac{\gamma}{2}\psi_t.
\end{align*}
Using lemmas \ref{lemma:byz_BVMNS}, \ref{lemma:prel_1_MARINA} and \ref{lemma:prel_2_MARINA}, we have
\begin{eqnarray*}
    \EE{\Phi_{t+1}}
    &=&\EE{f(x^{t+1}) - f^* + \frac{\gamma}{2}\psi_{t+1}} \\
    &=&\underbrace{\EE{f(x^{t+1}) - f^*} }_\text{use \ref{lemma:byz_BVMNS}} + \frac{\gamma}{2} \parens{C\underbrace{\EE{\norm{\overline{g}^{t+1} - \nabla f(x^{t+1})}^2}}_\text{use \ref{lemma:prel_1_MARINA}}  
    + D\underbrace{\EE{H^{t+1}}}_\text{use \ref{lemma:prel_2_MARINA}}} \\
    &\leq& \EE{f(x^{t}) - f^*} + \frac{\gamma}{2}\parens{\underbrace{C\EE{\norm{\overline{g}^t - \nabla f\parens{x^{t}}}^2} + D\EE{H^t}}_\text{$\EE{\psi_t}$}}
    - \frac{\gamma\kappa}{2}\EE{\norm{\nabla f(x^t)}^2}
    \\
    && - \parens{\frac{1}{2\gamma} - \frac{L}{2} - \frac{\gamma\eta}{2}}\EE{R^t} + 4\gamma c\delta (1+s)\zeta^2 \\
    &=& \EE{\Phi_t} - \frac{\gamma\kappa}{2} \EE{\norm{\nabla f(x^t)}^2}
    - \parens{\frac{1}{2\gamma} - \frac{L}{2} - \frac{\gamma\eta}{2}}\EE{R^t} + 4\gamma c\delta (1+s)\zeta^2,
\end{eqnarray*}
where 
\begin{align*}
    \eta \eqdef & C\parens{1-p}\parens{\frac{\omega}{G}\parens{\frac{\cL_{\pm}^2}{b}+ L_{\pm}^2 + L^2} +\frac{\cL_{\pm}^2}{Gb}} + D\parens{1-p} \parens{\omega\parens{\frac{\cL_{\pm}^2}{b}+ L_{\pm}^2 + L^2} +\frac{\cL_{\pm}^2}{b}} \\
    = & \parens{1-p}\parens{\omega\parens{\frac{\cL_{\pm}^2}{b}+ L_{\pm}^2 + L^2} +\frac{\cL_{\pm}^2}{b}}\parens{\frac{C}{G} + D} \\
    = & \frac{1-p}{p} \parens{\omega\parens{\frac{\cL_{\pm}^2}{b}+ L_{\pm}^2 + L^2} +\frac{\cL_{\pm}^2}{b}}\parens{\frac{1+s^{-1}}{G} + 8c\delta\parens{1+s}}.
\end{align*}
Taking $0 < \gamma \leq \frac{1}{L + \sqrt{\eta}}$ and using Lemma \ref{lemma:step_lemma} gives
\begin{equation*}
    \frac{1}{2\gamma} - \frac{L}{2} - \frac{\gamma\eta}{2} \geq 0
\end{equation*}
and hence
\begin{equation*}
    \EE{\Phi_{t+1}} \leq
    \EE{\Phi_t} - \frac{\gamma\kappa}{2} \EE{\norm{\nabla f(x^t)}^2}
    + 4\gamma c\delta (1+s)\zeta^2.
\end{equation*}
Now, let us take $s=\frac{1}{\sqrt{8c\delta G}}$. Recalling that $\kappa = 1 - 8 B c\delta (1+s)$, our assumption on $\delta$ implies that $\kappa>0$ and hence, summing up the above inequality for $t = 0,1,\ldots,T-1$ and rearranging the terms, we get
\begin{align*}
	\frac{1}{T} \sum \limits_{t=0}^{T} \EE{\norm{\nabla f(x^t)}^2}
    & \leq \frac{2}{\gamma\kappa T}\parens{\EE{\Phi_0} - \EE{\Phi_T}} + \frac{8c\delta(1+s)\zeta^2}{\kappa} \\
    & \leq \frac{1}{\kappa} \parens{\frac{2}{\gamma T} \EE{\Phi_0} + 8c\delta(1+s) \zeta^2} \\
    &= \frac{1}{1 - \parens{8c\delta+\sqrt{\frac{8c\delta}{G}}}B}\parens{\frac{2\parens{f(x^0) - f^*}}{\gamma T} + \parens{8c\delta + \sqrt{\frac{8c\delta}{G}}}\zeta^2}
\end{align*}
as needed.
\end{proof}

%\eduard{TODO: add the corollary for each case from the tables}

\begin{corollary}
    Let the assumptions of Theorem \ref{thm:main_result_Byz-VR-MARINA_2} hold, $p = \min\{\nicefrac{1}{\omega}, \nicefrac{b}{m}\}$ and $\gamma = \parens{L + \sqrt{\eta}}^{-1}$,
    where $\eta = \max\{\omega, \nicefrac{m}{b}\} \parens{\omega\parens{\frac{\cL_+^2}{b}+ L_{\pm}^2 + L^2} +\frac{\cL_+^2}{b}}\parens{\sqrt{\frac{1}{G}} + \sqrt{8c\delta}}^2$. Then for all $T \geq 0$, $\EE{\norm{\nabla f(\hx^T)}^2}$ is of the order
    \begin{align*}
        \cO \parens{\parens{\frac{\parens{L + \sqrt{\max\{\omega^2, \nicefrac{m\omega}{b}\} \parens{\frac{\cL_+^2}{b}+ L_{\pm}^2 + L^2}}\parens{\sqrt{\frac{1}{G}} + \sqrt{c\delta}}} \delta^0}{T \parens{1 - \parens{c\delta+\sqrt{\frac{c\delta}{G}}}B}}
        + \frac{\parens{c\delta + \sqrt{\frac{c\delta}{G}}}\zeta^2}{1 - \parens{c\delta+\sqrt{\frac{c\delta}{G}}}B}}},
    \end{align*}
    where $\delta^0 = f(x^0) - f^*$  and $\hx^T$ is chosen uniformly at random from $x^0, x^1, \ldots, x^{T-1}$. Hence, to guarantee $\EE{\norm{\nabla f(\hx^T)}^2} \leq \varepsilon^2$ for $\varepsilon^2 > \parens{1 - \parens{8c\delta+\sqrt{\nicefrac{8c\delta}{G}}}B}^{-1} {\parens{8c\delta+\sqrt{\nicefrac{8c\delta}{G}}}\zeta^2}$, \algname{Byz-VR-MARINA 2.0} requires
    \begin{align*}
        \cO \parens{\frac{1}{\varepsilon^2}\parens{1+\sqrt{\max \left\{\omega^2, \frac{m\omega}{b} \right\}}\parens{\sqrt{\frac{1}{G}} + \sqrt{c\delta}}}}
    \end{align*}    
    communication rounds.
\end{corollary}

\clearpage

\section{MISSING PROOFS FOR \algname{Byz-DASHA-PAGE}}\label{appendix:proofs_dasha}

In this section, we again use the notation $\overline{g}^{t} \eqdef \tfrac{1}{G}\sum_{i\in\cG}g_i^{t}$, $R^t \eqdef \norm{x^{t+1} - x^t}^2$, $H_i^t \eqdef \norm{g_i^t - \nabla f_i(x^t)}^2$, $H^t \eqdef \frac{1}{G}\sum_{i\in \cG} H_i^t$.

\subsection{Technical Lemmas}

\begin{lemma} [Descent Lemma]
\label{lemma:byz_descent_dp}
Suppose that Assumptions \ref{as:smoothness} and \ref{as:bounded_heterogeneity} hold. Then for all $s>0$ we have
\begin{eqnarray*}
  \EE{f(x^{t + 1})} &\leq& \EE{f(x^t)} - \frac{\gamma\kappa}{2} \EE{\norm{\nabla f(x^t)}^2} - \left(\frac{1}{2\gamma} - \frac{L}{2}\right)
  \EE{R^t}\\
  &&+ 8\gamma c\delta(1+s)\parens{\frac{1}{G}\sum \limits_{i\in \cG} \EE{\norm{g_i^t - h_i^t}^2}}
  + 8\gamma c\delta(1+s)\parens{\frac{1}{G}\sum \limits_{i\in \cG} \EE{\norm{h_i^t - \nabla f_i(x^t)}^2}} \\
  &&+ \gamma (1 + s^{-1})\EE{\norm{\overline{g}^t - h^{t}}^2} + \gamma (1 + s^{-1})\EE{\norm{h^t - \nabla f(x^t)}^2}
  + 4\gamma c\delta(1+s)\zeta^2,
\end{eqnarray*}
where $\kappa = 1 - 8 B c\delta (1+s)$.
\end{lemma}
\begin{proof}
    Using Lemma \ref{lemma:byz_BVMNS}, we get 
    \begin{eqnarray*}
        \EE{f(x^{t + 1})} &\leq& \EE{f(x^t)} - \frac{\gamma\kappa}{2} \EE{\norm{\nabla f(x^{t})}^{2}}
        - \left(\frac{1}{2\gamma} - \frac{L}{2}\right)
        \EE{R^t} \\
        &&+ 4\gamma c\delta (1+s) \parens{\frac{1}{G}\sum \limits_{i\in \cG} \EE{\norm{g_i^t - \nabla f_i(x^t)}^2}} \\
        &&+ \frac{\gamma}{2}(1 + s^{-1})\EE{\norm{\overline{g}^t - \nabla f(x^t)}^2}
        + 4\gamma c\delta (1+s)\zeta^2 \\
        &\overset{\eqref{eq:young}}{\leq}& \EE{f(x^t)} - \frac{\gamma\kappa}{2} \EE{\norm{\nabla f(x^{t})}^{2}}
        - \left(\frac{1}{2\gamma} - \frac{L}{2}\right)
        \EE{R^t} \\
        &&+ 8\gamma c\delta (1+s) \parens{\frac{1}{G}\sum \limits_{i\in \cG} \EE{\norm{g_i^t - h_i^t}^2}}
        + 8\gamma c\delta (1+s) \parens{\frac{1}{G}\sum \limits_{i\in \cG} \EE{\norm{h_i^t - \nabla f_i(x^t)}^2}} \\
        &&+ \gamma(1 + s^{-1})\EE{\norm{\overline{g}^t - h^t}^2}
        + \gamma(1 + s^{-1})\EE{\norm{h^t - \nabla f(x^t)}^2}
        + 4\gamma c\delta (1+s)\zeta^2.
    \end{eqnarray*}
\end{proof}

In what follows, we denote by $\EEC{\cdot}$ the expectation taken with respect to the randomness of the compressor, and by $\EEh{\cdot}$ the expectation taken with respect to the choice of $\{h_i\}_{i\in \cG}$.
\begin{lemma} \label{lemma: prel_prel_1}
    Suppose that Assumptions \ref{as:smoothness}, \ref{as:hessian_variance} and \ref{as:hessian_variance_local} hold. Then
        \begin{align*}
            \frac{1}{G}\sum \limits_{i \in \cG}
            \EE{\norm{h^{t+1}_i - h^{t}_i}^2} \leq \parens{\frac{(1 - p) \cL_{\pm}^{2}}{ b} + 2 \parens{L_{\pm}^2 + L^{2}}} \EE{R^t} + 2 p \parens{\frac{1}{G} \sum \limits_{i \in \cG}\EE{\norm{h^{t}_i - \nabla f_i(x^{t})}^2}}.
        \end{align*}
\end{lemma}
\begin{proof}
    First, we have
    \begin{eqnarray*}
        \EEh{\norm{h^{t+1}_i - h^{t}_i}^2}
        &=&p\norm{\nabla f_i(x^{t+1}) - h^{t}_i}^2 + (1 - p) \EEh{\norm{ \hdelta_{i}(x^{t+1}, x^{t})}^2} \\
        &\overset{\eqref{eq:vardecomp}}{=}& p\norm{\nabla f_i(x^{t+1}) - h^{t}_i}^2
        + (1 - p) \EEh{\norm{\hdelta_{i}(x^{t+1}, x^{t}) - \Delta_{i}(x^{t+1},x^{t})}^2}\\
        && + (1 - p) \norm{\nabla f_{i}(x^{t+1}) - \nabla f_{i}(x^{t})}^2\\
        &\overset{\eqref{eq:young}}{\leq}& 2p\norm{\nabla f_{i}(x^{t+1}) - \nabla f_{i}(x^{t})}^2 + 2p \norm{h_i^t - \nabla f_{i}(x^{t})}^2 \\
        && + (1 - p) \EEh{\norm{\hdelta_{i}(x^{t+1}, x^{t}) - \Delta_{i}(x^{t+1},x^{t})}^2} \\
        && + (1 - p) \norm{\nabla f_{i}(x^{t+1}) - \nabla f_{i}(x^{t})}^2\\
        % &=&  2p \norm{h_i^t - \nabla f_{i}(x^{t})}^2 + (1 - p) \EEh{\norm{\hdelta_{i}(x^{t+1}, x^{t}) - \Delta_{i}(x^{t+1},x^{t})}^2} \\
        % &&  + (1 + p) \norm{\nabla f_{i}(x^{t+1}) - \nabla f_{i}(x^{t})}^2\\
        &\leq& 2p \norm{h_i^t - \nabla f_{i}(x^{t})}^2 + (1 - p) \EEh{\norm{\hdelta_{i}(x^{t+1}, x^{t}) - \Delta_{i}(x^{t+1},x^{t})}^2} \\
        &&  + 2 \norm{\nabla f_{i}(x^{t+1}) - \nabla f_{i}(x^{t})}^2.
    \end{eqnarray*}
    Taking full expectation and averaging over all the good workers, we get
    \begin{eqnarray*}
        \frac{1}{G}\sum \limits_{i \in \cG}\EE{\norm{h^{t+1}_i - h^{t}_i}^2}
        &\leq& (1 - p) \frac{1}{G}\sum \limits_{i \in \cG}\EE{\norm{\hdelta_{i}(x^{t+1}, x^{t}) - \Delta_{i}(x^{t+1},x^{t})}^2}\\
        &&+ \frac{2}{G}\sum \limits_{i \in \cG}\EE{\norm{\nabla f_{i}(x^{t+1}) - \nabla f_{i}(x^{t})}}^2
        + \frac{2p}{G}\sum \limits_{i \in \cG}\EE{\norm{h_i^t - \nabla f_{i}(x^{t})}}^2.
    \end{eqnarray*}
    Using Assumptions \ref{as:smoothness}, \ref{as:hessian_variance} and \ref{as:hessian_variance_local}, we can conclude that
        \begin{align*}
            \frac{1}{G}\sum \limits_{i \in \cG}
            \EE{\norm{h^{t+1}_i - h^{t}_i}^2} \leq \parens{\frac{(1 - p) \cL_{\pm}^{2}}{ b} + 2 \parens{L_{\pm}^2 + L^{2}}} \EE{R^t} + 2 p \parens{\frac{1}{G} \sum \limits_{i \in \cG}\EE{\norm{h^{t}_i - \nabla f_i(x^{t})}^2}}.
        \end{align*}
\end{proof}

\begin{lemma}
    \label{lemma: pre1_dp}
    Suppose that Assumptions \ref{as:smoothness}, \ref{as:hessian_variance} and \ref{as:hessian_variance_local} hold. Then
\begin{align*}
    \EE{\norm{\overline{g}^{t+1} - h^{t+1}}^2}
    &\leq \left(1 - a\right)^2\EE{\norm{\overline{g}^t - h^{t}}^2}
    + \frac{2\omega}{G} \parens{\frac{(1 - p) \cL_{\pm}^{2}}{ b} + 2 \parens{L_{\pm}^2 + L^{2}}} \EE{R^t} \\  
    &\quad + \frac{4p\omega}{G}\parens{\frac{1}{G} \sum \limits_{i \in \cG}\EE{\norm{h^{t}_i - \nabla f_i(x^{t})}^2}}
    + \frac{2 a^2 \omega}{G}\parens{\frac{1}{G} \sum \limits_{i \in \cG}\EE{\norm{g_i^t - h_i^{t}}^2}}.
\end{align*}
\end{lemma}
\begin{proof}
    First, we have
    \begin{align*}
        &\EEC{\norm{\overline{g}^{t+1} - h^{t+1}}^2}\\
        &\quad=\EEC{\norm{\overline{g}^t + \frac{1}{G}\sum\limits_{i \in \cG} \cC_i\left(h_i^{t+1} - h_i^{t} - a \left(g_i^t - h_i^{t}\right)\right) - h^{t+1}}^2} \\
        &\quad=\EEC{\norm{\frac{1}{G}\sum\limits_{i \in \cG} \cC_i\left(h_i^{t+1} - h_i^{t} - a \left(g_i^t - h_i^{t}\right)\right) - \frac{1}{G}\sum\limits_{i \in \cG} \left( h_i^{t+1} - h_i^{t} - a \left(g_i^t - h_i^{t}\right)\right)}^2} \\
        &\qquad+ \left(1 - a\right)^2\norm{\overline{g}^t - h^{t}}^2.
    \end{align*}
Using the independence of compressors, we get
\begin{eqnarray*}
    \EEC{\norm{\overline{g}^{t+1} - h^{t+1}}^2}
    &=&\frac{1}{G^2} \sum \limits_{i \in \cG}\EEC{\norm{\cC_i\left(h_i^{t+1} - h_i^{t} - a \left(g_i^t - h_i^{t}\right)\right) - \left(h_i^{t+1} - h_i^{t} - a \left(g_i^t - h_i^{t}\right)\right)}^2} \\
    &&\quad + \left(1 - a\right)^2\norm{\overline{g}^t - h^{t}}^2\\
    &\leq& \frac{\omega}{G^2} \sum \limits_{i \in \cG}\norm{h_i^{t+1} - h_i^{t} - a \left(g_i^t - h_i^{t}\right)}^2 + \left(1 - a\right)^2\norm{\overline{g}^t - h^{t}}^2 \\
    &\overset{\eqref{eq:young}}{\leq}& \frac{2\omega}{G^2} \sum \limits_{i \in \cG}\norm{h_i^{t+1} - h_i^{t}}^2 + \frac{2 a^2 \omega}{G^2} \sum \limits_{i \in \cG}\norm{g_i^t - h_i^{t}}^2 + \left(1 - a\right)^2\norm{\overline{g}^t - h^{t}}^2.
\end{eqnarray*}
Taking full expectation and using Lemma \ref{lemma: prel_prel_1} gives
\begin{eqnarray*}
    \EE{\norm{\overline{g}^{t+1} - h^{t+1}}^2} &\leq & \frac{2\omega}{G} \parens{\frac{(1 - p) \cL_{\pm}^{2}}{ b} + 2 \parens{L_{\pm}^2 + L^{2}}} \EE{R^t}\\  
    &&+ \frac{4p\omega}{G}\parens{\frac{1}{G} \sum \limits_{i \in \cG}\EE{\norm{h^{t}_i - \nabla f_i(x^{t})}^2}} \\  
    &&+ \frac{2 a^2 \omega}{G}\parens{\frac{1}{G} \sum \limits_{i \in \cG}\EE{\norm{g_i^t - h_i^{t}}^2}} 
    + \left(1 - a\right)^2\EE{\norm{\overline{g}^t - h^{t}}^2}
\end{eqnarray*}
as needed.
\end{proof}

\begin{lemma}
    \label{lemma: pre2_dp}
    Suppose that Assumptions \ref{as:smoothness}, \ref{as:hessian_variance} and \ref{as:hessian_variance_local} hold. Then
    \begin{align*}
    \frac{1}{G}\sum\limits_{i\in \cG}\EE{\norm{g^{t+1}_i - h_i^{t+1}}^2}
    &\leq \parens{2 a^2 \omega + \parens{1 - a}^2} \frac{1}{G}\sum\limits_{i\in \cG} \EE{\norm{g_i^t - h_i^{t}}^2}
    + 4p\omega \frac{1}{G}\sum\limits_{i\in \cG} \EE{\norm{h^{t}_i - \nabla f_i(x^{t})}^2} \\
    &\quad+ 2\omega \parens{\frac{(1 - p) \cL_{\pm}^{2}}{ b} + 2 \parens{L_{\pm}^2 + L^{2}}}\EE{R^t}.
    \end{align*}
\end{lemma}
\begin{proof}
    First, we use the definition of compression operators to bound $\EEC{\norm{g^{t+1}_i - h_i^{t+1}}^2}$:
    \begin{eqnarray*}
        \EEC{\norm{g^{t+1}_i - h_i^{t+1}}^2}
        &=& \EEC{\norm{g_i^t + \cC_i\left(h_i^{t+1} - h_i^{t} - a \left(g_i^t - h_i^{t}\right)\right) - h_i^{t+1}}^2} \\
        &\overset{\eqref{eq:vardecomp}}{\leq}& \EEC{\norm{\cC_i\left(h_i^{t+1} - h_i^{t} - a \left(g_i^t - h_i^{t}\right)\right) - \left( h_i^{t+1} - h_i^{t} - a \left(g_i^t - h_i^{t}\right)\right)}^2} \\
        && + \left(1 - a\right)^2\norm{g_i^t - h_i^{t}}^2 \\
        &\leq& \omega \norm{h_i^{t+1} - h_i^{t} - a \left(g_i^t - h_i^{t}\right)}^2 + \left(1 - a\right)^2\norm{g_i^t - h_i^{t}}^2 \\
        &\overset{\eqref{eq:young}}{\leq}& 2 \omega \norm{h_i^{t+1} - h_i^{t}}^2 + 2 a^2 \omega \norm{g_i^t - h_i^{t}}^2 + \left(1 - a\right)^2\norm{g_i^t - h_i^{t}}^2 \\
        &=& 2 \omega \norm{h_i^{t+1} - h_i^{t}}^2 + \left(2 a^2 \omega + \left(1 - a\right)^2\right) \norm{g_i^t - h_i^{t}}^2.
    \end{eqnarray*}
    Taking full expectation and averaging over all $i \in \cG$, we get
    \begin{align*}
        \frac{1}{G}\sum \limits_{i\in\cG}\EE{\norm{g^{t+1}_i - h_i^{t+1}}^2}
        \leq 2 \omega \frac{1}{G}\sum \limits_{i\in\cG} \EE{\norm{h_i^{t+1} - h_i^{t}}^2} + \left(2 a^2 \omega + \left(1 - a\right)^2\right) \frac{1}{G}\sum \limits_{i\in\cG}\EE{\norm{g_i^t - h_i^{t}}^2}.
    \end{align*}
    Using Lemma \ref{lemma: prel_prel_1}, we can conclude that
    \begin{align*}
        \frac{1}{G}\sum\limits_{i\in \cG}\EE{\norm{g^{t+1}_i - h_i^{t+1}}^2}
        &\leq
        2\omega \parens{\frac{(1 - p) \cL_{\pm}^{2}}{ b} + 2 \parens{L_{\pm}^2 + L^{2}}}\EE{R^t} + 4p\omega \frac{1}{G}\sum\limits_{i\in \cG} \EE{\norm{h^{t}_i - \nabla f_i(x^{t})}^2}\\
        &\quad+ \parens{2 a^2 \omega + \parens{1 - a}^2} \frac{1}{G}\sum\limits_{i\in \cG} \EE{\norm{g_i^t - h_i^{t}}^2}.
    \end{align*}
\end{proof}

\begin{lemma}
    \label{lemma:pre3_dp}
    Suppose that Assumption \ref{as:hessian_variance_local}  holds. Then
    \begin{align*}
        \EE{\norm{h^{t+1} - \nabla f(x^{t+1})}^2}
        \leq \parens{1 - p} \EE{\norm{h^{t} - \nabla f(x^{t})}^2}
        + \frac{\parens{1 - p}\cL_{\pm}^2}{G b}\EE{R^t}.
    \end{align*}
\end{lemma}
\begin{proof}
    Using the definition of $h^{t+1}$, we obtain
    \begin{align*}
        &\EEh{\norm{h^{t+1} - \nabla f(x^{t+1})}^2} \\
        &\quad=\parens{1 - p}\EEh{\norm{h^t + \frac{1}{G}\sum \limits_{i \in \cG}\hdelta_{i}(x^{t+1}, x^{t}) - \nabla f(x^{t+1})}^2} \\
        &\quad=\parens{1 - p}\EEh{\norm{\parens{h^t - \nabla f(x^t)} + \frac{1}{G}\sum \limits_{i \in \cG} \parens{\hdelta_{i}(x^{t+1}, x^{t}) - \Delta_i(x^{t+1}, x^{t})}}^2} \\
        &\quad\overset{\eqref{eq:vardecomp}}{=} \parens{1 - p}\EEh{\norm{\frac{1}{G}\sum \limits_{i \in \cG} \parens{\hdelta_{i}(x^{t+1}, x^{t}) - \Delta_i(x^{t+1}, x^{t})}}^2}
        + \parens{1 - p} \norm{h^{t} - \nabla f(x^{t})}^2.
    \end{align*}
    From the unbiasedness and independence of mini-batch estimators, we get
    \begin{align*}
        &\EEh{\norm{h^{t+1} - \nabla f(x^{t+1})}^2} \\
        &\quad\leq\frac{\parens{1 - p}}{G^{2}}\sum \limits_{i \in \cG}\EEh{\norm{\hdelta_{i}(x^{t+1}, x^{t}) - \Delta_i(x^{t+1}, x^{t})}^2} + \parens{1 - p} \norm{h^{t} - \nabla f(x^{t})}^2 \\
        &\quad\overset{\eqref{as:hessian_variance_local}}{\leq}
        \frac{\parens{1 - p}\cL_{\pm}^2}{G b}\norm{x^{t+1} - x^t}^2 + \parens{1 - p} \norm{h^{t} - \nabla f(x^{t})}^2.
    \end{align*}
    Taking full expectation gives the result.
\end{proof}

\begin{lemma}\label{lemma:pre4_dp}
    Suppose that Assumption \ref{as:hessian_variance_local} holds. Then
    \begin{align*}
        \frac{1}{G}\sum\limits_{i\in \cG}\EE{\norm{h^{t+1}_i - \nabla f_i(x^{t+1})}^2}
        \leq \parens{1 - p}  \frac{1}{G}\sum\limits_{i\in \cG}\EE{\norm{h^{t}_i - \nabla f_i(x^{t})}^2}
        + \frac{\parens{1 - p}\cL_{\pm}^2}{ b}\EE{R^t}.
    \end{align*}
\end{lemma}
\begin{proof}
     Using the same reasoning as in the proof of the previous lemma, we have
    \begin{align*}
        &\EEh{\norm{h^{t+1}_i - \nabla f_i(x^{t+1})}^2} \\
        &\quad=\parens{1 - p}\EEh{\norm{h^t_i + \hdelta_i(x^{t+1}, x^t) - \nabla f_i(x^{t+1})}^2} \\
        &\quad=\parens{1 - p}\EEh{\norm{h^t_i -\nabla f_i(x^{t}) + \hdelta_i(x^{t+1}, x^t) - \Delta_i(x^{t+1}, x^t)}^2} \\
        &\quad\overset{\eqref{eq:vardecomp}}{=}\parens{1 - p}\EEh{\norm{\hdelta_i(x^{t+1}, x^t) - \Delta_i(x^{t+1}, x^t)}^2} + \parens{1 - p} \norm{h^{t}_i - \nabla f_i(x^{t})}^2.
    \end{align*}
Taking full expectation and calculating the average over all $i \in \cG$, we get
    \begin{align*}
        &\frac{1}{G}\sum\limits_{i\in \cG}\EE{\norm{h^{t+1}_i - \nabla f_i(x^{t+1})}^2}\\
        &\quad= \parens{1 - p}\frac{1}{G}\sum\limits_{i\in \cG}\EE{\norm{\hdelta_i(x^{t+1}, x^t) - \Delta_i(x^{t+1}, x^t)}^2} + \parens{1 - p}  \frac{1}{G}\sum\limits_{i\in \cG}\EE{\norm{h^{t}_i - \nabla f_i(x^{t})}^2} \\
        &\quad\overset{\eqref{as:hessian_variance_local}}{\leq} \frac{\parens{1 - p}\cL_{\pm}^2}{b}\EE{\norm{x^{t+1}-x^t}^2} + \parens{1 - p}  \frac{1}{G}\sum\limits_{i\in \cG}\EE{\norm{h^{t}_i - \nabla f_i(x^{t})}^2}.
    \end{align*}
\end{proof}

The proof of the main result for \algname{Byz-DASHA-PAGE} will require us to solve the following system of equations:

\begin{lemma}[Calculations]\label{lemma:calc}
Assuming that $0<a<\frac{2}{2\omega + 1}$ and $C, D, E, F \geq 0$, the solution of the system of equations
\begin{align*}
\begin{cases}
    C = (1-a)^{2}C + 2 \parens{1+s^{-1}}\\
    D = \parens{2a^{2}\omega + \parens{1-a}^{2}}D + \frac{2a^{2}\omega}{G}C + 16c\delta\parens{1+s} \\
    E = \parens{1-p}E + 2\parens{1+s^{-1}}\\
    F = \parens{1-p}F + 4p\omega \parens{\frac{C}{G}+D} + 16c\delta\parens{1+s}
\end{cases}
\end{align*}
is
\begin{align}\label{eq:sys_bdp_sol}
\begin{cases}
    C = 2\frac{1+s^{-1}}{1-\parens{1-a}^2}\\
    D = \frac{1}{1-2\omega a^2 - (1-a)^2} \parens{16 \parens{1+s}c\delta + \frac{2\omega a^2}{G}C} \\
    E = 2\frac{1+s^{-1}}{p} \\
    F = 4\omega \parens{\frac{C}{G} + D} + \frac{16 c\delta}{p}\parens{1 + s}.
\end{cases}
\end{align}
Moreover,
\begin{align*}
\frac{C}{G} + D &= \frac{2}{1-2\omega a^2 - (1-a)^2}\parens{\frac{1+s^{-1}}{G} + \parens{1+s}8c\delta}\\
\frac{E}{G} + F 
& = \parens{\frac{8\omega}{1-2\omega a^2 - (1-a)^2} + \frac{2}{p}}\parens{\frac{1+s^{-1}}{G} + \parens{1+s}8c\delta}.
\end{align*}
\end{lemma}
\begin{proof}
The fact that that this choice of $C$, $D$, $E$, $F$ satisfies the equations can easily be verified by direct substitution of the values in \eqref{eq:sys_bdp_sol}.
% \begin{align*}
%     \frac{C}{G} + D &= \frac{C}{G} + \frac{1}{1-2\omega a^2 - (1-a)^2} \parens{16\parens{1+s}c\delta + C\frac{2\omega a^2}{G}}\\
%     & = \frac{1}{1-2\omega a^2 - (1-a)^2}\parens{\frac{C}{G}\parens{1- 2\omega a^2 - (1-a)^2} + 16\parens{1+s}c\delta + C\frac{2\omega a^2}{G}}\\
%     & = \frac{1}{1-2\omega a^2 - (1-a)^2}\parens{\frac{C}{G}\parens{1 - (1-a)^2} + 16\parens{1+s}c\delta}\\
%     & = \frac{2}{1-2\omega a^2 - (1-a)^2}\parens{\frac{1+s^{-1}}{G} + \parens{1+s}8c\delta}\\
%     \frac{E}{G} + F &= 2\frac{1+s^{-1}}{pG} + 4\omega \parens{C+\frac{D}{G}} + 16\frac{c\delta}{p}\parens{1+s}\\
%     &= \frac{2}{p}\parens{\frac{1+s^{-1}}{G} + 8c\delta\parens{1+s}} + 4\omega \parens{C+\frac{D}{G}}\\
%     & = \parens{\frac{8\omega}{1-2\omega a^2 - (1-a)^2} + \frac{2}{p}}\parens{\frac{1+s^{-1}}{G} + \parens{1+s}8c\delta}
% \end{align*}
% as needed.
\end{proof}

\subsection{Proof of Theorem~\ref{thm:main_result_dp}}

For the readers' convenience, we repeat the statement of the theorem.
\BYZDP*
\begin{proof}
Let
\begin{align*}
    C &= \frac{2(1+s^{-1})}{1-\parens{1-a}^2}\\
    D &= \frac{1}{1-2\omega a^2 - (1-a)^2} \parens{16 \parens{1+s}c\delta + \frac{2\omega a^2}{G}C} \\
    E &= \frac{2}{p}(1+s^{-1}) \\
    F &= 4\omega \parens{\frac{C}{G} + D} + \frac{16 c\delta}{p}\parens{1 + s}
\end{align*}
and
\begin{align*}
    \psi_t &= C\norm{\overline{g}^t - h^{t}}^2 + D\parens{\frac{1}{G}\sum \limits_{i \in \cG}\norm{g_i^t - h_i^t}^2} + E \norm{h^{t}-\nabla f(x^{t})}^{2} + F\parens{\frac{1}{G}\sum \limits_{i \in \cG}\norm{h_i^t - \nabla f_i(x^t)}^2}, \\
    \Phi_t &= f(x^t) - f^* + \frac{\gamma}{2}\psi_t.
\end{align*}
Using Lemmas \ref{lemma:byz_descent_dp}, \ref{lemma: pre1_dp}, \ref{lemma: pre2_dp}, \ref{lemma:pre3_dp} and \ref{lemma:pre4_dp}, we have
\begin{align*}
    &\EE{\Phi_{t+1}}\\
    &=\EE{f(x^{t+1}) - f^* + \frac{\gamma}{2}\psi_{t+1}}\\
    &=\underbrace{\EE{f(x^{t+1}) - f^*} }_\text{use \ref{lemma:byz_descent_dp}} + \frac{\gamma}{2}C\underbrace{\EE{\norm{\overline{g}^{t+1} - h^{t+1}}^2}}_\text{use \ref{lemma: pre1_dp}}
    + \frac{\gamma}{2}D\underbrace{\parens{\frac{1}{G}\sum \limits_{i \in \cG}\EE{\norm{g_i^{t+1} - h_i^{t+1}}^2}}}_\text{use \ref{lemma: pre2_dp}}  \\
    & + \frac{\gamma}{2}E\underbrace{\EE{\norm{h^{t+1} - \nabla f(x^{t+1})}^2}}_\text{use \ref{lemma:pre3_dp}}  
    + \frac{\gamma}{2}F\underbrace{\parens{\frac{1}{G}\sum \limits_{i \in \cG}\EE{\norm{h_i^{t+1} - \nabla f_i(x^{t+1})}^2}}}_\text{use \ref{lemma:pre4_dp}}   \\
    &\leq \EE{f(x^{t}) - f^*} -\frac{\gamma\kappa}{2} \EE{\norm{\nabla f(x^t)}^2} 
    - \parens{\frac{1}{2\gamma} - \frac{L}{2} - \frac{\gamma\eta}{2}}\EE{R^t} + 4\gamma c\delta(1+s)\zeta^2\\ 
    & + \frac{\gamma}{2}\parens{\underbrace{C\EE{\norm{\overline{g}_t - h^{t}}^2} + D\parens{\frac{1}{G}\sum \limits_{i \in \cG}\EE{\norm{g_i^t - h_i^t}^2}} + E \EE{\norm{h^{t}-\nabla f(x^{t})}^{2}} + F\parens{\frac{1}{G}\sum \limits_{i \in \cG}\EE{\norm{h_i^t - \nabla f_i(x^t)}^2}}}_\text{$\EE{\psi_t}$}} \\
    &= \EE{f(x^{t}) - f^*} + \frac{\gamma}{2}\EE{\psi_t} - \frac{\gamma\kappa}{2}\EE{\norm{\nabla f(x^t)}^2} - \parens{\frac{1}{2\gamma} - \frac{L}{2} - \frac{\gamma\eta}{2}}\EE{R^t} + 4\gamma c\delta(1+s)\zeta^2,
\end{align*}
where the last inequality follows from Lemma \ref{lemma:calc} and
\begin{align*}
    \eta \eqdef &2\omega \parens{\frac{1-p}{b}\cL_{\pm}^{2} + 2 \parens{L_{\pm}^2 + L^2}}\parens{\frac{C}{G} + D} + \frac{1-p}{b}\cL_{\pm}^{2}\parens{\frac{E}{G} + F} \\
    % = &4\omega \parens{L_{\pm}^2 + L^2}\parens{\frac{C}{G}+D} + \frac{1-p}{b}\cL_{\pm}^{2}\parens{2\omega \parens{\frac{C}{G} + D} + \parens{\frac{E}{G}+F}} \\
    % = &\frac{8\omega \parens{L_{\pm}^2 + L^2}}{1-2\omega a^2 - (1-a)^2}\parens{\frac{1+\alpha^{-1}}{G} + \parens{1+\alpha}8c\delta} \\
    %  &+ \frac{1-p}{b}\cL_{\pm}^{2} \parens{\frac{12\omega}{1-2\omega a^2 - (1-a)^2} + \frac{2}{p}} \parens{\frac{1+\alpha^{-1}}{G} + \parens{1+\alpha}8c\delta}\\
    = & \parens{\frac{8\omega \parens{L_{\pm}^2 + L^2}}{1-2\omega a^2 - (1-a)^2} + \frac{1-p}{b}\cL_{\pm}^{2} \parens{\frac{12\omega}{1-2\omega a^2 - (1-a)^2} + \frac{2}{p}}}\parens{\frac{1+s^{-1}}{G} + \parens{1+s}8c\delta}.
\end{align*}
Taking $0 < \gamma \leq \frac{1}{L + \sqrt{\eta}}$ and using Lemma \ref{lemma:step_lemma} gives
\begin{equation*}
    \frac{1}{2\gamma} - \frac{L}{2} - \frac{\gamma\eta}{2} \geq 0
\end{equation*}
and hence
\begin{eqnarray*}
    \EE{\Phi_{t+1}}
    \leq \EE{\Phi_{t}} - \frac{\gamma\kappa}{2}\EE{\norm{\nabla f(x^t)}^2} 
    + 4\gamma c\delta(1+s)\zeta^2.
\end{eqnarray*}
Therefore, summing up the above inequality for $t = 0,1,\ldots,T-1$ and rearranging the terms, we get
\begin{align*}
	\frac{1}{T} \sum \limits_{t=0}^{T} \EE{\norm{\nabla f(x^t)}^2}
    &\leq \frac{1}{\kappa}\parens{\frac{2}{\gamma T}\parens{\EE{\Phi_0} - \EE{\Phi_T}} +8c\delta(1+s)\zeta^2} \\
    &\leq \frac{1}{\kappa}\parens{\frac{2}{\gamma T}\parens{\EE{\Phi_0}} + 8c\delta(1+s) \zeta^2} \\
    &= \frac{1}{\kappa}\parens{\frac{2\Phi_0 }{\gamma T}+ 8c\delta(1+s)\zeta^2}.
\end{align*}
Letting $s = \sqrt{\frac{1}{8c\delta G}}$ gives the result.
\end{proof}

%\eduard{TODO: add the corollary for each case from the tables}

\begin{corollary}
    Let the assumptions of Theorem \ref{thm:main_result_dp} hold, $p = \nicefrac{b}{m}$ and $\gamma = \parens{L + \sqrt{\eta}}^{-1}$,
    where $\eta = \parens{8\omega(2\omega +1) \parens{L_{\pm}^2 + L^2} + \frac{1}{b} \parens{12\omega\parens{2\omega+1} + \frac{2}{p}}\cL_{\pm}^{2}}\parens{\sqrt{\frac{1}{G}} + \sqrt{8c\delta}}^2$. Then for all $T \geq 0$, $\EE{\norm{\nabla f(\hx^T)}^2}$ is of the order
    \begin{align*}
        \cO \parens{\parens{\frac{\parens{L + \sqrt{\omega(\omega +1) \parens{L_{\pm}^2 + L^2} + \parens{\frac{\omega\parens{\omega+1}}{b} + \frac{m}{b^2}} \cL_{\pm}^{2}}\parens{\sqrt{\frac{1}{G}} + \sqrt{c\delta}}} \delta^0}{T \parens{1 - \parens{c\delta+\sqrt{\frac{c\delta}{G}}}B}}
        + \frac{\parens{c\delta + \sqrt{\frac{c\delta}{G}}}\zeta^2}{1 - \parens{c\delta+\sqrt{\frac{c\delta}{G}}}B}}},
    \end{align*}
    where $\delta^0 = f(x^0) - f^*$  and $\hx^T$ is chosen uniformly at random from $x^0, x^1, \ldots, x^{T-1}$. Hence, to guarantee $\EE{\norm{\nabla f(\hx^T)}^2} \leq \varepsilon^2$ for $\varepsilon^2 > \parens{1 - \parens{8c\delta+\sqrt{\nicefrac{8c\delta}{G}}}B}^{-1} {\parens{8c\delta+\sqrt{\nicefrac{8c\delta}{G}}}\zeta^2}$, \algname{Byz-DASHA-PAGE} requires
    \begin{align*}
        \cO \parens{\frac{1}{\varepsilon^2}\parens{1+\parens{\omega + \frac{\sqrt{m}}{b}}\parens{\sqrt{\frac{1}{G}} + \sqrt{c\delta}}}}
    \end{align*}    
    communication rounds.
\end{corollary}

\clearpage

\section{MISSING PROOFS FOR \algname{Byz-EF21} AND \algname{Byz-EF21-BC}}

For the sake of clarity, we again adopt the notation $R^t \eqdef \norm{x^{t+1} - x^t}^2$, $H_i^t \eqdef \norm{g_i^t - \nabla f_i(x^t)}^2$, $H^t \eqdef \frac{1}{G}\sum_{i\in \cG} H_i^t$, and additionally denote $G^t \eqdef \norm{x^t-w^t}^2$.

\subsection{Technical Lemmas}

\begin{lemma}[Descent Lemma]\label{lemma:descent_byz_ef21}
    Suppose that Assumptions \ref{as:smoothness} and \ref{as:bounded_heterogeneity} hold. Then
    \begin{eqnarray*}
        \EE{f(x^{t + 1})} &\leq& \EE{f(x^t)} - \frac{\gamma\kappa}{2} \EE{\norm{\nabla f(x^{t})}^{2}}
        - \left(\frac{1}{2\gamma} - \frac{L}{2}\right)
        \EE{R^t}\\
        &&\qquad+ \frac{\gamma}{2} \parens{1 + \sqrt{8c\delta}}^2 \EE{H^t}
        + 4\gamma c\delta \left(1+\frac{1}{\sqrt{8c\delta}} \right)\zeta^2,
    \end{eqnarray*}
    where $\kappa = 1 - 8 B c\delta \left(1+\frac{1}{\sqrt{8c\delta}} \right)$.
\end{lemma}
\begin{proof}

    First, from Lemma \ref{lemma:byz_BVMNS}, for any $s>0$ we have
    \begin{eqnarray*}
        \EE{f(x^{t + 1})} &\leq& \EE{f(x^t)} - \frac{\gamma\kappa}{2} \EE{\norm{\nabla f(x^{t})}^{2}}
        - \left(\frac{1}{2\gamma} - \frac{L}{2}\right)
        \EE{R^t}
        + 4\gamma c\delta (1+s) \EE{H^t}\\
        &&\qquad+ \frac{\gamma}{2}(1 + s^{-1})\EE{\norm{\overline{g}^t - \nabla f(x^t)}^2}
        + 4\gamma c\delta (1+s)\zeta^2.
    \end{eqnarray*}
    Moreover, since by Jensen's inequality
    \begin{eqnarray*}
        \EE{\norm{\overline{g}^t - \nabla f(x^t)}^2} &=& \EE{\norm{\frac{1}{G}\sum_{i\in \cG} \parens{g_i^t - \nabla f_i(x^t)}}^2} \\
        &\stackrel{\eqref{eq:jensen}}{\leq}& \frac{1}{G}\sum \limits_{i\in \cG} \EE{\norm{g_i^t - \nabla f_i(x^t)}^2}
        = \EE{H^t},
    \end{eqnarray*}
    we get
    \begin{eqnarray*}
        \EE{f(x^{t + 1})}
        &\leq& \EE{f(x^t)} - \frac{\gamma\kappa}{2} \EE{\norm{\nabla f(x^{t})}^{2}}
        - \left(\frac{1}{2\gamma} - \frac{L}{2}\right) \EE{R^t}\\
        &&\qquad+ \frac{\gamma}{2} \parens{8(1+s)c\delta + 1 + s^{-1}} \EE{H^t}
        + 4\gamma c\delta (1+s)\zeta^2.
    \end{eqnarray*}
    Choosing $s = \frac{1}{\sqrt{8c\delta}}$ to minimize $8(1+s)c\delta + 1 + s^{-1}$ gives the result.
\end{proof}

\begin{lemma}\label{lemma:gf_rec_byzef21bc}
    Let Assumptions \ref{as:smoothness} and \ref{as:hessian_variance} hold. Then
    \begin{eqnarray}\label{eq:gf_rec_byzef21bc}
        \EE{H^{t+1}}
        &\leq& \left(1-\frac{\alpha_D}{4}\right) \EE{H^t}
        + \frac{8}{\alpha_D} \parens{L_{\pm}^2 + L^2} \EE{R^t} + \frac{10}{\alpha_D} \parens{L_{\pm}^2 + L^2} \EE{G^{t+1}}.
    \end{eqnarray}
\end{lemma}
\begin{proof}
    First, the update rule of the algorithm implies that
    \begin{eqnarray}
    \label{eq:gkef21pd}
        \EE{H^{t+1}}
        &=& \frac{1}{G}\sum\limits_{i\in \cG} \EE{\norm{g_i^{t+1} - \nabla f_i(x^{t+1})}^2} \nonumber \\
        &=& \frac{1}{G}\sum\limits_{i\in \cG} \EE{\norm{ g_i^t + \mathcal{C}_i^D(\nabla f_i(w^{t+1}) - g_i^t) - \nabla f_i(x^{t+1})}^2} \nonumber \\
        &\overset{\eqref{eq:young}}{\leq}& \left(1+\frac{\alpha_D}{2}\right) \frac{1}{G}\sum\limits_{i\in \cG} \EE{\norm{ \mathcal{C}_i^D(\nabla f_i(w^{t+1}) - g_i^t) - (\nabla f_i(w^{t+1}) - g_i^t)}^2} \nonumber \\
        && + \left(1+\frac{2}{\alpha_D}\right) \frac{1}{G}\sum\limits_{i\in \cG} \EE{\norm{ \nabla f_i(w^{t+1}) - \nabla f_i(x^{t+1})}^2} \nonumber \\
        &\overset{\eqref{as:hessian_variance}}{\leq}& \left(1+\frac{\alpha_D}{2}\right)\left(1-\alpha_D\right) \frac{1}{G}\sum\limits_{i\in \cG} \EE{\norm{g_i^t - \nabla f_i(w^{t+1})}^2} \nonumber \\
        && + \left(1+\frac{2}{\alpha_D}\right) \parens{L_{\pm}^2 + L^2} \EE{\norm{w^{t+1} - x^{t+1}}^{2}} \nonumber \\
        &\overset{\eqref{eq:ineq1}}{\leq}& \left(1-\frac{\alpha_D}{2}\right) \frac{1}{G}\sum\limits_{i\in \cG} \EE{\norm{g_i^t - \nabla f_i(w^{t+1})}^2}
        + \left(1+\frac{2}{\alpha_D}\right) \parens{L_{\pm}^2 + L^2} \EE{G^{t+1}}.
    \end{eqnarray}
    The first term on the right hand side of \eqref{eq:gkef21pd} can be bounded as
    \begin{align*}
        \frac{1}{G}&\sum_{i\in \cG} \EE{\norm{g_i^t - \nabla f_i(w^{t+1})}^2} \\
        &\overset{\eqref{eq:young}}{\leq} \left(1+\frac{\alpha_D}{4}\right) \frac{1}{G}\sum\limits_{i\in \cG} \EE{\norm{ g_i^t - \nabla f_i(x^t) }^2} \\
        &\quad + \left(1+\frac{4}{\alpha_D}\right) \frac{1}{G}\sum\limits_{i\in \cG} \EE{\norm{ \nabla f_i(x^t) - \nabla f_i(w^{t+1}) }^2} \\
        &\overset{\eqref{as:hessian_variance}}{\leq} \left(1+\frac{\alpha_D}{4}\right) \frac{1}{G}\sum\limits_{i\in \cG} \EE{\norm{ g_i^t - \nabla f_i(x^t) }^2}
        + \left(1+\frac{4}{\alpha_D}\right) \parens{L_{\pm}^2 + L^2} \EE{\norm{ x^t-w^{t+1} }^2} \\
        &\overset{\eqref{eq:young}}{\leq} \left(1+\frac{\alpha_D}{4}\right) \frac{1}{G}\sum\limits_{i\in \cG} \EE{\norm{ g_i^t - \nabla f_i(x^t) }^2}
        + 2 \left(1+\frac{4}{\alpha_D}\right) \parens{L_{\pm}^2 + L^2} \EE{\norm{ x^t-x^{t+1} }^2} \\
        &\quad+ 2 \left(1+\frac{4}{\alpha_D}\right) \parens{L_{\pm}^2 + L^2} \EE{\norm{ x^{t+1}-w^{t+1} }^2} \\
        &= \left(1+\frac{\alpha_D}{4}\right) \EE{H^t}
        + 2 \left(1+\frac{4}{\alpha_D}\right) \parens{L_{\pm}^2 + L^2} \EE{R^t}\nonumber \\
        &\quad + 2 \left(1+\frac{4}{\alpha_D}\right) \parens{L_{\pm}^2 + L^2} \EE{G^{t+1}}.
    \end{align*}
    Applying the bound above in \eqref{eq:gkef21pd}, we obtain
    \begin{eqnarray*}
        \EE{H^{t+1}}
        &\overset{\eqref{eq:gkef21pd}}{\leq}& \left(1-\frac{\alpha_D}{2}\right) \frac{1}{G}\sum\limits_{i\in \cG} \EE{\norm{g_i^t - \nabla f_i(w^{t+1})}^2}
        + \left(1+\frac{2}{\alpha_D}\right) \parens{L_{\pm}^2 + L^2} \EE{G^{t+1}} \\
        &\leq& \left(1-\frac{\alpha_D}{2}\right) \left(1+\frac{\alpha_D}{4}\right) \EE{H^t} 
        + 2 \left(1-\frac{\alpha_D}{2}\right) \left(1+\frac{4}{\alpha_D}\right) \parens{L_{\pm}^2 + L^2} \EE{R^t} \\
        &&+ 2 \left(1-\frac{\alpha_D}{2}\right) \left(1+\frac{4}{\alpha_D}\right) \parens{L_{\pm}^2 + L^2} \EE{G^{t+1}} \\
        && + \left(1+\frac{2}{\alpha_D}\right) \parens{L_{\pm}^2 + L^2} \EE{G^{t+1}} \\
        &\overset{\eqref{eq:ineq1}}{\leq}& \left(1-\frac{\alpha_D}{4}\right) \EE{H^t}
        + 2 \left(\frac{4}{\alpha_D} - \frac{\alpha_D}{2} - 1\right) \parens{L_{\pm}^2 + L^2} \EE{R^t} \\
        &&+ \parens{
        2 \parens{\frac{4}{\alpha_D} - \frac{\alpha_D}{2} - 1} + 1 + \frac{2}{\alpha_D}} \parens{L_{\pm}^2 + L^2} \EE{G^{t+1}} \\
        &\leq& \left(1-\frac{\alpha_D}{4}\right) \EE{H^t}
        + \frac{8}{\alpha_D} \parens{L_{\pm}^2 + L^2} \EE{R^t}
        + \frac{10}{\alpha_D} \parens{L_{\pm}^2 + L^2} \EE{G^{t+1}}
    \end{eqnarray*}
    as required.
\end{proof}

\begin{lemma}\label{lemma:g_byzef21bc}
    Let $\cC^P$ be a contractive compressor. Then
    \begin{eqnarray}\label{eq:g_byzef21bc}
        \EE{G^{t+1}}
        \leq \left(1 - \frac{\alpha_P}{2}\right) \EE{G^t}
        + \frac{2}{\alpha_P} \EE{R^t}.
    \end{eqnarray}
\end{lemma}
\begin{proof}
    Using the update rule of $w^t$, we obtain
    \begin{eqnarray*}
        \EE{G^{t+1}}
        &=& \EE{\norm{w^{t+1} - x^{t+1}}^2} \\
        &=& \EE{\norm{w^{t} + \cC^{P}\left(x^{t+1} - w^t\right) - x^{t+1}}^2} \\
        &\leq& (1-\alpha_P) \EE{\norm{x^{t+1} - w^t}^2} \\
        &\stackrel{\eqref{eq:young}}{\leq}& (1-\alpha_P) \left(1 + \frac{\alpha_P}{2}\right) \EE{\norm{x^{t} - w^t}^2}
        + (1-\alpha_P) \left(1 + \frac{2}{\alpha_P}\right) \EE{\norm{x^{t+1} - x^t}^2} \\
        &\stackrel{\eqref{eq:ineq1}, \eqref{eq:ineq2}}{\leq}& \left(1-\frac{\alpha_P}{2}\right) \EE{\norm{x^{t} - w^t}^2}
        + \frac{2}{\alpha_P} \EE{\norm{x^{t+1} - x^t}^2}.
    \end{eqnarray*}
\end{proof}

\subsection{Proof of Theorem~\ref{thm:Byz_EF21_BC}}

We can now prove the main result.

\BYZEFTW*
\begin{proof}
    We start with Lemma \ref{lemma:descent_byz_ef21}:
    \begin{eqnarray*}
        \EE{f(x^{t + 1})} &\leq& \EE{f(x^t)} - \frac{\gamma \kappa}{2} \EE{\norm{\nabla f(x^{t})}^{2}}
        - \left(\frac{1}{2\gamma} - \frac{L}{2}\right) \EE{R^t}\\
        &&\qquad+ \frac{\gamma}{2} \parens{1 + \sqrt{8c\delta}}^2 \EE{H^t}
        + 4\gamma c\delta \left(1+\frac{1}{\sqrt{8c\delta}} \right)\zeta^2.
    \end{eqnarray*}
    Adding a $\frac{2\gamma}{\alpha_D} \parens{1 + \sqrt{8c\delta}}^2$ multiple of \eqref{eq:gf_rec_byzef21bc} gives
    \begin{align*}
        \EE{f(x^{t + 1})} &+ \frac{2\gamma}{\alpha_D} \parens{1 + \sqrt{8c\delta}}^2 \EE{H^{t+1}}
        \leq \EE{f(x^t)} - \frac{\gamma \kappa}{2} \EE{\norm{\nabla f(x^{t})}^{2}} \\
        &\qquad - \left(\frac{1}{2\gamma} - \frac{L}{2}\right) \EE{R^t}
        + \frac{\gamma}{2} \parens{1 + \sqrt{8c\delta}}^2 \EE{H^t}
        + 4\gamma c\delta \left(1+\frac{1}{\sqrt{8c\delta}} \right)\zeta^2 \\
        &\qquad + \frac{2\gamma}{\alpha_D} \parens{1 + \sqrt{8c\delta}}^2 \left(1-\frac{\alpha_D}{4}\right) \EE{H^t} \\
        &\qquad + \underbrace{\frac{2\gamma}{\alpha_D} \parens{1 + \sqrt{8c\delta}}^2 \frac{8}{\alpha_D} \parens{L_{\pm}^2 + L^2}}_{\frac{\gamma \nu}{2}} \EE{R^t} \\
        &\qquad + \underbrace{\frac{2\gamma}{\alpha_D} \parens{1 + \sqrt{8c\delta}}^2 \frac{10}{\alpha_D} \parens{L_{\pm}^2 + L^2}}_{\frac{5 \gamma \nu}{8}} \EE{G^{t+1}} \\
        &= \EE{f(x^t)}
        + \frac{2\gamma}{\alpha_D} \parens{1 + \sqrt{8c\delta}}^2 \EE{H^t} - \frac{\gamma \kappa}{2} \EE{\norm{\nabla f(x^{t})}^{2}} \\
        &\qquad - \left(\frac{1}{2\gamma} - \frac{L}{2} - \frac{\gamma \nu}{2}\right) \EE{R^t}
        + 4\gamma c\delta \left(1+\frac{1}{\sqrt{8c\delta}} \right)\zeta^2 
        + \frac{5\gamma \nu}{8} \EE{G^{t+1}},
    \end{align*}
    where $\nu \eqdef \frac{32}{\alpha_D^2} \parens{1 + \sqrt{8c\delta}}^2 \parens{L_{\pm}^2 + L^2}$.
    Next, adding a $\frac{5\gamma \nu}{4\alpha_P}$ multiple of \eqref{eq:g_byzef21bc}, we obtain
    \begin{align*}
        \EE{f(x^{t + 1})} + \frac{2\gamma}{\alpha_D} &\parens{1 + \sqrt{8c\delta}}^2 \EE{H^{t+1}}
        + \frac{5\gamma \nu}{4\alpha_P}\EE{G^{t+1}} \\
        &\leq \EE{f(x^t)}
        + \frac{2\gamma}{\alpha} \parens{1 + \sqrt{8c\delta}}^2 \EE{H^t}
        - \frac{\gamma \kappa}{2} \EE{\norm{\nabla f(x^{t})}^{2}} \\
        &\qquad - \left(\frac{1}{2\gamma} - \frac{L}{2} - \frac{\gamma \nu}{2}\right) \EE{R^t}
        + 4\gamma c\delta \left(1+\frac{1}{\sqrt{8c\delta}} \right)\zeta^2 \\
        &\qquad + \frac{5\gamma \nu}{8} \EE{G^{t+1}}
        + \frac{5\gamma \nu}{4\alpha_P}\left(1 - \frac{\alpha_P}{2}\right) \EE{G^t}
        + \frac{5\gamma \nu}{4\alpha_P}\frac{2}{\alpha_P} \EE{R^t} \\
        &= \EE{f(x^t)}
        + \frac{2\gamma}{\alpha} \parens{1 + \sqrt{8c\delta}}^2 \EE{H^t}
        - \frac{\gamma \kappa}{2} \EE{\norm{\nabla f(x^{t})}^{2}} \\
        &\qquad - \left(\frac{1}{2\gamma} - \frac{L}{2} - \frac{\gamma \eta}{2} \right) \EE{R^t}
        + 4\gamma c\delta \left(1+\frac{1}{\sqrt{8c\delta}} \right)\zeta^2 
        + \frac{5\gamma \nu}{8} \EE{G^{t+1}} \\
        &\qquad + \frac{5\gamma \nu}{4\alpha_P}\left(1 - \frac{\alpha_P}{2}\right) \EE{G^t},
    \end{align*}
    where $\eta \eqdef \nu + \frac{5\nu}{\alpha_P^2}$.
    Rearranging the above inequality, we get
    \begin{align*}
        \EE{f(x^{t + 1})} + \frac{2\gamma}{\alpha_D}& \parens{1 + \sqrt{8c\delta}}^2 \EE{H^{t+1}}
        + \frac{5\gamma \nu}{4\alpha_P}\left(1 - \frac{\alpha_P}{2}\right) \EE{G^{t+1}} \\
        &\leq \EE{f(x^t)}
        + \frac{2\gamma}{\alpha} \parens{1 + \sqrt{8c\delta}}^2 \EE{H^t}
        - \frac{\gamma \kappa}{2} \EE{\norm{\nabla f(x^{t})}^{2}} \\
        &\qquad - \left(\frac{1}{2\gamma} - \frac{L}{2} - \frac{\gamma \eta}{2} \right) \EE{R^t}
        + \frac{5\gamma \nu}{4\alpha_P}\left(1 - \frac{\alpha_P}{2}\right) \EE{G^t}
        + 4\gamma c\delta \left(1+\frac{1}{\sqrt{8c\delta}} \right)\zeta^2 \\
        &\leq \EE{f(x^t)}
        + \frac{2\gamma}{\alpha} \parens{1 + \sqrt{8c\delta}}^2 \EE{H^t}
        - \frac{\gamma \kappa}{2} \EE{\norm{\nabla f(x^{t})}^{2}} \\
        &\qquad + \frac{5\gamma \nu}{4\alpha_P}\left(1 - \frac{\alpha_P}{2}\right) \EE{G^t}
        + 4\gamma c\delta \left(1+\frac{1}{\sqrt{8c\delta}} \right)\zeta^2,
    \end{align*}
    where the last inequality follows from the fact that by Lemma \ref{lemma:step_lemma} and our assumption on the stepsize we have $\frac{1}{2\gamma} - \frac{L}{2} - \frac{\gamma \eta}{2} > 0$. Hence, for $\Psi^t \eqdef f(x^{t}) - f^* + \frac{2\gamma}{\alpha_D} \parens{1 + \sqrt{8c\delta}}^2 H^{t} + \frac{5\gamma \nu}{4\alpha_P}\left(1 - \frac{\alpha_P}{2}\right) G^{t}$, we obtain
    \begin{align*}
        \EE{\Psi^{t + 1}}
        &\leq \EE{\Psi^t}
        - \frac{\gamma \kappa}{2} \EE{\norm{\nabla f(x^{t})}^{2}}
        + 4\gamma c\delta \left(1+\frac{1}{\sqrt{8c\delta}} \right)\zeta^2,
    \end{align*}
    Summing the terms for $t=0,\ldots,T-1$ gives
    \begin{align*}
        \EE{\Psi^{T}}
        &\leq \EE{\Psi^0}
        - \frac{\gamma \kappa}{2} \sum_{t=0}^{T-1} \EE{\norm{\nabla f(x^{t})}^{2}}
        + 4\gamma c\delta \left(1+\frac{1}{\sqrt{8c\delta}} \right) T \zeta^2.
    \end{align*}
    Since from our assumption on $\delta$ it follows that $\kappa > 0$, rearranging the terms and dividing by $T$, we obtain
    \begin{align*}
        \sum_{t=0}^{T-1} \frac{1}{T} \EE{\norm{\nabla f(x^{t})}^{2}}
        &\leq \frac{2}{\gamma \kappa T} \EE{\Psi^0} - \frac{2}{\gamma \kappa T} \EE{\Psi^{T}}
        + \frac{8 c\delta}{\kappa} \left(1+\frac{1}{\sqrt{8c\delta}} \right) \zeta^2 \\
        &\leq \frac{2 \Psi^0}{\gamma \kappa T}
        + \frac{8 c\delta}{\kappa} \left(1+\frac{1}{\sqrt{8c\delta}} \right) \zeta^2,
    \end{align*}
    which proves the result.
\end{proof}

\begin{corollary}\label{cor:main_result_BC_general}
    Let the assumptions of Theorem \ref{thm:Byz_EF21_BC} hold and $\gamma = \parens{L + \sqrt{\eta}}^{-1}$, where $\eta = \frac{32}{\alpha_D^2} \parens{1 + \frac{5}{\alpha_P^2}} \parens{1 + \sqrt{8c\delta}}^2 \parens{L_{\pm}^2 + L^2}$. Then for all $T \geq 0$, $\EE{\norm{\nabla f(\hx^T)}^2}$ with $\hx^T$ chosen uniformly at random from $x^0, x^1, \ldots, x^{T-1}$ is of order
    \begin{align*}
        \cO\parens{\frac{\parens{L + \frac{1+\sqrt{c\delta}}{\alpha_D} \sqrt{\parens{1 + \frac{1}{\alpha_P^2}} \parens{L_{\pm}^2 + L^2}}}\delta^0}{T \parens{1 - \parens{c\delta+\sqrt{c\delta}}B}}
        + \frac{\left(c\delta + \sqrt{c\delta} \right) \zeta^2}{1 - \parens{c\delta+\sqrt{c\delta}}B}},
    \end{align*}
    where $\delta^0 = f(x^0) - f^*$.
    
    Hence, to guarantee $\EE{\norm{\nabla f(\hx^T)}^2} \leq \varepsilon^2$ for $\varepsilon^2 > {8 c\delta} \parens{1 - 8 B c\delta \left(1+\frac{1}{\sqrt{8c\delta}} \right)}^{-1} \left(1+\frac{1}{\sqrt{8c\delta}} \right) \zeta^2$, \algname{Byz-EF21}/\algname{Byz-EF21-BC} requires
    \begin{align*}
	    \cO\parens{\frac{1+\sqrt{c\delta}}{\alpha_D\alpha_P\varepsilon^2}}
    \end{align*}
    communication rounds.
\end{corollary}

\clearpage

\section{CONVERGENCE FOR POLYAK-{\L}OJASIEWICZ FUNCTIONS}\label{appendix:PL}

The $\cO(\nicefrac{1}{T})$ rate obtained for general smooth non-convex objective functions can be improved to a fast linear rate without assuming strong convexity of $f$ upon employing a weaker Polyak-Łojasiewicz condition.

\begin{assumption}[Polyak-Łojasiewicz condition]\label{as:pl}
    The function $f$ satisfies Polyak-Łojasiewicz (PŁ) condition with parameter $\mu$, i.e., for all $x \in \R^d$ there exists $x^*\in\arg\min_{x\in\R^d} f(x)$ such that
    \begin{align}
        2\mu\parens{f(x) - f(x^*)} \leq \norm{\nabla f(x)}^2.
    \end{align}
\end{assumption}
Below we prove the convergence results of our methods under the above assumption.

\subsection{\algname{Byz-VR-MARINA 2.0}}

\begin{lemma}[Descent lemma]\label{lemma:byz_BVMNS_pl}
    Suppose that Assumptions \ref{as:smoothness}, \ref{as:bounded_heterogeneity} and \ref{as:pl} hold. Then for all $s>0$ we have
    \begin{eqnarray*}
        \EE{f(x^{t + 1}) - f^*} &\leq& \parens{1-\gamma \mu \kappa} \EE{f(x^t) - f^*}
        - \left(\frac{1}{2\gamma} - \frac{L}{2}\right)
        \EE{R^t}
        + 4\gamma c\delta (1+s) \EE{H^t}\\
        &&\qquad+ \frac{\gamma}{2}(1 + s^{-1})\EE{\norm{\overline{g}^t - \nabla f(x^t)}^2}
        + 4\gamma c\delta (1+s)\zeta^2,
    \end{eqnarray*}
    where $\kappa = 1 - 8 B c\delta (1+s)$.
\end{lemma}
\begin{proof}
    The result follows from combining Lemma \ref{lemma:byz_BVMNS} with Assumption \ref{as:pl}.
\end{proof}

\begin{theorem}\label{thm:general_pl_MARINA}
	Let Assumptions \ref{as:smoothness}, \ref{as:hessian_variance}, \ref{as:hessian_variance_local}, \ref{as:bounded_heterogeneity}, and \ref{as:pl} hold. Suppose $g_i^0 = \nabla f_i(x^0)$ for all $i\in \cG$ and
    \begin{align*}
    0 < \gamma \leq \min\left\{\frac{1}{L + \sqrt{\eta}}, \frac{p}{2\mu} \right\},\qquad
    0<\delta<\frac{1}{\parens{8c+4\sqrt{c}}B},
    \end{align*}
    where $\eta = 2\frac{1-p}{p} \parens{\omega\parens{\frac{\cL_{\pm}^2}{b}+ L_{\pm}^2 + L^2} +\frac{\cL_{\pm}^2}{b}}\parens{\sqrt{\frac{1}{G}} + \sqrt{8c\delta}}^2$. Then for all $T \geq 0$ the iterates produced by \algname{Byz-VR-MARINA 2.0} satisfy
% \begin{equation*}
% \EE{f(x^T) - f^*} \leq \parens{1-\gamma \mu \kappa}^{T}\Phi_0 + \frac{4c\delta(1+s)\zeta^2}{\mu\kappa}
% \end{equation*}
\begin{eqnarray*}
    \EE{f(x^T) - f^*}
    &\leq& \parens{1-\gamma \mu \parens{1 - \parens{8 c\delta+\sqrt{\frac{8 c\delta}{G}}}B}}^{T}\EE{f(x^0) - f^*} \\
    &&+ \frac{\parens{4c\delta+\sqrt{\frac{2c\delta}{G}}}\zeta^2}{\mu\parens{1 - \parens{8 c\delta+\sqrt{\frac{8 c\delta}{G}}}B}}.
\end{eqnarray*}
where $\Phi_0 = f(x^0) - f^* + \gamma \parens{\frac{1+s^{-1}}{p} \norm{\overline{g}^0 - \nabla f\parens{x^0}}^2 + \frac{8c\delta\parens{1+s}}{p} H^t}$.
\end{theorem}

\begin{proof}
Let $C = \frac{2}{p}(1+s^{-1})$ and $D = \frac{16}{p}c\delta\parens{1+s}$ for some $s >0$ and denote
\begin{align*}
\psi_t &= C\norm{\overline{g}^t - \nabla f(x^{t})}^2 + D H^t, \\
\Phi_t &= f(x^t) - f^* + \frac{\gamma}{2}\psi_t.
\end{align*}
Using lemmas \ref{lemma:byz_BVMNS_pl}, \ref{lemma:prel_1_MARINA} and \ref{lemma:prel_2_MARINA}, we have:
\begin{eqnarray*}
    \EE{\Phi_{t+1}}
    &=&\EE{f(x^{t+1}) - f^* + \frac{\gamma}{2}\psi_{t+1}}\\
    &=&\underbrace{\EE{f(x^{t+1}) - f^*} }_\text{use \ref{lemma:byz_BVMNS_pl}} + \frac{\gamma}{2} \parens{C\underbrace{\EE{\norm{\overline{g}^{t+1} - \nabla f(x^{t+1})}^2}}_\text{use \ref{lemma:prel_1_MARINA}}  
    + D\underbrace{\EE{H^{t+1}}}_\text{use \ref{lemma:prel_2_MARINA}}}  \\
    &\leq& \parens{1-\gamma \mu\kappa}\EE{f(x^t)-f^*}\\ 
    && + \frac{\gamma}{2}\parens{1-\gamma \mu \kappa} \parens{\underbrace{C\EE{\norm{\overline{g}^t - \nabla f\parens{x^{t}}}^2} + D\EE{H^t} }_\text{$\EE{\psi_t}$}}
    \\
    && - \parens{\frac{1}{2\gamma} - \frac{L}{2} - \frac{\gamma\eta}{2}}\EE{R^t} + 4\gamma c\delta(1+s)\zeta^2 \\
    &=& \parens{1-\gamma \mu \kappa}\EE{\Phi_t} - \parens{\frac{1}{2\gamma} - \frac{L}{2} - \frac{\gamma\eta}{2}}\EE{R^t} + 4\gamma c\delta (1+s)\zeta^2,
\end{eqnarray*}
where 
\begin{eqnarray*}
    \eta &\eqdef & C\parens{1-p}\parens{\frac{\omega}{G}\parens{\frac{\cL_{\pm}^2}{b}+ L_{\pm}^2 + L^2} +\frac{\cL_{\pm}^2}{Gb}} + D\parens{1-p} \parens{\omega\parens{\frac{\cL_{\pm}^2}{b}+ L_{\pm}^2 + L^2} +\frac{\cL_{\pm}^2}{b}} \\
    % &= & 2\parens{1-p}\parens{\omega\parens{\frac{\cL_{\pm}^2}{b}+ L_{\pm}^2 + L^2} +\frac{\cL_{\pm}^2}{b}}\parens{\frac{C}{G} + D} \\
    &= & 2\frac{1-p}{p} \parens{\omega\parens{\frac{\cL_{\pm}^2}{b}+ L_{\pm}^2 + L^2} +\frac{\cL_{\pm}^2}{b}}\parens{\frac{1+s^{-1}}{G} + 8c\delta\parens{1+s}}.
\end{eqnarray*}
Taking $0 < \gamma \leq \frac{1}{L + \sqrt{\eta}}$ and using Lemma \ref{lemma:step_lemma} gives
\begin{equation*}
    \frac{1}{2\gamma} - \frac{L}{2} - \frac{\gamma\eta}{2} \geq 0,
\end{equation*}
and hence
\begin{equation*}
    \EE{\Phi_{t+1}} \leq \parens{1-\gamma \mu\kappa}\EE{\Phi_t} + 4\gamma c\delta(1+s)\zeta^2.
\end{equation*}
Similarly as in the proof of Theorem \ref{thm:main_result_Byz-VR-MARINA_2}, taking $s=\frac{1}{\sqrt{8c\delta G}}$, unrolling the recurrence and rearranging the terms, we get
\begin{eqnarray*}
    \EE{\Phi_{T}}
    &\leq & \parens{1-\gamma \mu\kappa}^{T}\EE{\Phi_0} + 4c\delta\gamma(1+s)\zeta^2\sum \limits_{t=0}^{T}\parens{1-\gamma \mu \kappa}^{t} \\
    &\leq & \parens{1-\gamma \mu \kappa}^{T}\EE{\Phi_0} + 4c\delta\gamma(1+s)\zeta^2\sum \limits_{t=0}^{\infty}\parens{1-\gamma \mu \kappa}^{t} \\
    &= &\parens{1-\gamma \mu \kappa}^{T}\EE{\Phi_0} + \frac{4c\delta(1+s)\zeta^2}{\mu\kappa} \\
    &= &\parens{1-\gamma \mu \parens{1 - \parens{8 c\delta+\sqrt{\frac{8 c\delta}{G}}}B}}^{T}\EE{f(x^0) - f^*}\\
    &&+ \frac{\parens{4c\delta+\sqrt{\frac{2c\delta}{G}}}\zeta^2}{\mu\parens{1 - \parens{8 c\delta+\sqrt{\frac{8 c\delta}{G}}}B}}.
\end{eqnarray*}
The result follows from the fact that $\Phi_T \geq f(x^T) - f^*$.
\end{proof}

\subsection{\algname{Byz-DASHA-PAGE}}

\begin{lemma} [Descent lemma]
\label{lemma:descent_dp_pl}
Suppose that Assumptions \ref{as:smoothness}, \ref{as:bounded_heterogeneity} and \ref{as:pl} hold. Then for all $s>0$ we have
\begin{eqnarray*}
  &&\EE{f(x^{t + 1}) - f^*} \leq \parens{1-\gamma \mu \kappa} \EE{f(x^t) - f^*}
  - \left(\frac{1}{2\gamma} - \frac{L}{2}\right)
  \EE{R^t}\\
  &&\qquad+ 8\gamma c\delta(1+s)\parens{\frac{1}{G}\sum \limits_{i\in \cG} \EE{\norm{g_i^t - h_i^t}^2}}
  + 8\gamma c\delta(1+s)\parens{\frac{1}{G}\sum \limits_{i\in \cG} \EE{\norm{h_i^t - \nabla f_i(x^t)}^2}} \\
  &&\qquad+ \gamma (1 + s^{-1})\EE{\norm{\overline{g}^t - h^{t}}^2} + \gamma (1 + s^{-1})\EE{\norm{h^t - \nabla f(x^t)}^2}
  + 4\gamma c\delta(1+s)\zeta^2,
\end{eqnarray*}
where $\kappa = 1 - 8 B c\delta (1+s)$.
\end{lemma}
\begin{proof}
    The result follows from combining Lemma \ref{lemma:byz_descent_dp} with Assumption \ref{as:pl}.
\end{proof}

\begin{lemma}[Calculations]\label{lemma:system_1}
Let $\kappa = 1 - 8 B c\delta (1+s)$ and assume that
\begin{align*}
    0& < a \leq \frac{1}{2\omega +1}, \\
    0& < \gamma < \min \left\{\frac{p}{2\mu \kappa}\text{, } \frac{a}{2\mu \kappa} \right\}.
\end{align*}
The inequalities
\begin{align}\label{system_1}
\begin{cases}
    C, D, E, F\geq 0\\
     \parens{1-\gamma \mu \kappa}C \geq (1-a)^{2}C + 2 \parens{1+s^{-1}}\\
     \parens{1-\gamma \mu \kappa}D \geq \parens{2a^{2}\omega + \parens{1-a}^{2}}D + \frac{2a^{2}\omega}{G}C + 16c\delta\parens{1+s} \\
     \parens{1-\gamma \mu \kappa}E \geq \parens{1-p}E + 2\parens{1+s^{-1}}\\
     \parens{1-\gamma \mu \kappa}F \geq \parens{1-p}F + 4p\omega \parens{\frac{C}{G}+D} + 16c\delta\parens{1+s}
\end{cases}
\end{align}
are satisfied when
$$\begin{cases}
    C = \frac{2(1+s^{-1})}{1-\parens{1-a}^2 - \frac{a}{2}}\\
    D = \frac{1}{1-2\omega a^2 - (1-a)^2 - \frac{a}{2}} \parens{16 \parens{1+s}c\delta + \frac{2\omega a^2}{G}C} \\
    E = \frac{4(1+s^{-1})}{p} \\
    F  = 8\omega \parens{\frac{C}{G} + D} + \frac{32 c\delta}{p}\parens{1 + s}.
\end{cases}$$
Further, for this choice of $C$, $D$, $E$, $F$, we have
\begin{align}\label{eq:cdefeqbdp}
    \frac{C}{G} + D & = \frac{2}{1-2\omega a^2 - (1-a)^2 - \frac{a}{2}}\parens{\frac{1+s^{-1}}{G} + \parens{1+s}8c\delta} \nonumber \\
    \frac{E}{G} + F & = \parens{\frac{16\omega}{1-2\omega a^2 - (1-a)^2-\frac{a}{2}} + \frac{4}{p}}\parens{\frac{1+s^{-1}}{G} + \parens{1+s}8c\delta}.
\end{align}
\end{lemma}
\begin{proof}
Let $w = \gamma \mu  \kappa$. Under the assumption that $0< \gamma< \min \{\frac{p}{2\mu \kappa}\text{, } \frac{a}{2\mu \kappa}\}$, we have that $0<w< \min \{\frac{p}{2}\text{, }\frac{a}{2}\}$. Furthermore, \eqref{system_1} implies that
\begin{align*}\begin{cases}
    C \geq 2\frac{1+s^{-1}}{1-\parens{1-a}^2 - w}\\
    D \geq \frac{1}{1-2\omega a^2 - (1-a)^2 - w} \parens{16 \parens{1+s}c\delta + \frac{2\omega a^2}{G}C} \\
    E \geq 2\frac{1+s^{-1}}{p-w} \\
    F \geq \frac{4\omega p}{p-w} \parens{\frac{C}{G} + D} + \frac{16 c\delta}{p-w}\parens{1 + s}
\end{cases}
\end{align*}
The assumption that $0<a \leq \frac{1}{2\omega +1}$ ensures that $1-2\omega a^2 - (1-a)^2 - \frac{a}{2} > 0$.
Since $0<w< \min \{\frac{p}{2}\text{, }\frac{a}{2}\}$, one can take
\begin{align*}\begin{cases}
    C = 2\frac{1+s^{-1}}{1-\parens{1-a}^2 - \frac{a}{2}}\\
    D = \frac{1}{1-2\omega a^2 - (1-a)^2 - \frac{a}{2}} \parens{16 \parens{1+s}c\delta + \frac{2\omega a^2}{G}C} \\
    E = 4\frac{1+s^{-1}}{p} \\
    F  = 8\omega \parens{\frac{C}{G} + D} + \frac{32 c\delta}{p}\parens{1 + s}.
\end{cases}
\end{align*}
It remains to verify \eqref{eq:cdefeqbdp} by direct substitution of the values above.
% \begin{align*}
% \frac{C}{G} + D &= \frac{C}{G} + \frac{1}{1-2\omega a^2 - (1-a)^2 - \frac{a}{2}} \parens{16\parens{1+s}c\delta + C\frac{2\omega a^2}{G}}\\
% & = \frac{1}{1-2\omega a^2 - (1-a)^2-\frac{a}{2}}\parens{\frac{C}{G}\parens{1- 2\omega a^2 - (1-a)^2-\frac{a}{2}} + 16\parens{1+s}c\delta + C\frac{2\omega a^2}{G}}\\
% & = \frac{1}{1-2\omega a^2 - (1-a)^2-\frac{a}{2}}\parens{\frac{C}{G}\parens{1 - (1-a)^2-\frac{a}{2}} + 16\parens{1+s}c\delta}\\
% & = \frac{2}{1-2\omega a^2 - (1-a)^2 - \frac{a}{2}}\parens{\frac{1+s^{-1}}{G} + \parens{1+s}8c\delta}\\
% \frac{E}{G} + F &= 4\frac{1+s^{-1}}{pG} + 8\omega \parens{C+\frac{D}{G}} + 32\frac{c\delta}{p}\parens{1+s}\\
% &= \frac{4}{p}\parens{\frac{1+s^{-1}}{G} + 8c\delta\parens{1+s}} + 8\omega \parens{C+\frac{D}{G}}\\
% & = \parens{\frac{16\omega}{1-2\omega a^2 - (1-a)^2-\frac{a}{2}} + \frac{4}{p}}\parens{\frac{1+s^{-1}}{G} + \parens{1+s}8c\delta}
% \end{align*}
% as needed.
\end{proof}

\begin{theorem}[General Convergence]\label{thm:general_pl_DASHA}
	Let Assumptions \ref{as:smoothness}, \ref{as:hessian_variance}, \ref{as:bounded_heterogeneity}, \ref{as:hessian_variance_local}, and \ref{as:pl} hold. Let
$$\begin{cases}
    C = \frac{2(1+\sqrt{8c\delta G})}{1-\parens{1-a}^2 - \frac{a}{2}}\\
    D = \frac{1}{1-2\omega a^2 - (1-a)^2 - \frac{a}{2}} \parens{16 \parens{1+\frac{1}{\sqrt{8c\delta G}}}c\delta + \frac{2\omega a^2}{G}C} \\
    E = \frac{4(1+\sqrt{8c\delta G})}{p} \\
    F  = 8\omega \parens{\frac{C}{G} + D} + \frac{32 c\delta}{p}\parens{1 + \frac{1}{\sqrt{8c\delta G}}}
\end{cases}$$
and assume that
\begin{align*}
	0 < \gamma &\leq \min \left\{\frac{1}{L + \sqrt{\eta}}, \frac{p}{2\mu\kappa}\text{, } \frac{a}{2\mu\kappa} \right\}, \qquad
    0 < a \leq \frac{1}{2\omega +1}, \qquad
    0<\delta <\frac{1}{8c\parens{1+\frac{1}{\sqrt{8c\delta G}}}B},
\end{align*}
where $\eta = \parens{\frac{8\omega \parens{L_{\pm}^2 + L^2}}{1-2\omega a^2 - (1-a)^2 - \frac{a}{2}} + \frac{1-p}{b}\cL_{\pm}^{2} \parens{\frac{20\omega}{1-2\omega a^2 - (1-a)^2 - \frac{a}{2}} + \frac{4}{p}}}\parens{\frac{1+\sqrt{8c\delta G}}{G} + \parens{1+\frac{1}{\sqrt{8c\delta G}}}8c\delta}$ and $\kappa = 1 - 8 B c\delta \parens{1+\frac{1}{\sqrt{8c\delta G}}}$. Then for all $T \geq 0$ the iterates produced by \algname{Byz-DASHA-PAGE} satisfy
\begin{align*}
	\EE{f(x^T) - f^*} &\leq \parens{1-\gamma \mu \parens{1 - \parens{8 c\delta+\sqrt{\frac{8 c\delta}{G}}}B}}^{T}\EE{f(x^0) - f^*}\\
    &\quad+ \frac{\parens{4c\delta+\sqrt{\frac{2c\delta}{G}}}\zeta^2}{\mu \parens{1 - \parens{8 c\delta+\sqrt{\frac{8 c\delta}{G}}}B}}.
\end{align*}
\end{theorem}

\begin{proof}
Let
\begin{align*}
\psi_t &= C\norm{\overline{g}_t - h^{t}}^2 + D\parens{\frac{1}{G}\sum \limits_{i \in \cG}\norm{g_i^t - h_i^t}^2} + E \norm{h^{t}-\nabla f(x^{t})}^{2} \\
&\quad+ F\parens{\frac{1}{G}\sum \limits_{i \in \cG}\norm{h_i^t - \nabla f_i(x^t)}^2}, \\
\Phi_t &= f(x^t) - f^* + \frac{\gamma}{2}\psi_t.
\end{align*}

Using lemmas \ref{lemma:descent_dp_pl}, \ref{lemma: pre1_dp}, \ref{lemma: pre2_dp}, \ref{lemma:pre3_dp} and \ref{lemma:pre4_dp}, we have
\begin{align*}
    \EE{\Phi_{t+1}}
    &=\EE{f(x^{t+1}) - f^* + \frac{\gamma}{2}\psi_{t+1}}\\
    &=\underbrace{\EE{f(x^{t+1}) - f^*} }_\text{use \ref{lemma:descent_dp_pl}} + \frac{\gamma}{2} C\underbrace{\EE{\norm{\overline{g}^{t+1} - h^{t+1}}^2}}_\text{use \ref{lemma: pre1_dp}}  
    + \frac{\gamma}{2} D\underbrace{\parens{\frac{1}{G}\sum \limits_{i \in \cG}\EE{\norm{g_i^{t+1} - h_i^{t+1}}^2}}}_\text{use \ref{lemma: pre2_dp}}  \\ & + \frac{\gamma}{2} E\underbrace{\EE{\norm{h^{t+1} - \nabla f(x^{t+1})}^2}}_\text{use \ref{lemma:pre3_dp}}  
    + \frac{\gamma}{2} F\underbrace{\parens{\frac{1}{G}\sum \limits_{i \in \cG}\EE{\norm{h_i^{t+1} - \nabla f_i(x^{t+1})}^2}}}_\text{use \ref{lemma:pre4_dp}}   \\
    &\leq \parens{1-\gamma \mu\kappa}\EE{f(x^{t}) - f^*}\\
    &\quad+ \parens{1-\gamma \mu\kappa} \frac{\gamma}{2}\left(C\EE{\norm{\overline{g}_t - h^{t}}^2} + D\parens{\frac{1}{G}\sum \limits_{i \in \cG}\EE{\norm{g_i^t - h_i^t}^2}} \right. \\ 
    &\quad\left.+ E \EE{\norm{h^{t}-\nabla f(x^{t})}^{2}}
    + F\parens{\frac{1}{G}\sum \limits_{i \in \cG}\EE{\norm{h_i^t - \nabla f_i(x^t)}^2}}\right) \\
    &\quad- \parens{\frac{1}{2\gamma} - \frac{L}{2} - \frac{\gamma\eta}{2}}\EE{R^t} + 4\gamma c\delta(1+s)\zeta^2 \\
    % =& \parens{1-\gamma \mu\kappa}\parens{\EE{f(x^{t}) - f^*} + \frac{\gamma}{2}\EE{\psi_t}} - \parens{\frac{1}{2\gamma} - \frac{L}{2} - \frac{\gamma\eta}{2}}\EE{R^t} + 4\gamma c\delta(1+s)\zeta^2 \\
    &= \parens{1-\gamma \mu\kappa}\EE{\Phi_t} - \parens{\frac{1}{2\gamma} - \frac{L}{2} - \frac{\gamma\eta}{2}}\EE{R^t} + 4\gamma c\delta(1+s)\zeta^2,
\end{align*}
where the inequality holds by Lemma \ref{lemma:system_1} and
\begin{align*}
    \eta &= 2\omega \parens{\frac{1-p}{b}\cL_{\pm}^{2} + 2 \parens{L_{\pm}^2 + L^2}}\parens{\frac{C}{G} + D} + \frac{1-p}{B}\cL_{\pm}^{2}\parens{\frac{E}{G} + F} \\
    % &=4\omega \parens{L_{\pm}^2 + L^2}\parens{\frac{C}{G}+D} + \frac{1-p}{b}\cL_{\pm}^{2}\parens{2\omega \parens{\frac{C}{G} + D} + \parens{\frac{E}{G}+F}} \\
    % &=\frac{8\omega \parens{L_{\pm}^2 + L^2}}{1-2\omega a^2 - (1-a)^2 -\frac{a}{2}}\parens{\frac{1+s^{-1}}{G} + \parens{1+s}8c\delta} \\
    % &+ \frac{1-p}{b}\cL_{\pm}^{2} \parens{\frac{20\omega}{1-2\omega a^2 - (1-a)^2 - \frac{a}{2}} + \frac{4}{p}} \parens{\frac{1+s^{-1}}{G} + \parens{1+s}8c\delta}\\
    &= \parens{\frac{8\omega \parens{L_{\pm}^2 + L^2}}{1-2\omega a^2 - (1-a)^2 - \frac{a}{2}} + \frac{1-p}{b}\cL_{\pm}^{2} \parens{\frac{20\omega}{1-2\omega a^2 - (1-a)^2 - \frac{a}{2}} + \frac{4}{p}}}\\
    &\quad\times\parens{\frac{1+s^{-1}}{G} + \parens{1+s}8c\delta}.
\end{align*}
Using the assumption on the stepsize and Lemma \ref{lemma:step_lemma}, we have
\begin{equation*}
    \frac{1}{2\gamma} - \frac{L}{2} - \frac{\gamma\eta}{2} \geq 0
\end{equation*}
and hence
\begin{equation*}
    \EE{\Phi_{t+1}} \leq \parens{1-\gamma \mu\kappa}\EE{\Phi_t} + 4\gamma c\delta(1+s)\zeta^2.
\end{equation*}
Unrolling the recurrence and rearranging the terms, we get
\begin{align*}
	\EE{\Phi_{T}} &\leq \parens{1-\gamma \mu\kappa}^{T}\EE{\Phi_0} + 4c\delta\gamma(1+s)\zeta^2\sum \limits_{t=0}^{T}\parens{1-\gamma \mu\kappa}^{t} \\
    &\leq \parens{1-\gamma \mu\kappa}^{T}\EE{\Phi_0} + 4c\delta\gamma(1+s)\zeta^2\sum \limits_{t=0}^{\infty}\parens{1-\gamma \mu\kappa}^{t} \\
    &= \parens{1-\gamma \mu\kappa}^{T}\EE{\Phi_0} + \frac{4c\delta(1+s)\zeta^2}{ \mu\kappa}.
\end{align*}
Noting that $\Phi_T \geq f(x^T) - f^*$ and letting $s=\frac{1}{\sqrt{8c\delta G}}$ gives the result.
\end{proof}

\subsection{\algname{Byz-EF21} and \algname{Byz-EF21-BC}}

\begin{theorem}\label{thm:main_result_BC_PL}
    Let Assumptions \ref{as:smoothness}, \ref{as:hessian_variance}, \ref{as:bounded_heterogeneity} and \ref{as:pl} hold and suppose that
    \begin{align}\label{eq:gamma_pl_bc}
	0 < \gamma \leq \min \left\{ \frac{1}{L + \sqrt{\eta}}, \frac{\alpha_D}{8\kappa\mu}, \frac{\alpha_P}{4\kappa\mu} \right\}, \qquad \delta < \frac{1}{8c(\sqrt{B}+B)^2},
    \end{align}
    where $\eta = \frac{64}{\alpha_D^2} \parens{1 + \frac{10}{\alpha_P^2}\parens{1 - \frac{\alpha_P}{4}}} \parens{1 + \sqrt{8c\delta}}^2 \parens{L_{\pm}^2 + L^2}$ and $\kappa = 1 - 8 B c\delta \left(1+\frac{1}{\sqrt{8c\delta}} \right)$. Then for all $T \geq 0$ the iterates of \algname{Byz-EF21-BC} satisfy
    \begin{eqnarray*}
        \EE{f(x^{T}) - f^*}
        &\leq& \parens{1 - \gamma\mu\parens{1 - \parens{8 c\delta+\sqrt{8c\delta}}B}}^T \Psi^{0}
        + \frac{\parens{4 c\delta+\sqrt{2c\delta}} \zeta^2}{\parens{1 - \parens{8 c\delta+\sqrt{8c\delta}}B}\mu},
    \end{eqnarray*}
    where $\Psi^0 \eqdef f(x^{t}) - f^* + \frac{4\gamma}{\alpha_D} \parens{1 + \sqrt{8c\delta}}^2 H^{0} + \frac{320 \gamma}{2\alpha_P \alpha_D^2} \parens{1 - \frac{\alpha_P}{2}} \parens{1 + \sqrt{8c\delta}}^2 \parens{L_{\pm}^2 + L^2} G^{0}$.  
\end{theorem}

\begin{proof}
    Adding a $\frac{4\gamma}{\alpha_D} \parens{1 + \sqrt{8c\delta}}^2$ multiple of \eqref{eq:gf_rec_byzef21bc} to the inequality from Lemma \ref{lemma:descent_byz_ef21} gives
    \begin{align*}
        \EE{f(x^{t+1})} &+ \frac{4\gamma}{\alpha_D} \parens{1 + \sqrt{8c\delta}}^2 \EE{H^{t+1}} \\
        &\leq \EE{f(x^t)} - \frac{\gamma\kappa}{2} \EE{\norm{\nabla f(x^{t})}^{2}}
        - \left(\frac{1}{2\gamma} - \frac{L}{2}\right)
        \EE{R^t}\\
        &\qquad+ \frac{\gamma}{2} \parens{1 + \sqrt{8c\delta}}^2 \EE{H^t}
        + 4\gamma c\delta \left(1+\frac{1}{\sqrt{8c\delta}} \right)\zeta^2 \\
        &\qquad + \frac{4\gamma}{\alpha_D} \parens{1 + \sqrt{8c\delta}}^2 \left(1-\frac{\alpha_D}{4}\right) \EE{H^t} \\
        &\qquad + \underbrace{\frac{4\gamma}{\alpha_D} \parens{1 + \sqrt{8c\delta}}^2 \frac{8}{\alpha_D} \parens{L_{\pm}^2 + L^2}}_{\frac{\gamma\nu}{2}} \EE{R^t} \\
        &\qquad + \underbrace{\frac{4\gamma}{\alpha_D} \parens{1 + \sqrt{8c\delta}}^2 \frac{10}{\alpha_D} \parens{L_{\pm}^2 + L^2}}_{\frac{5\gamma\nu}{8}} \EE{G^{t+1}} \\
        &= \EE{f(x^t)} + \frac{4\gamma}{\alpha_D} \parens{1 + \sqrt{8c\delta}}^2 \parens{1 - \frac{\alpha_D}{8}} \EE{H^t}
        - \frac{\gamma\kappa}{2} \EE{\norm{\nabla f(x^{t})}^{2}} \\
        &\qquad - \left(\frac{1}{2\gamma} - \frac{L}{2} - \frac{\gamma\nu}{2}\right)
        \EE{R^t}
        + 4\gamma c\delta \left(1+\frac{1}{\sqrt{8c\delta}} \right)\zeta^2
        + \frac{5\gamma\nu}{8} \EE{G^{t+1}},
    \end{align*}
    where $\nu \eqdef \frac{64}{\alpha_D^2} \parens{1 + \sqrt{8c\delta}}^2 \parens{L_{\pm}^2 + L^2}$.
    Adding $\frac{5\gamma\nu}{2\alpha_P}\parens{1 - \frac{\alpha_P}{4}}$ multiple of \eqref{eq:g_byzef21bc}, we obtain
    \begin{align*}
        \EE{f(x^{t+1})} &+ \frac{4\gamma}{\alpha_D} \parens{1 + \sqrt{8c\delta}}^2 \EE{H^{t+1}}
        + \frac{5\gamma\nu}{2\alpha_P}\parens{1 - \frac{\alpha_P}{4}} \EE{G^{t+1}} \\
        &\leq \EE{f(x^t)} + \frac{4\gamma}{\alpha_D} \parens{1 + \sqrt{8c\delta}}^2 \parens{1 - \frac{\alpha_D}{8}} \EE{H^t}
        - \frac{\gamma\kappa}{2} \EE{\norm{\nabla f(x^{t})}^{2}} \\
        &\qquad - \left(\frac{1}{2\gamma} - \frac{L}{2} - \frac{\gamma\nu}{2}\right)
        \EE{R^t}
        + 4\gamma c\delta \left(1+\frac{1}{\sqrt{8c\delta}} \right)\zeta^2
        + \frac{5\gamma\nu}{8} \EE{G^{t+1}} \\
        &\qquad+ \frac{5\gamma\nu}{2\alpha_P}\parens{1 - \frac{\alpha_P}{4}} \left(1 - \frac{\alpha_P}{2}\right) \EE{G^t}
        + \frac{5\gamma\nu}{2\alpha_P}\parens{1 - \frac{\alpha_P}{4}} \frac{2}{\alpha_P} \EE{R^t} \\
        &= \EE{f(x^t)} + \frac{4\gamma}{\alpha_D} \parens{1 + \sqrt{8c\delta}}^2 \parens{1 - \frac{\alpha_D}{8}} \EE{H^t}
        - \frac{\gamma\kappa}{2} \EE{\norm{\nabla f(x^{t})}^{2}} \\
        &\qquad - \left(\frac{1}{2\gamma} - \frac{L}{2} - \frac{\gamma\eta}{2}\right)
        \EE{R^t}
        + 4\gamma c\delta \left(1+\frac{1}{\sqrt{8c\delta}} \right)\zeta^2
        + \frac{5\gamma\nu}{8} \EE{G^{t+1}} \\
        &\qquad+ \frac{5\gamma\nu}{2\alpha_P}\parens{1 - \frac{\alpha_P}{4}} \left(1 - \frac{\alpha_P}{2}\right) \EE{G^t},
    \end{align*}
    where $\eta \eqdef \nu + \frac{10\nu}{\alpha_P^2}\parens{1 - \frac{\alpha_P}{4}}$.
    Rearranging the above inequality gives
    \begin{align*}
        \EE{f(x^{t+1})} &+ \frac{4\gamma}{\alpha_D} \parens{1 + \sqrt{8c\delta}}^2 \EE{H^{t+1}}
        + \frac{5\gamma\nu}{2\alpha_P}\parens{1 - \frac{\alpha_P}{2}} \EE{G^{t+1}} \\
        &\leq \EE{f(x^t)} + \frac{4\gamma}{\alpha_D} \parens{1 + \sqrt{8c\delta}}^2 \parens{1 - \frac{\alpha_D}{8}} \EE{H^t}
        - \frac{\gamma\kappa}{2} \EE{\norm{\nabla f(x^{t})}^{2}} \\
        &\qquad - \left(\frac{1}{2\gamma} - \frac{L}{2} - \frac{\gamma\eta}{2}\right)
        \EE{R^t}
        + 4\gamma c\delta \left(1+\frac{1}{\sqrt{8c\delta}} \right)\zeta^2 \\
        &\qquad+ \frac{5\gamma\nu}{2\alpha_P} \parens{1 - \frac{\alpha_P}{2}} \parens{1 - \frac{\alpha_P}{4}} \EE{G^t} \\
        &\leq \EE{f(x^t)} + \frac{4\gamma}{\alpha_D} \parens{1 + \sqrt{8c\delta}}^2 \parens{1 - \frac{\alpha_D}{8}} \EE{H^t}
        - \frac{\gamma\kappa}{2} \EE{\norm{\nabla f(x^{t})}^{2}} \\
        &\qquad + 4\gamma c\delta \left(1+\frac{1}{\sqrt{8c\delta}} \right)\zeta^2
        + \frac{5\gamma\nu}{2\alpha_P} \parens{1 - \frac{\alpha_P}{2}} \parens{1 - \frac{\alpha_P}{4}} \EE{G^t},
    \end{align*}
    where in the last inequality we use the fact that $\frac{1}{2\gamma} - \frac{L}{2} - \frac{\gamma\eta}{2} \geq 0$ by Lemma \ref{lemma:step_lemma} and our assumption on the stepsize.
    Now, let us define $\Psi^t = f(x^{t}) - f^* + \frac{4\gamma}{\alpha_D} \parens{1 + \sqrt{8c\delta}}^2 H^{t} + \frac{5\gamma\nu}{2\alpha_P}\parens{1 - \frac{\alpha_P}{2}} G^{t}$. Subtracting $f^*$ from both sides of the above inequality and noting that our assumption on $\delta$ implies that $\kappa > 0$, the PŁ inequality (Assumption \ref{as:pl}) gives
    \begin{eqnarray*}
        \EE{\Psi^{t+1}}
        &=& \EE{f(x^{t+1}) - f^*} + \frac{4\gamma}{\alpha_D} \parens{1 + \sqrt{8c\delta}}^2 \EE{H^{t+1}}
        + \frac{5\gamma\nu}{2\alpha_P}\parens{1 - \frac{\alpha_P}{2}} \EE{G^{t+1}} \\
        &\stackrel{\eqref{as:pl}}{\leq}& \EE{f(x^t) - f^*} + \frac{4\gamma}{\alpha_D} \parens{1 + \sqrt{8c\delta}}^2 \parens{1 - \frac{\alpha_D}{8}} \EE{H^t}
        - \gamma\kappa\mu \EE{f(x^t) - f^*} \\
        &&\qquad + 4\gamma c\delta \left(1+\frac{1}{\sqrt{8c\delta}} \right)\zeta^2
        + \frac{5\gamma\nu}{2\alpha_P} \parens{1 - \frac{\alpha_P}{2}} \parens{1 - \frac{\alpha_P}{4}} \EE{G^t} \\
        &=& \parens{1 - \gamma\kappa\mu} \EE{f(x^t) - f^*} + \frac{4\gamma}{\alpha_D} \parens{1 + \sqrt{8c\delta}}^2 \parens{1 - \frac{\alpha_D}{8}} \EE{H^t} \\
        &&\qquad + \frac{5\gamma\nu}{2\alpha_P} \parens{1 - \frac{\alpha_P}{2}} \parens{1 - \frac{\alpha_P}{4}} \EE{G^t}
        + 4\gamma c\delta \left(1+\frac{1}{\sqrt{8c\delta}} \right)\zeta^2 \\
        &\stackrel{\eqref{eq:gamma_pl_bc}}{\leq}& \parens{1 - \gamma\kappa\mu} \EE{\Psi^t}
        + 4\gamma c\delta \left(1+\frac{1}{\sqrt{8c\delta}} \right)\zeta^2,
    \end{eqnarray*}
    where in the last step we use the assumption on $\gamma$ to establish that $1 - \frac{\alpha_D}{8} \leq 1 - \gamma\kappa\mu$ and $1 - \frac{\alpha_P}{4} \leq 1 - \gamma\kappa\mu$. Applying the inequality iteratively gives
    \begin{eqnarray*}
        \EE{\Psi^{T}}
        % &\leq& \parens{1 - \gamma\kappa\mu} \EE{\Psi^{T-1}}
        % + 4\gamma c\delta \left(1+\frac{1}{\sqrt{8c\delta}} \right)\zeta^2 \\
        &\leq& \parens{1 - \gamma\kappa\mu}^T \EE{\Psi^{0}}
        + 4\gamma c\delta \left(1+\frac{1}{\sqrt{8c\delta}} \right)\zeta^2 \sum_{t=0}^{T-1} \parens{1 - \gamma\kappa\mu}^t \\
        &\leq& \parens{1 - \gamma\kappa\mu}^T \EE{\Psi^{0}}
        + 4\gamma c\delta \left(1+\frac{1}{\sqrt{8c\delta}} \right)\zeta^2 \sum_{t=0}^{\infty} \parens{1 - \gamma\kappa\mu}^t \\
        &=& \parens{1 - \gamma\kappa\mu}^T \EE{\Psi^{0}}
        + \frac{4 c\delta}{\kappa\mu} \left(1+\frac{1}{\sqrt{8c\delta}} \right)\zeta^2.
    \end{eqnarray*}
    Noting that $\EE{\Psi^{T}} \geq \EE{f(x^{T}) - f^*}$, we finish the proof.   
\end{proof}

\clearpage

\section{NUMERICAL EXPERIMENTS: ADDITIONAL DETAILS AND RESULTS}\label{appendix:extra_exp_details}

\subsection{Logistic regression experiments}

\paragraph{Objective:}
We consider solving the following logistic regression problem with regularization $r$:
\begin{align*}
    \min_{x \in \R^d} \brac{f(x) = \frac{1}{N} \sum_{i=1}^N \log\left( 1 + \exp\left( -y_i a_i^Tx \right) \right) + \frac{\lambda}{2} r(x)},
\end{align*}
where $d$ is the dimension of the model, $N$ is the total number of samples, $a_i^{T}$ is the $i$th row of the design matrix $A \in \mathbb{R}^{N\times d}$, $y_i \in \{-1, 1\}$ are the corresponding labels and $\lambda > 0$ is the regularization parameter. We choose $\lambda = 0.1$ in all experiments.
The regularizer is chosen depending on the nature of the problem we are treating:
\begin{itemize}
    \item \textbf{Non-convex case:} a non-convex regularizer $r:$ $x \mapsto \sum_{j=1}^{d} \frac{x_j^2}{1+x_j^2}$
    \item \textbf{P{\L} case:} the ridge regularization $r:$ $x \mapsto \norm{x}^{2}$. It can be shown that for this chice of $r$, for $\lambda >0$, the objective function is strongly convex and therefore satisfies the PŁ condition.
\end{itemize}
% Following the theory that we derived for our algorithms we will be monitoring the following metrics in those settings:
% \begin{itemize}
%     \item \textbf{Non-convex case} The gradient norm $\norm{\nabla f(x^t)}^2$
%     \item \textbf{P{\L} case} The optimality gap $f(x^{t})-f(x^*)$
% \end{itemize}

\paragraph{Datasets:}

We consider two \texttt{LibSVM} datasets \citep{chang2011libsvm} (summarised in Table \ref{table:2}). In each case, the data is distributed among $n=16$ workers, out of which $3$ are Byzantine.
\begin{table}[H]
\centering
\caption{Overview of the \texttt{LibSVM} datasets used.}
\label{table:2}
\begin{tabular}{c c c} 
 \hline
 Dataset & N (\text{\# samples}) & d (\text{\# features}) \\
 \hline
 \texttt{phishing} & 11,055 & 68 \\
 \texttt{w8a} & 49,749 & 300 \\
 \hline
\end{tabular}
\end{table}

\noindent As for the data distribution among the workers, we consider two scenarios:
\begin{itemize}
    \item \textbf{Homogeneous setting:} all the workers (regular and Byzantine) have access to the entire dataset (hence the gradients calculated by the good workers are all equal).
    \item \textbf{Heterogeneous setting:} the data is evenly distributed among good workers, without any uniform sampling or shuffling, so that each worker has access to approximately the same amount of data and there is no overlap. Since we make no assumptions on the behaviour of the bad workers, they have access to the full dataset.
\end{itemize}

\paragraph{Byzantine attacks and aggregation rule:}

In each setting, we consider $4$ different byzantine attacks:
\begin{itemize}
    \item \textbf{Bit Flipping (BF)}: Byzantine workers compute $-\nabla f(x)$ and send it to the server.
    \item \textbf{Label Flipping (LF)}: Byzantine workers compute their gradients using poisoned labels  (i.e., $y_i \rightarrow - y_i$)
    \item \textbf{A Little Is Enough (ALIE)} \citep{baruch2019little}: Byzantine workers compute empirical mean $\mu_G$ and standard deviation $\sigma_G$ of $\{g_i^t\}_{i\in G}$ and send $\mu_G - z\sigma_G$ to the server, where $z$ is a constant that controls the strength of the attack.
    \item \textbf{Inner Product Manipulation (IPM)} \citep{xie2020fall}: Byzantine workers send $-\frac{z}{G} \sum_{i\in \cG} \nabla f_i(x)$ to the server, where $z > 0$ is a constant that controls the strength of the attack.
\end{itemize}
The aggregation rule is the Coordinate-wise Median (CM) aggregator \citep{yin2018byzantine} with bucketing \citep{karimireddy2020byzantine} (see Appendix \ref{appendix:robust_aggr_and_compr}).

\paragraph{Parameters choice:}

For each method, we finetune the stepsize within the set $\{2^{k} \times \gamma_{th}\}_{k \in \mathbb{N}}$, where $\gamma_{th}$ is the theoretical stepsize of the algorithm under consideration (indicated by $1\times, 2\times, 4\times, \ldots$ in the plots). The momentum parameter $a$ and probability $p$ are chosen to be the optimal values predicted by theory, and the value of the constant $c$  from Definition \ref{def:RAgg_def} is determined using the formula from \citet{karimireddy2020byzantine}.
% \begin{figure}[H]
% \centering
% \begin{subfigure}{.33\textwidth}
%   \centering
%   \includegraphics[width=1\linewidth]{Fine-tuning_Byz-VR-MARINA_phishing.png}
%   \caption{\algname{Byz-VR-MARINA}}
% \end{subfigure}%
% \begin{subfigure}{.33\textwidth}
%   \centering
%   \includegraphics[width=1\linewidth]{Fine-tuning_Byz-VR-MARINA-No-Sync_phishing.png}
%   \caption{\algname{Byz-VR-MARINA 2.0}}
% \end{subfigure}
% \begin{subfigure}{.33\textwidth}
%   \centering
%   \includegraphics[width=1\linewidth]{Finetuning_DASHA_PAGE_phishing.png}
%   \caption{\algname{Byz-DASHA-PAGE}}
% \end{subfigure}
% \caption{%\centering
% Finetuning the algorithms stepsizes on Phishing dataset under the Bit flip (BF) attack in non-convex setting}
% \label{fig:finetuning_phishing}
% \end{figure}

% As shown in figure \ref{fig:finetuning_phishing}, for each attack, for each setting, we finetune the stepsize of each algorithm individually on exponential grids $\parens{2^{k}\gamma_{th}}_{k \in \mathbb{N}}$ where $\gamma_{th}$ is the theoretical stepsize predicted by our theory. For every experiment, we therefore choose for each algorithm the highest stepsize that yields the best convergence results.

In algorithms with stochastic gradients, uniform sampling with no replacement and batch size $b=0.01m$ is used.
In experiments comparing methods using unbiased compressors, we use the Rand$K$ compressor with $K=0.1d$. In error feedback experiments (Appendix  \ref{appendix:ef_exp}), \algname{Byz-VR-MARINA} and \algname{Byz-EF21} employ the Rand$K$ and Top$K$ sparsifiers, respectively, with $K=1$.

\subsubsection{Extra dataset}

\begin{figure}[H]
\centering
\begin{subfigure}{.24\textwidth}
  \centering
  \includegraphics[width=1\linewidth]{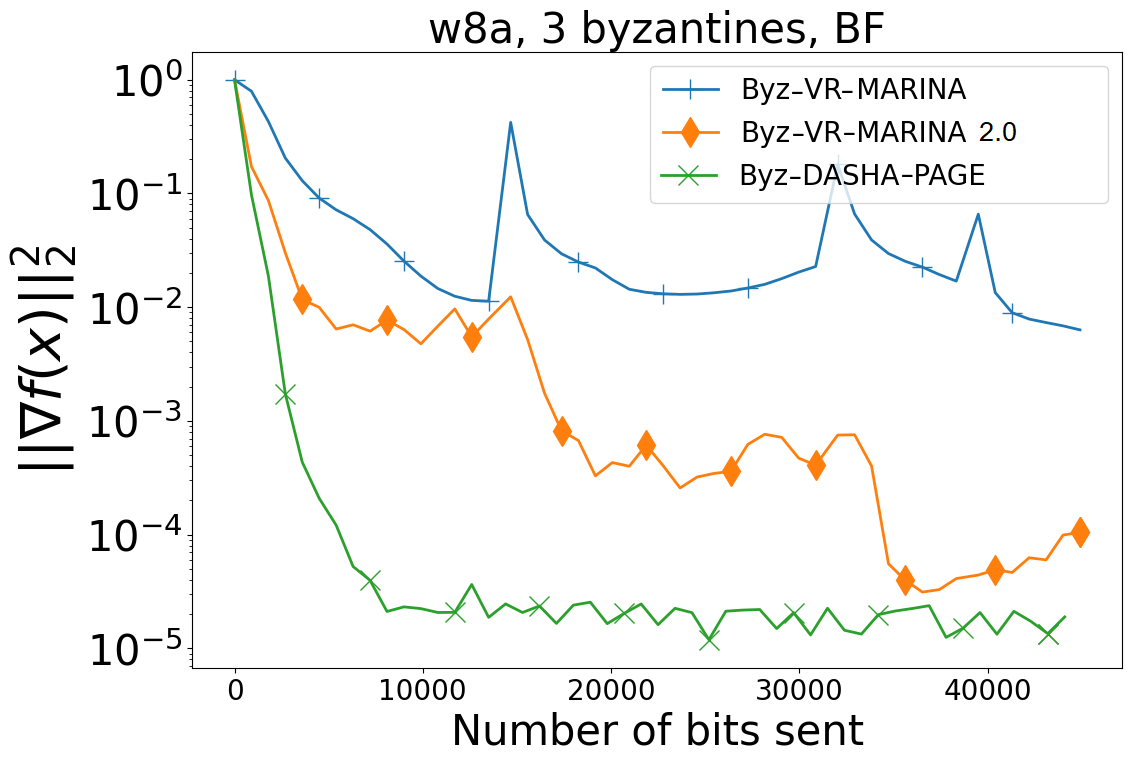}
  \caption{BF attack}
\end{subfigure}%
\begin{subfigure}{.24\textwidth}
  \centering
  \includegraphics[width=1\linewidth]{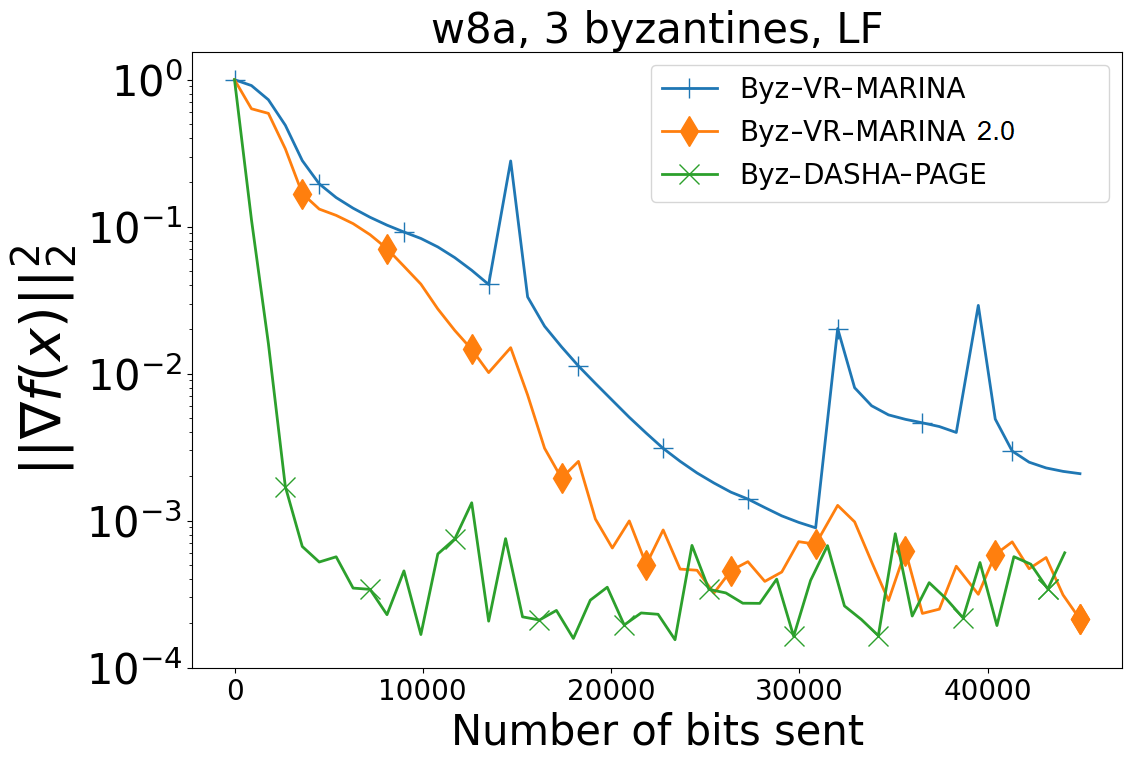}
  \caption{LF attack}
\end{subfigure}
\begin{subfigure}{.24\textwidth}
  \centering
  \includegraphics[width=1\linewidth]{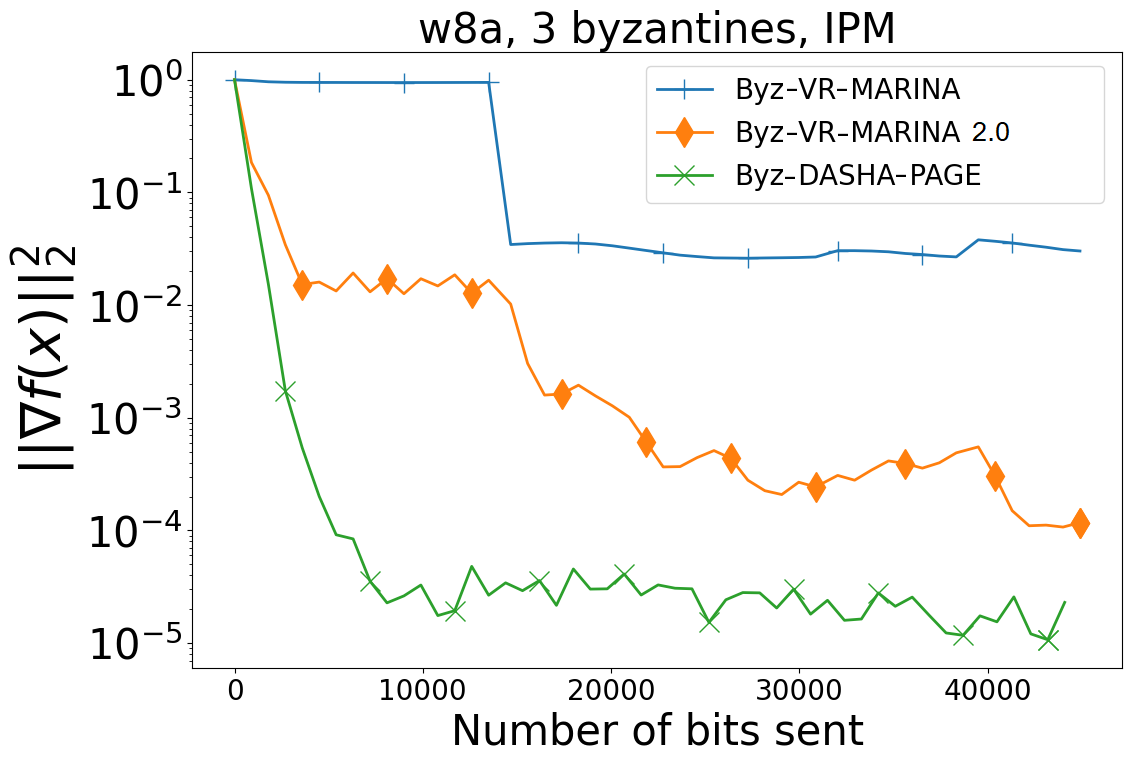}
  \caption{IPM attack}
\end{subfigure}
\begin{subfigure}{.24\textwidth}
  \centering
  \includegraphics[width=1\linewidth]{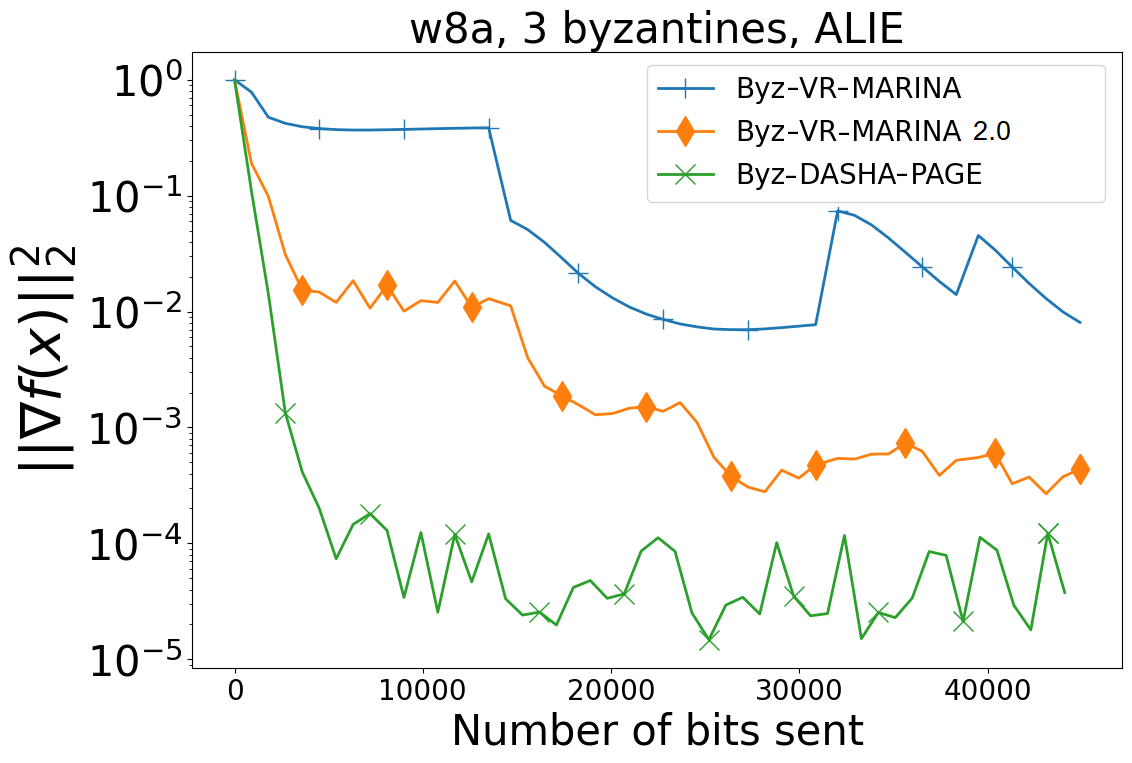}
  \caption{ALIE attack}
\end{subfigure}
\caption{%\centering
Communication complexity comparison in the heterogeneous non-convex setting on the \texttt{w8a} dataset.}
\label{fig:NC_w8a}
\end{figure}

Similarly to the results obtained with the \texttt{phishing} dataset, both \algname{Byz-VR-MARINA 2.0} and \algname{Byz-DASHA-PAGE} have superior performance compared to \algname{Byz-VR-MARINA} when tested on the \texttt{w8a} dataset (Figure~\ref{fig:NC_w8a}), across all types of attacks. Our methods converge faster and achieve better accuracy.

\subsubsection{Convergence under Polyak-Łojasiewicz condition}

\begin{figure}[H]
\centering
\begin{subfigure}{.24\textwidth}
  \centering
  \includegraphics[width=1\linewidth]{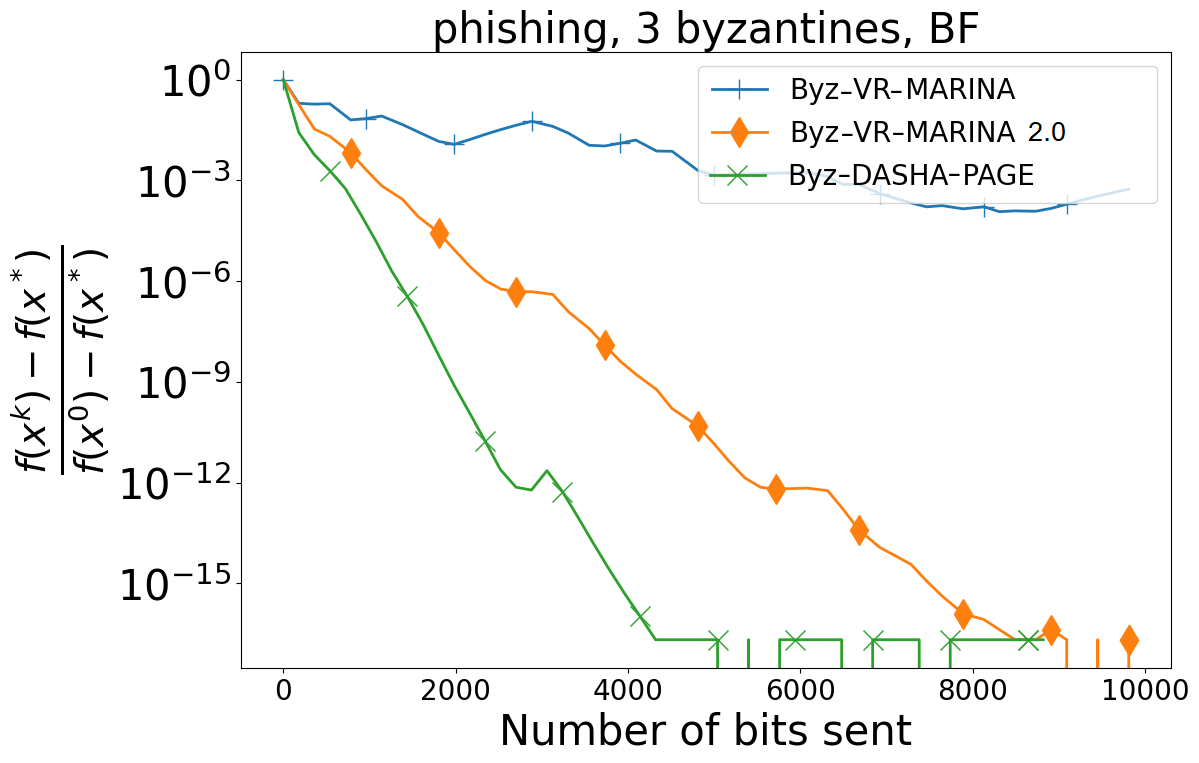}
  \caption{BF attack}
\end{subfigure}%
\begin{subfigure}{.24\textwidth}
  \centering
  \includegraphics[width=1\linewidth]{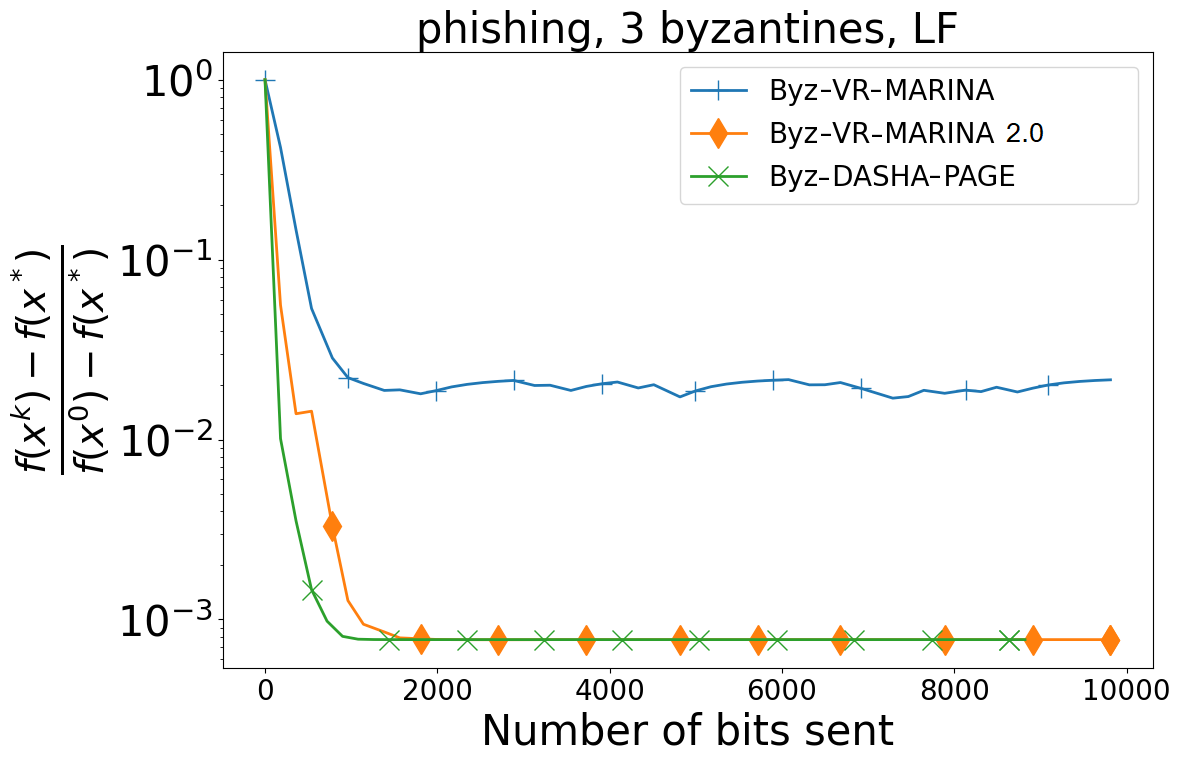}
  \caption{LF attack}
\end{subfigure}
\begin{subfigure}{.24\textwidth}
  \centering
  \includegraphics[width=1\linewidth]{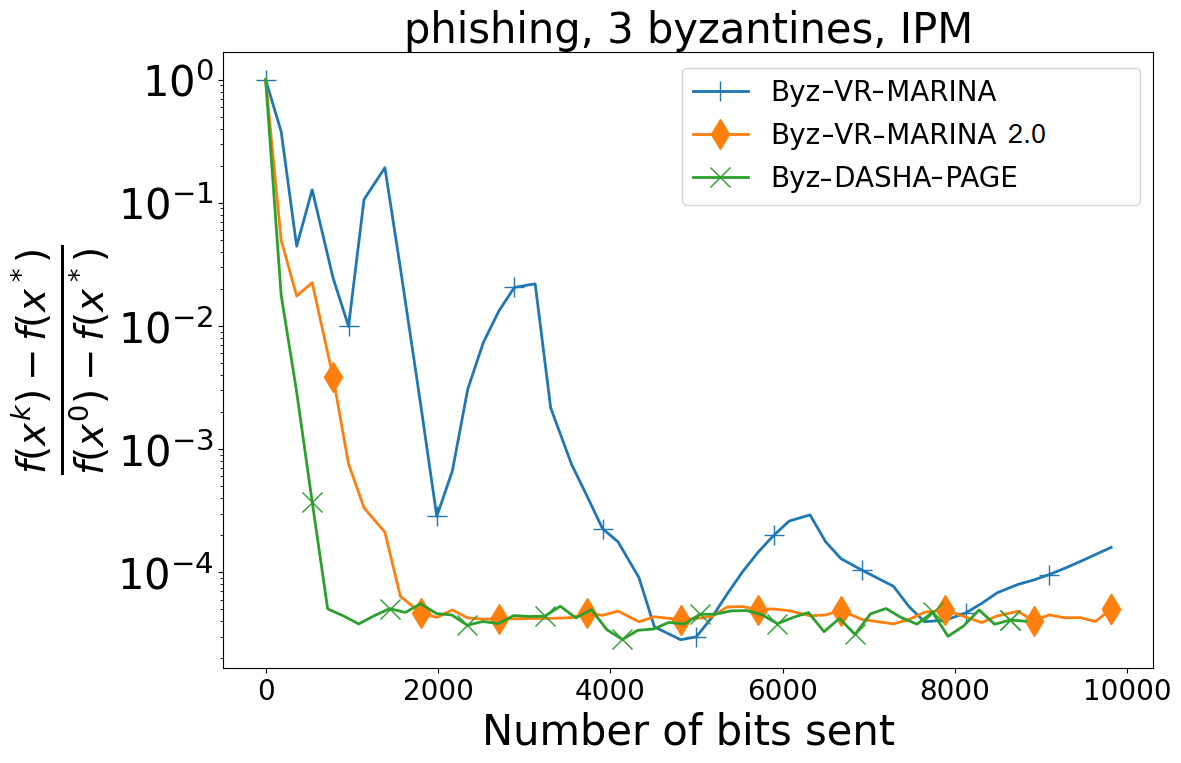}
  \caption{IPM attack}
\end{subfigure}
\begin{subfigure}{.24\textwidth}
  \centering
  \includegraphics[width=1\linewidth]{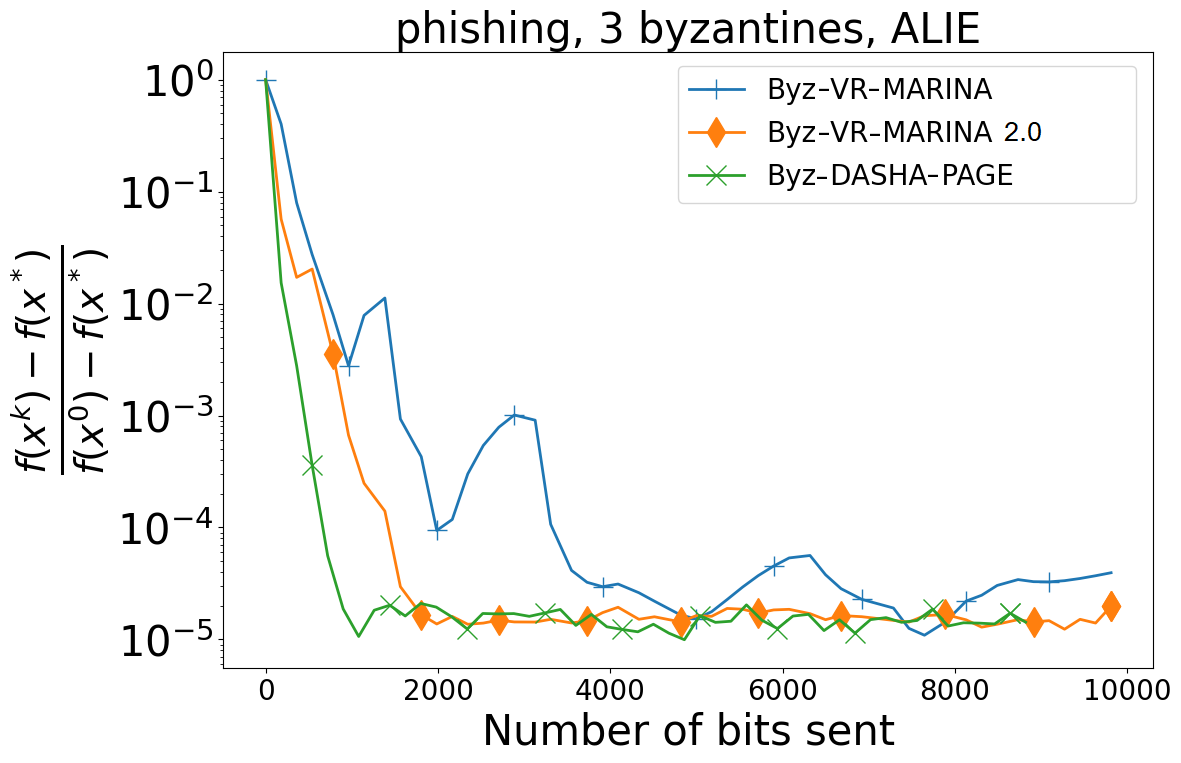}
  \caption{ALIE attack}
\end{subfigure}
\caption{%\centering
Communication complexity comparison in the heterogeneous strongly convex setting on the \texttt{phishing} dataset.}
\label{fig:pl1}
\end{figure}

\begin{figure}[H]
\centering
\begin{subfigure}{.24\textwidth}
  \centering
  \includegraphics[width=1\linewidth]{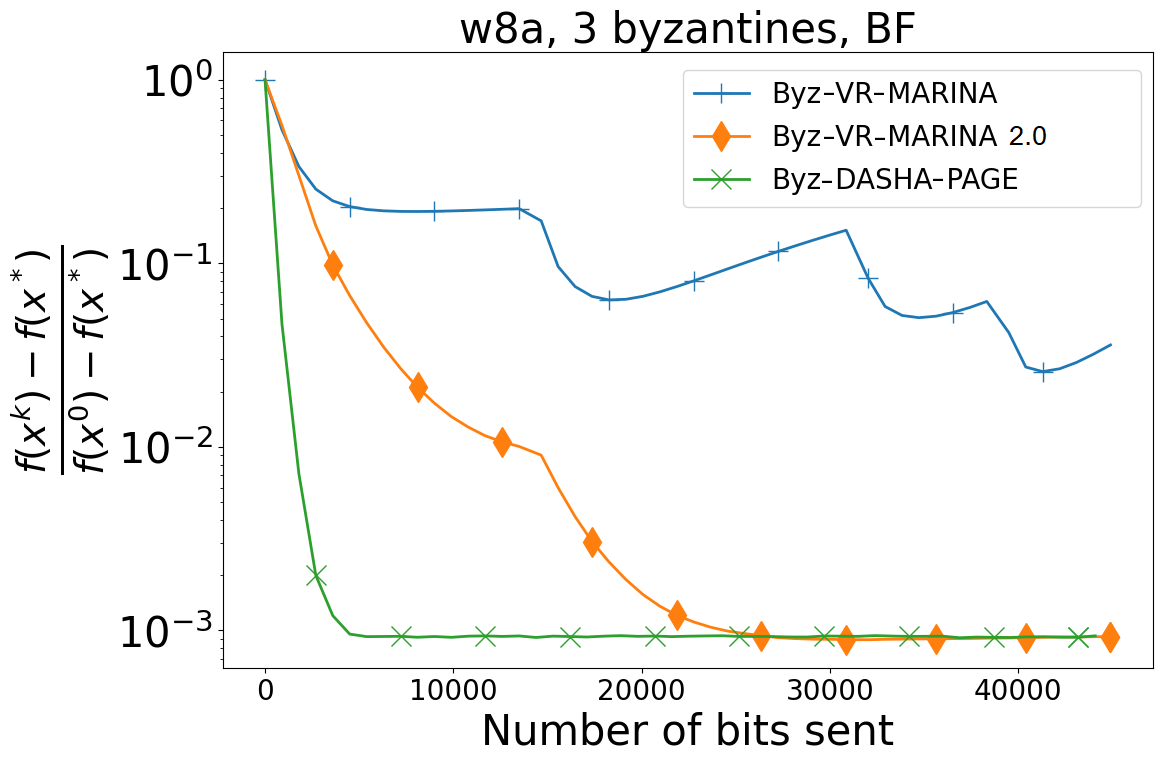}
  \caption{BF attack}
\end{subfigure}%
\begin{subfigure}{.24\textwidth}
  \centering
  \includegraphics[width=1\linewidth]{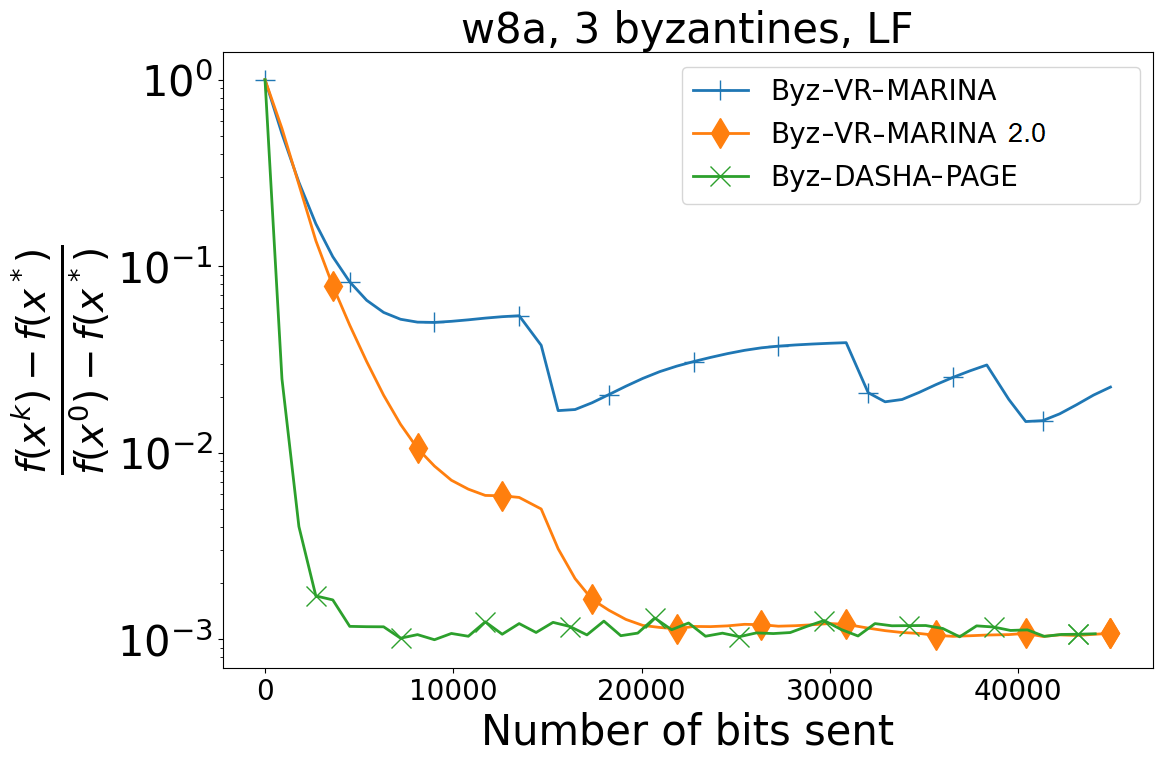}
  \caption{LF attack}
\end{subfigure}
\begin{subfigure}{.24\textwidth}
  \centering
  \includegraphics[width=1\linewidth]{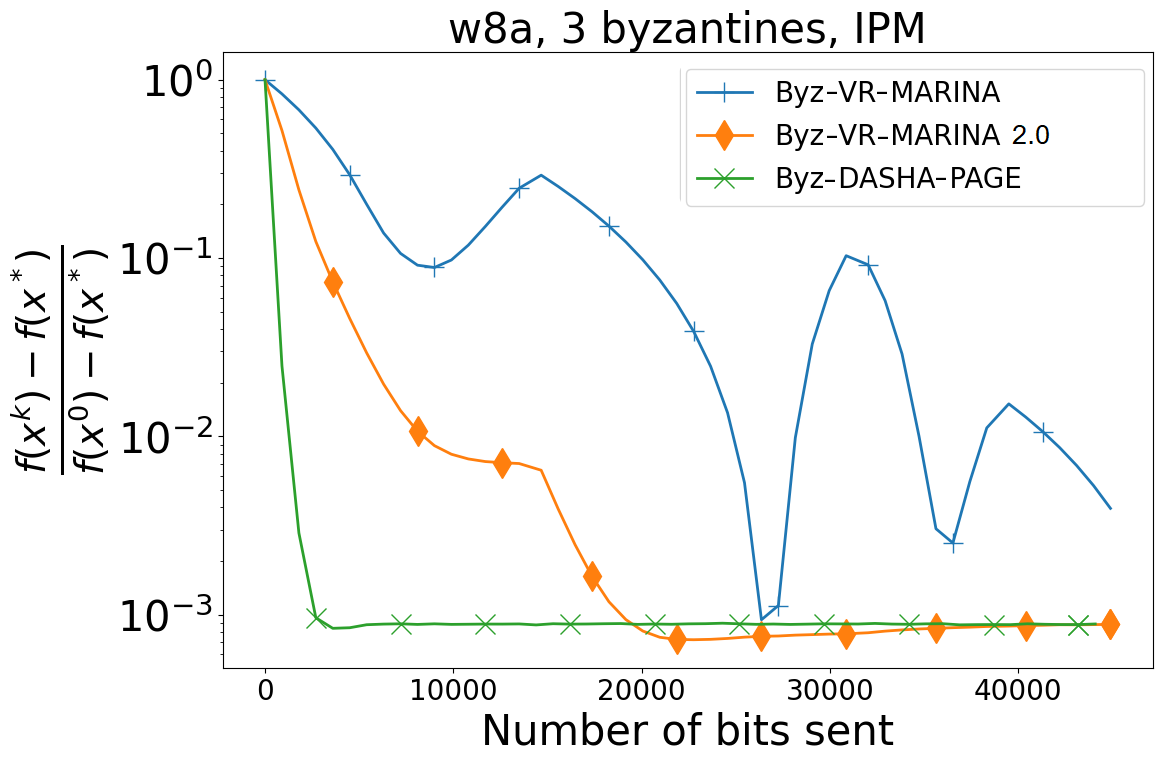}
  \caption{IPM attack}
\end{subfigure}
\begin{subfigure}{.24\textwidth}
  \centering
  \includegraphics[width=1\linewidth]{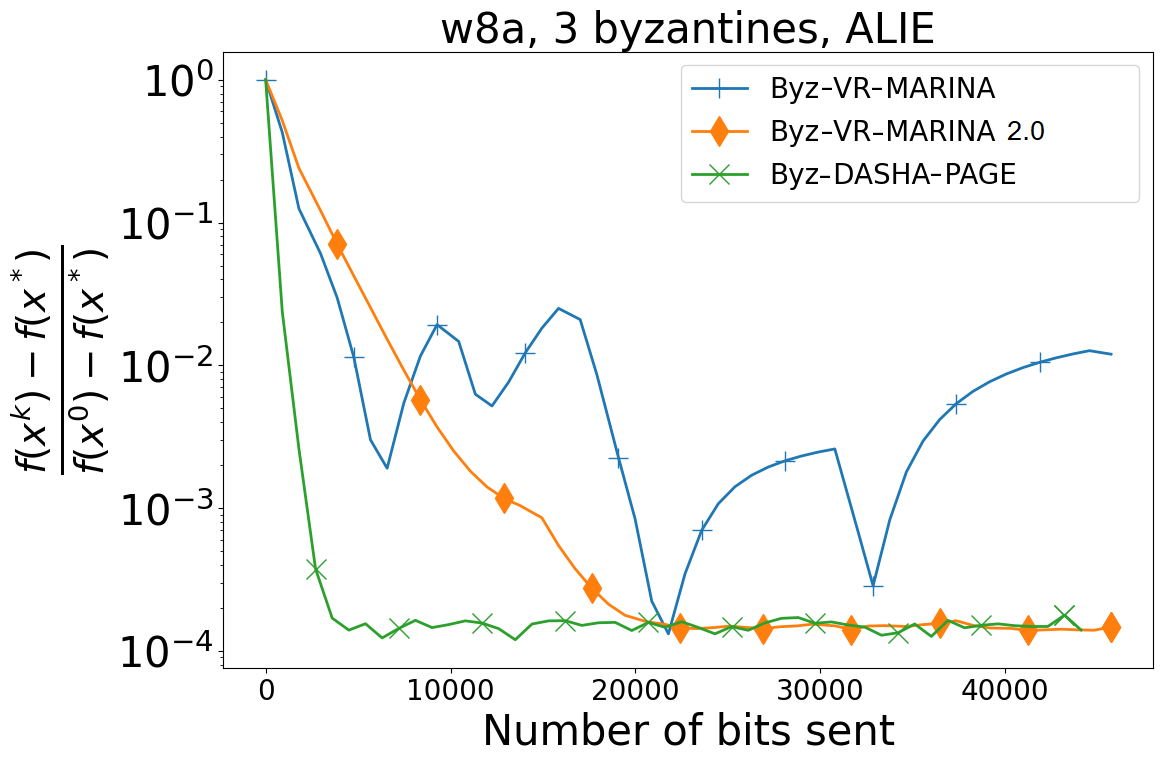}
  \caption{ALIE attack}
\end{subfigure}
\caption{%\centering
Communication complexity comparison in the heterogeneous strongly convex setting on the \texttt{w8a} dataset.}
\label{fig:pl2}
\end{figure}

As indicated by Theorems \ref{thm:general_pl_MARINA} and \ref{thm:general_pl_DASHA}, the experiments in the heterogeneous Polyak-{\L}ojasiewicz setting (Figures \ref{fig:pl1} and \ref{fig:pl2}) confirm that our methods outperform \algname{Byz-VR-MARINA}, converging faster, being more stable and achieving higher accuracy. The improvement is most pronounced in the case of BF and LF attacks.

\subsubsection{Error Feedback experiments}\label{appendix:ef_exp}

We next compare the empirical performance of \algname{Byz-EF21} and \algname{Byz-VR-MARINA} in the heterogeneous non-convex setting. To ensure fair comparison, full gradients are calculated.

\begin{figure}[H]
\centering
\begin{subfigure}{.24\textwidth}
  \centering
  \includegraphics[width=1\linewidth]{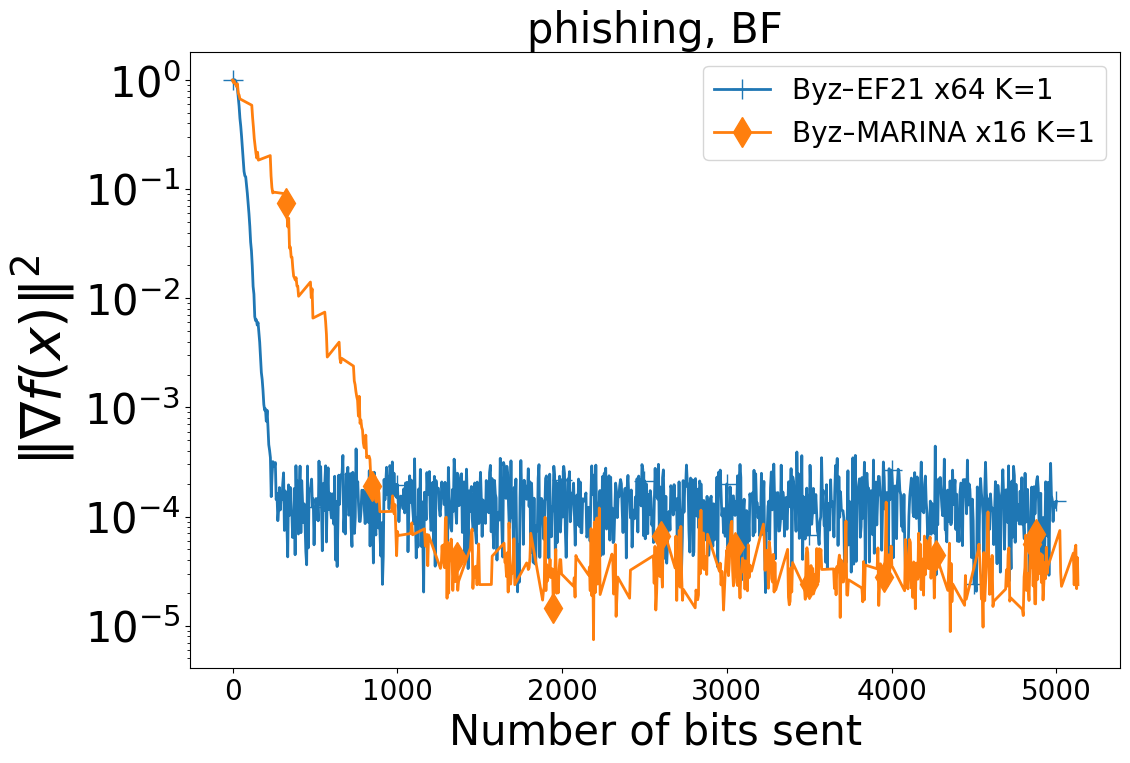}
  \caption{BF attack}
\end{subfigure}%
\begin{subfigure}{.24\textwidth}
  \centering
  \includegraphics[width=1\linewidth]{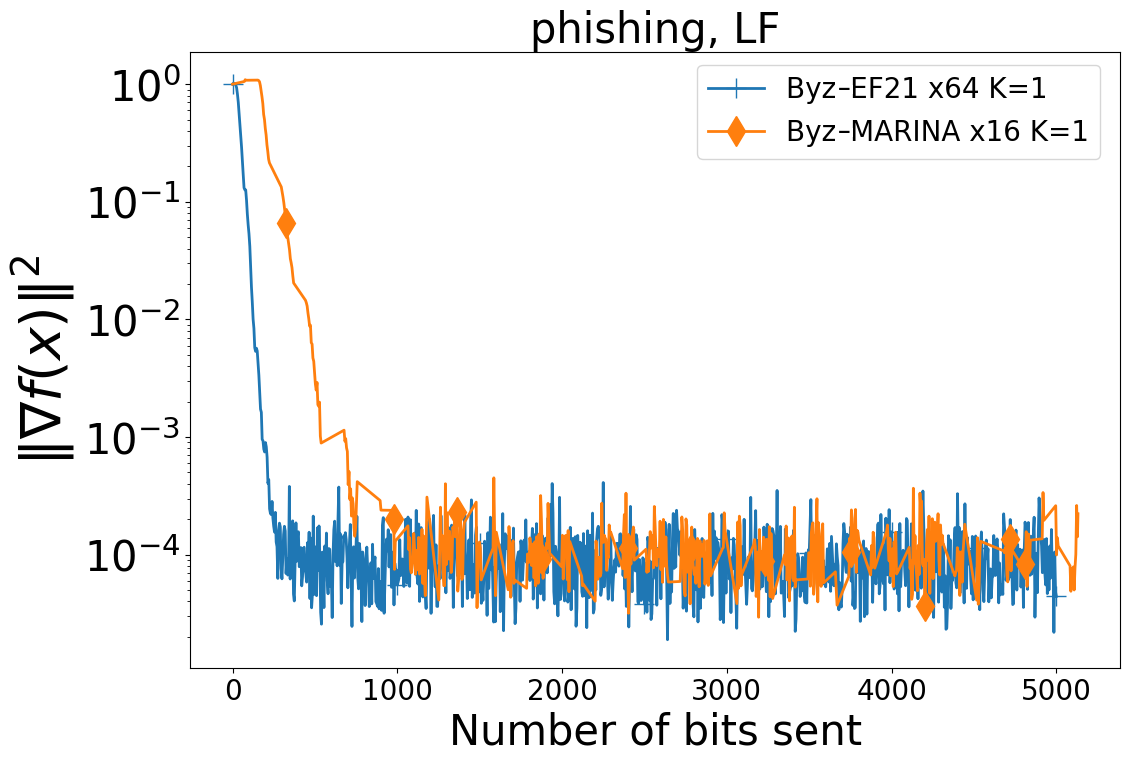}
  \caption{LF attack}
\end{subfigure}
\begin{subfigure}{.24\textwidth}
  \centering
  \includegraphics[width=1\linewidth]{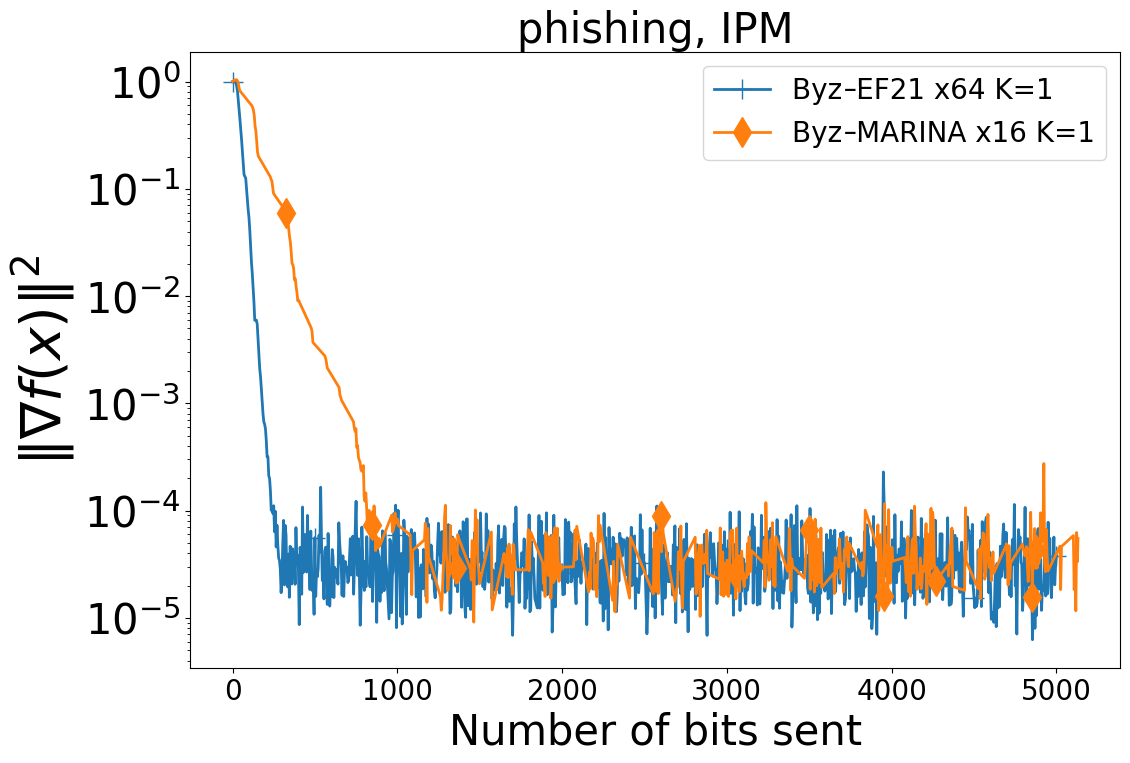}
  \caption{IPM attack}
\end{subfigure}
\begin{subfigure}{.24\textwidth}
  \centering
  \includegraphics[width=1\linewidth]{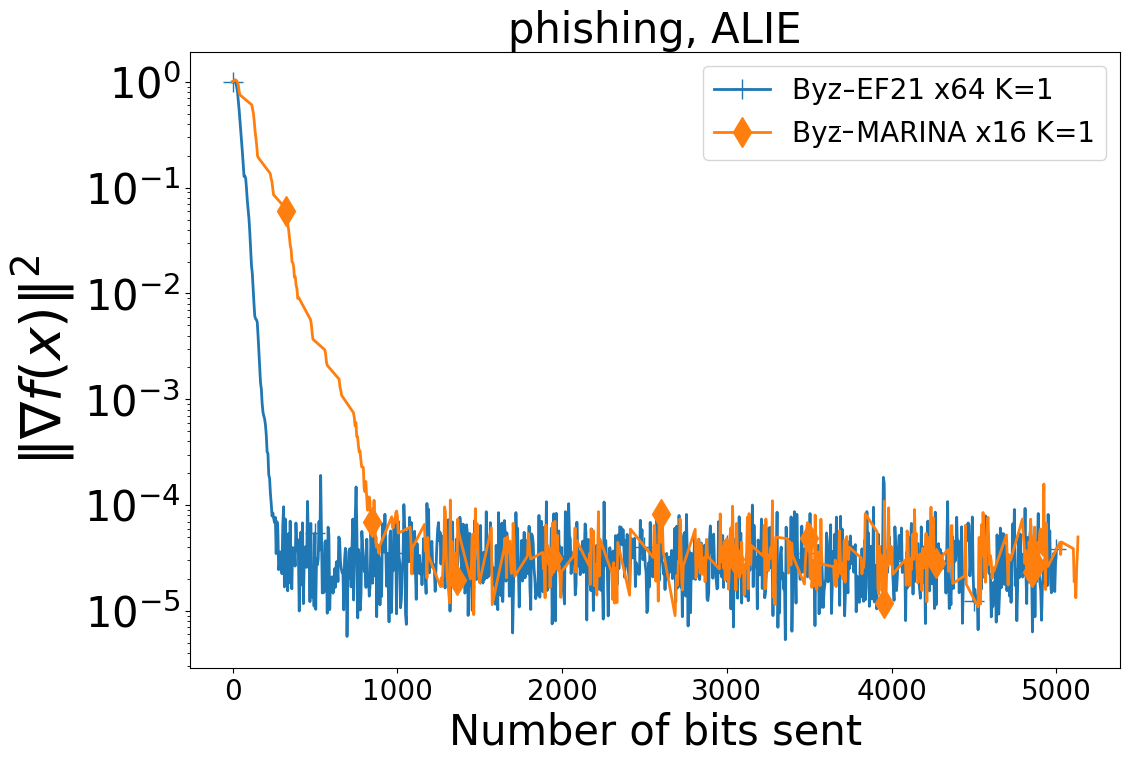}
  \caption{ALIE attack}
\end{subfigure}
\caption{%\centering 
Communication complexity comparison in the heterogeneous non-convex setting on the \texttt{phishing} dataset.}
\label{fig:comm_EF21_phishing}
\end{figure}

\begin{figure}[H]
\centering
\begin{subfigure}{.24\textwidth}
  \centering
  \includegraphics[width=1\linewidth]{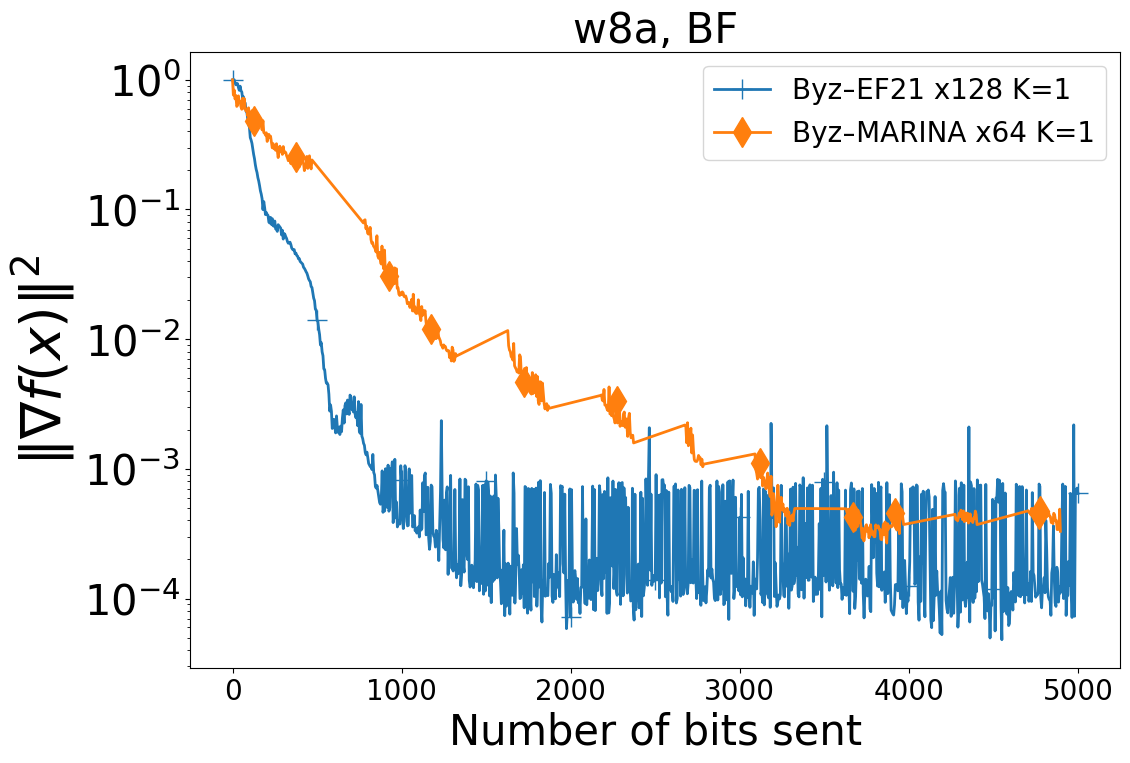}
  \caption{BF attack}
\end{subfigure}%
\begin{subfigure}{.24\textwidth}
  \centering
  \includegraphics[width=1\linewidth]{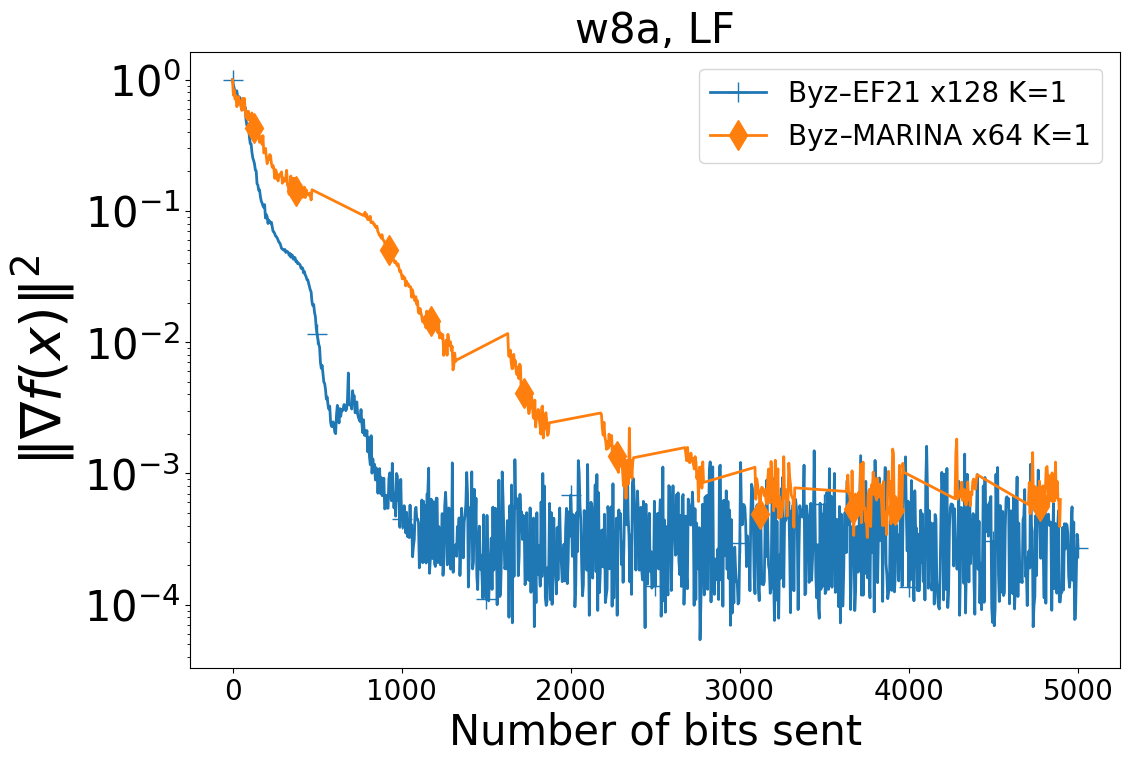}
  \caption{LF attack}
\end{subfigure}
\begin{subfigure}{.24\textwidth}
  \centering
  \includegraphics[width=1\linewidth]{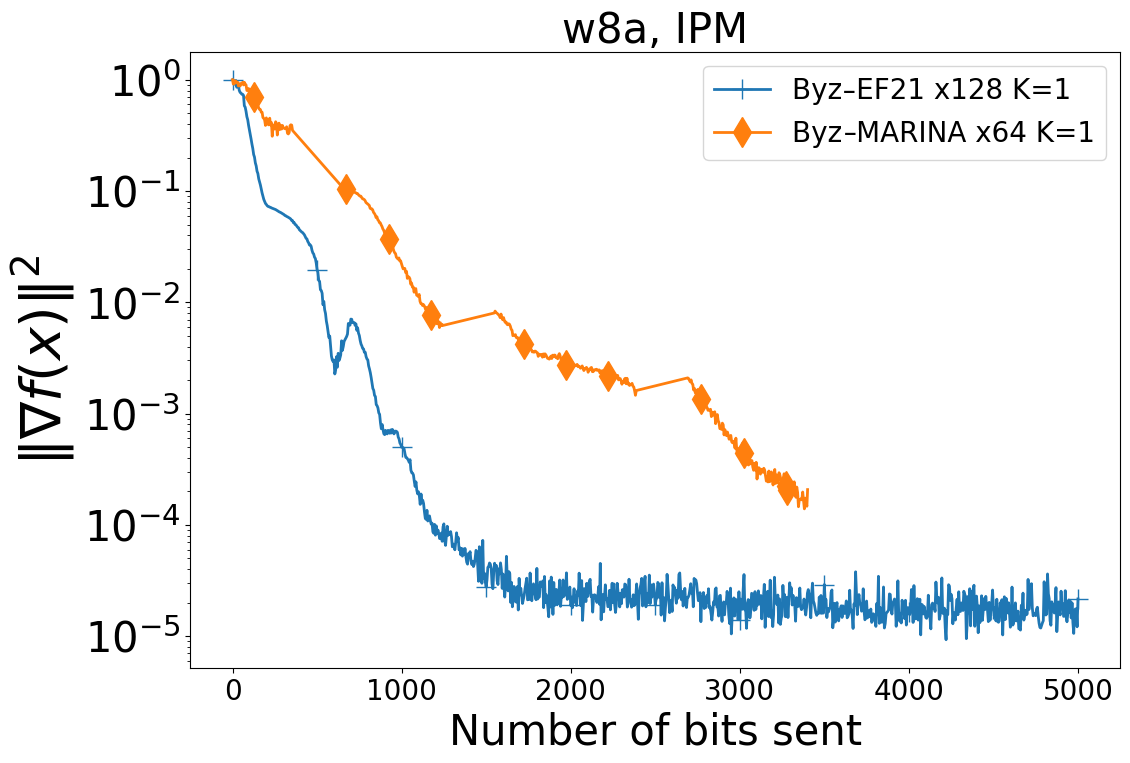}
  \caption{IPM attack}
\end{subfigure}
\begin{subfigure}{.24\textwidth}
  \centering
  \includegraphics[width=1\linewidth]{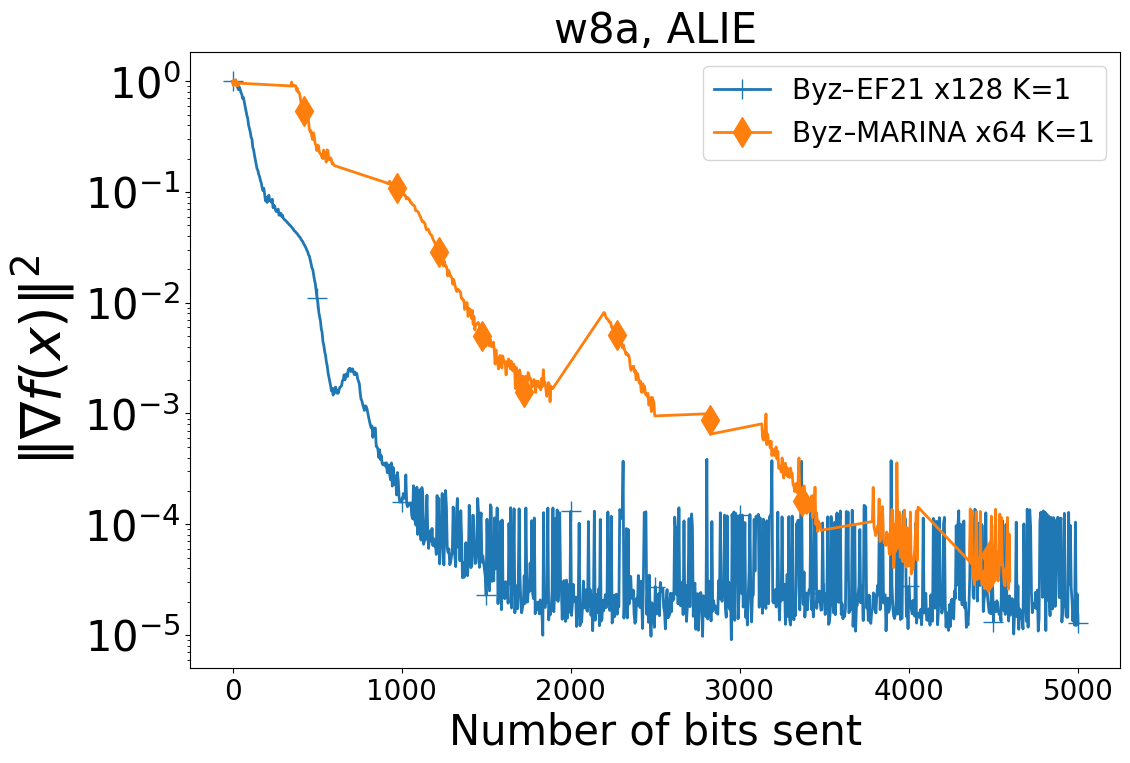}
  \caption{ALIE attack}
\end{subfigure}
\caption{%\centering
Communication complexity comparison in the heterogeneous non-convex setting on the \texttt{w8a} dataset.}
\label{fig:comm_EF21_w8a}
\end{figure}

The experiments unveil promising potential for communication improvement. As shown in Figures~\ref{fig:comm_EF21_phishing} and \ref{fig:comm_EF21_w8a}, \algname{Byz-EF21} converges faster than \algname{Byz-VR-MARINA}, before both reach a point of stagnation.

\subsection{Neighborhood size}

We next compare the accuracy of \algname{Byz-DASHA-PAGE} and \algname{Byz-VR-MARINA}. The data is distributed among $28$ clients, out of which $8$ are Byzantine. The $20$ good clients are divided into two groups, $\cG_1$ and $\cG_2$, of equal size. The local functions calculated by the clients are $f_i(x) = \frac{1}{2}\norm{x}^2 + \langle \zeta_1, x \rangle$ for $i\in \cG_1 $ and $f_i(x) = \frac{1}{2}\norm{x}^2 + \langle \zeta_2, x \rangle$ for $i\in \cG_2 $ where $\zeta_1, \zeta_2 \in \R^{d}$ with $d=100$. The Byzantines mimic the behavior of group $\cG_1$. $3$ different aggregation rules are considered: standard averaging, CM aggregator and GM aggregator (see Appendix \ref{appendix:robust_aggr_and_compr}).
In both algorithms, we use the Rand$K$ compressor with $K=5$ and stepsize equal to $10$ times the theoretical one. The results, presented in Figure \ref{fig:neighbourhood_exp}, not only show that \algname{Byz-DASHA-PAGE} converges faster, but it is also more stable, converging to a smaller neighborhood than \algname{Byz-VR-MARINA}.

\begin{figure}[H]
\centering
\begin{subfigure}{.33\textwidth}
  \centering
  \includegraphics[width=1\linewidth]{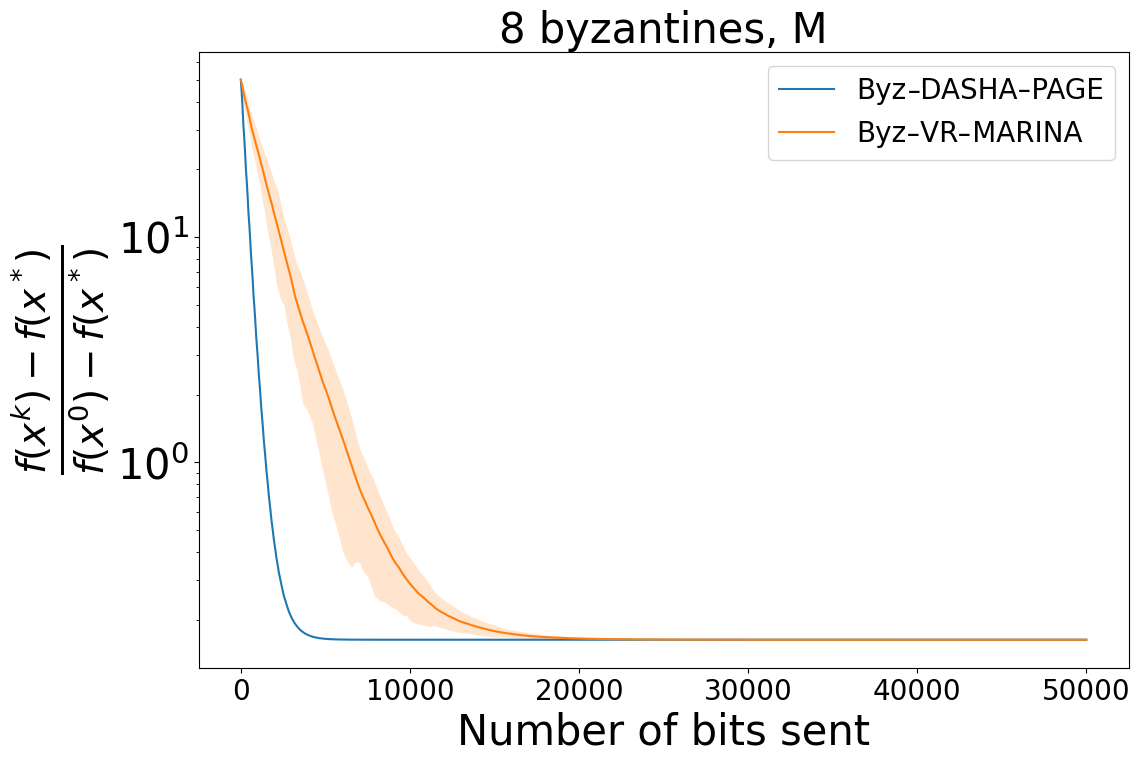}
  \caption{Mean aggregator}
\end{subfigure}%
\begin{subfigure}{.33\textwidth}
  \centering
  \includegraphics[width=1\linewidth]{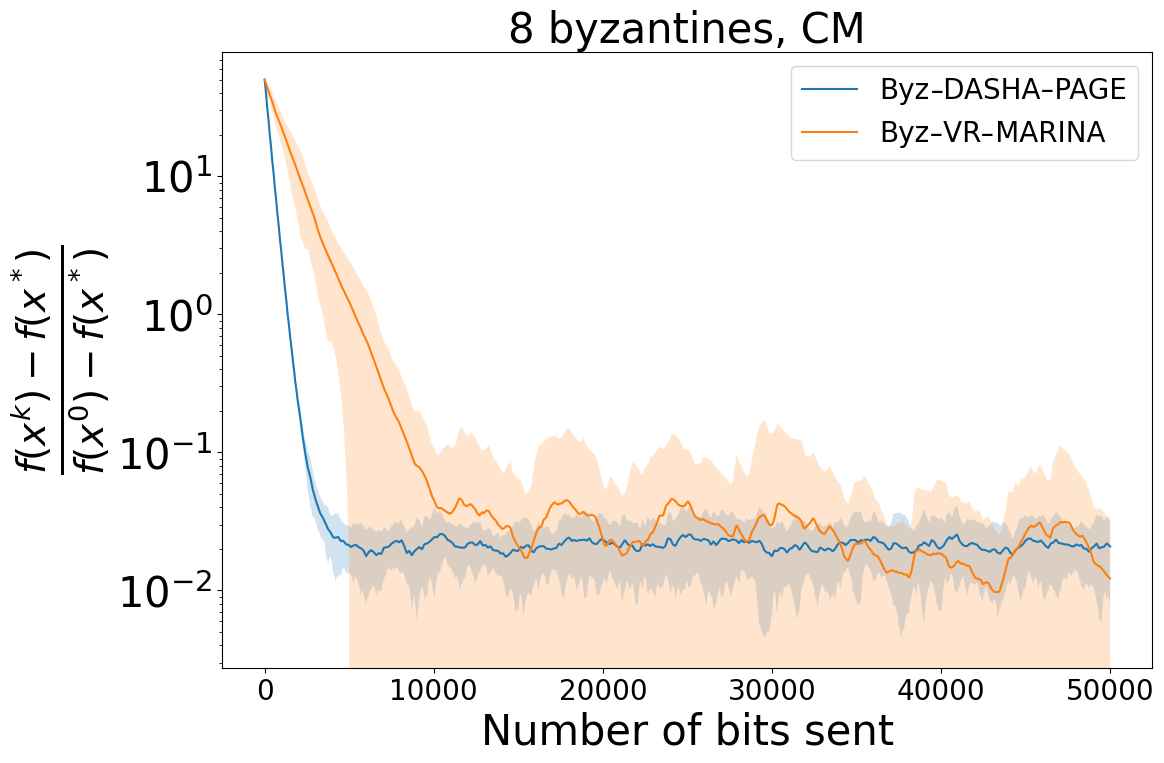}
  \caption{CM aggregator}
\end{subfigure}
\begin{subfigure}{.33\textwidth}
  \centering
  \includegraphics[width=1\linewidth]{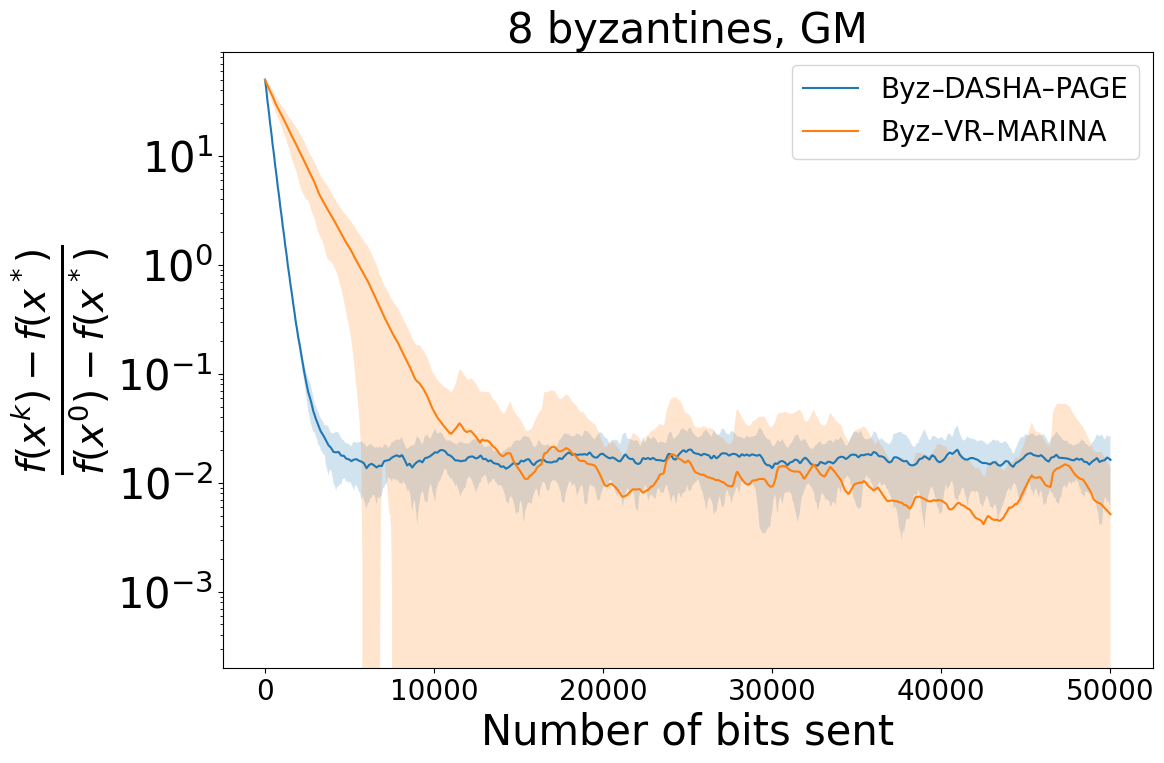}
  \caption{GM aggregator}
\end{subfigure}
\caption{%\centering
The mean optimality gap of \algname{Byz-DASHA-PAGE} and \algname{Byz-VR-MARINA} for $3$ aggregation rules calculated from $15$ runs of both algorithms. The shadowed area corresponds to one standard deviation.}
\label{fig:neighbourhood_exp}
\end{figure}

\subsection{Numerical Comparison of Theoretical Stepsizes}

Table \ref{table:stepsize_comparison} gives the comparison of theoretical stepsizes calculated for the \texttt{phishing} dataset.
\begin{table}[H]
\centering
\caption{Comparison of theoretical stepsizes for \texttt{phishing} dataset.}
\label{table:stepsize_comparison}
\begin{tabular}{c c c} 
 \hline
 \algname{Byz-VR-MARINA} & \algname{Byz-VR-MARINA 2.0} & \algname{Byz-DASHA-PAGE} \\
 \hline
 4e-4 & 2e-2 & 1.2e-2 \\
 \hline
\end{tabular}
\end{table}

The theory behind \algname{Byz-VR-MARINA 2.0} and \algname{Byz-DASHA-PAGE} allows for significantly larger stepsize than \algname{Byz-VR-MARINA}, as a result of which our methods require fewer communication rounds. 
This improvement also shows in experiments, where \algname{Byz-VR-MARINA 2.0} and \algname{Byz-DASHA-PAGE} tolerate much larger stepsizes, which makes them more efficient in practice.

\subsection{Convergence under Additional Attacks}
In this subsection, we compare the convergence of \algname{Byz-VR-MARINA}, \algname{Byz-VR-MARINA 2.0}, and \algname{Byz-DASHA-PAGE} on the non-convex logistic regression problem under two additional attacks: Local Model Poisoning (LMP) attack \citep{fang2020local} and Relocated Orthogonal Perturbation (ROP) attack \citep{ozfatura2023byzantines} adjusted to the considered methods. The results are given in Figure~\ref{fig:NC_additional_attacks}. Similarly to the results under previously considered attacks, \algname{Byz-VR-MARINA 2.0} and \algname{Byz-DASHA-PAGE} outperform \algname{Byz-VR-MARINA} under LMP and ROP attacks.

\begin{figure}[H]
\centering
\begin{subfigure}{.24\textwidth}
  \centering
  \includegraphics[width=1\linewidth]{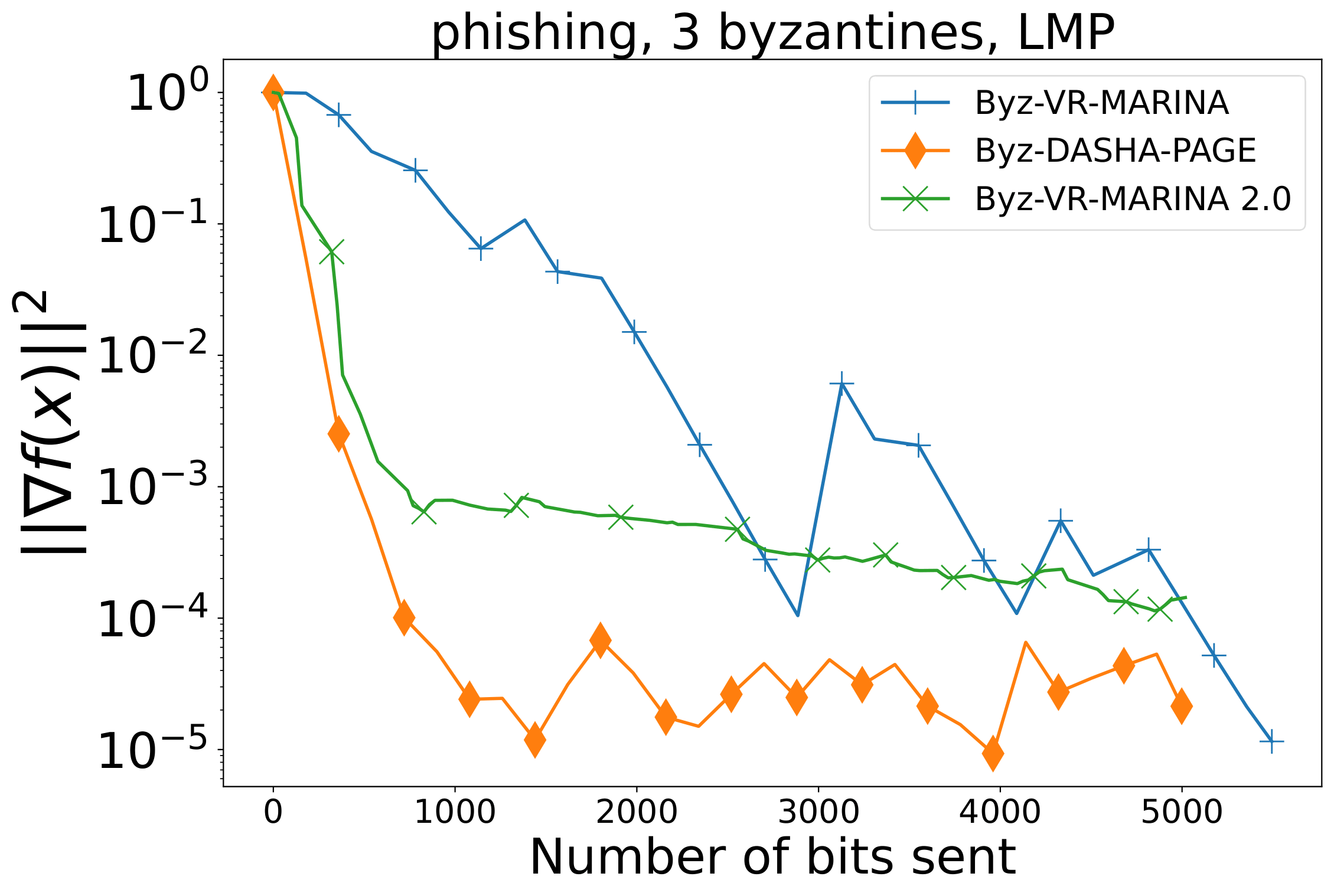}
  \caption{LMP attack, \texttt{phishing}}
\end{subfigure}%
\begin{subfigure}{.24\textwidth}
  \centering
  \includegraphics[width=1\linewidth]{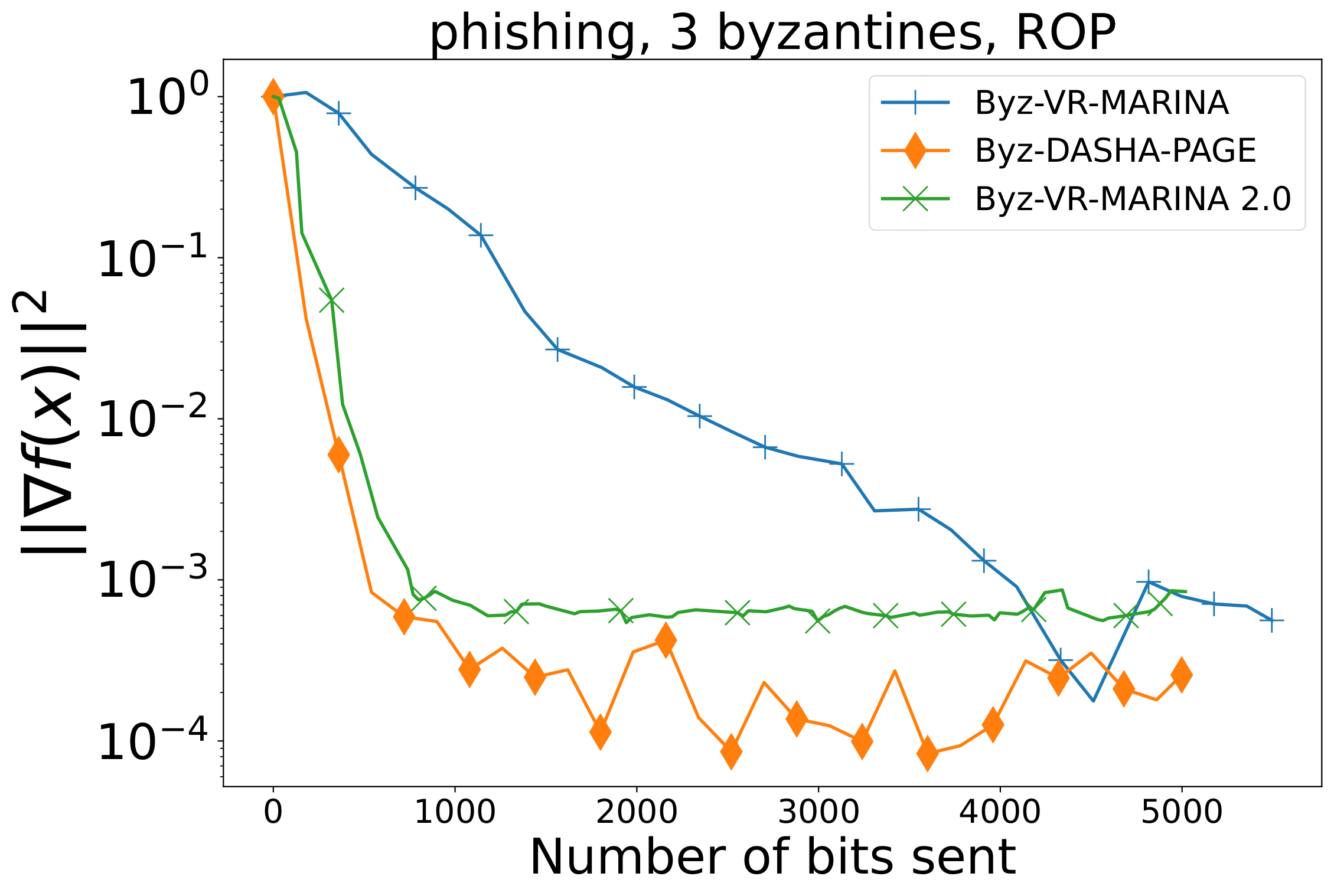}
  \caption{ROP attack, \texttt{phishing}}
\end{subfigure}
\begin{subfigure}{.24\textwidth}
  \centering
  \includegraphics[width=1\linewidth]{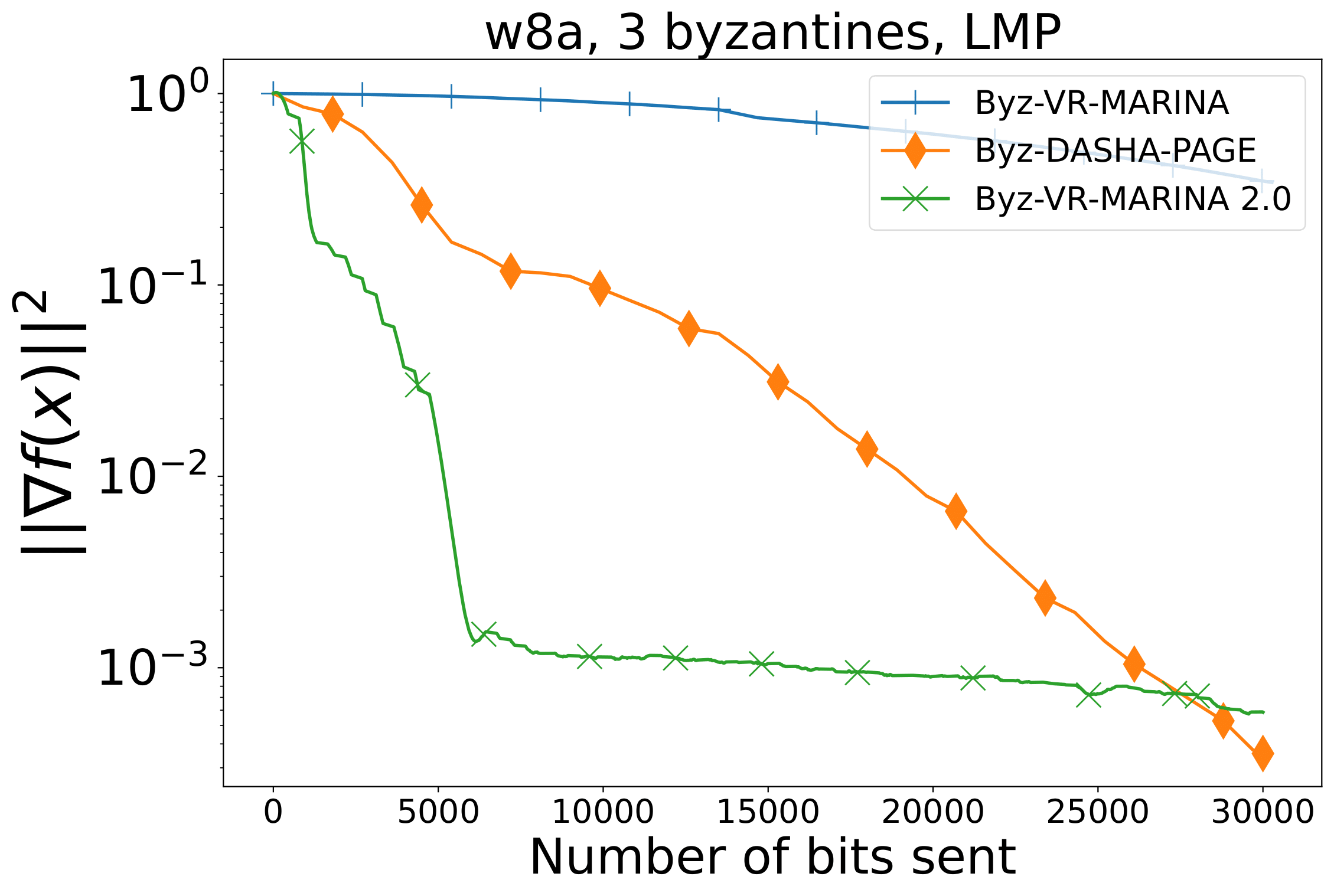}
  \caption{LMP attack, \texttt{phishing}}
\end{subfigure}%
\begin{subfigure}{.24\textwidth}
  \centering
  \includegraphics[width=1\linewidth]{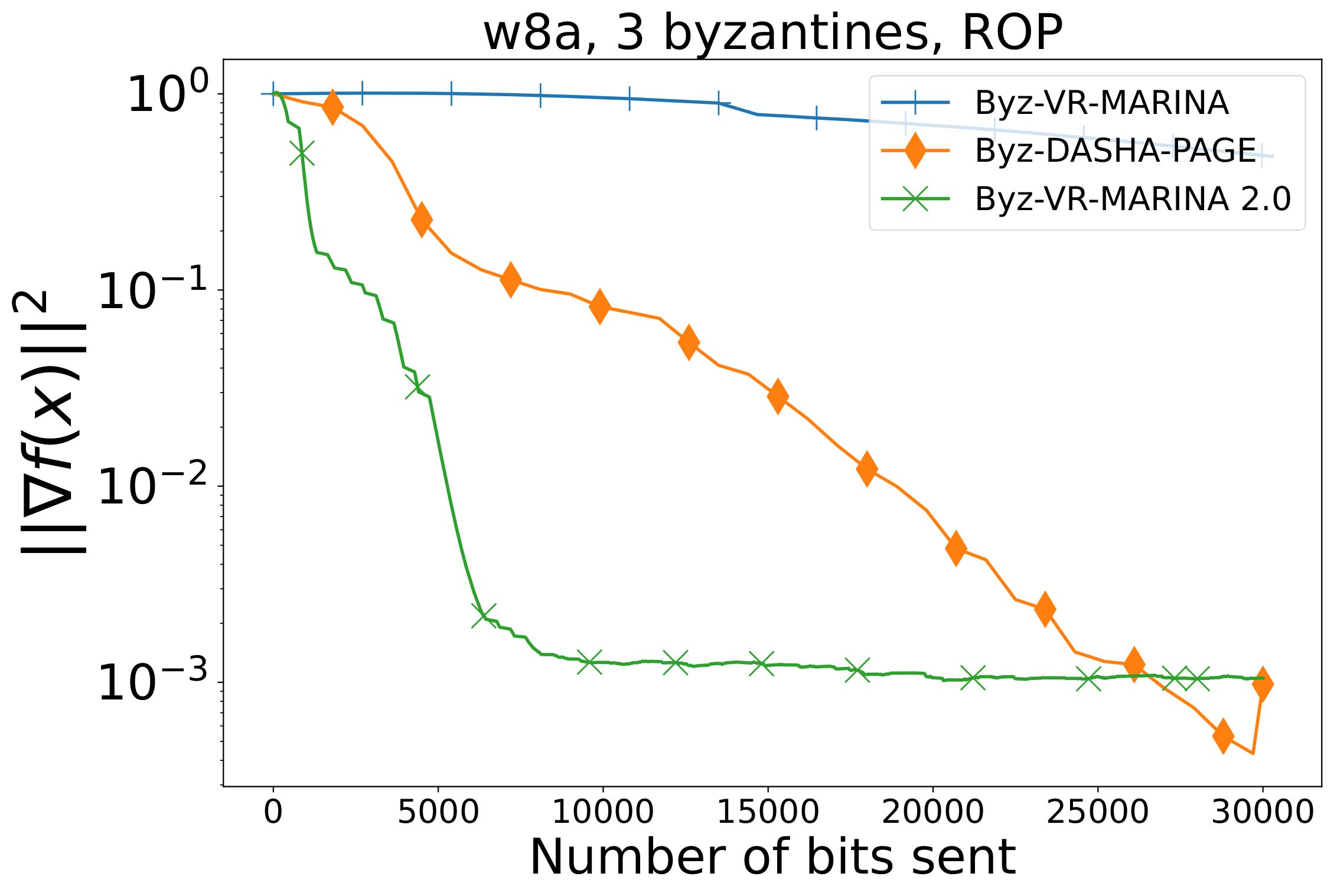}
  \caption{ROP attack, \texttt{phishing}}
\end{subfigure}
\caption{%\centering
Communication complexity comparison in the heterogeneous non-convex setting on the \texttt{phishing} and \texttt{w8a} datasets under LMP and ROP attacks.}
\label{fig:NC_additional_attacks}
\end{figure}

\subsection{Biased Compression vs Unbiased Compression}

In this subsection, we compare \algname{Byz-EF21} and \algname{Byz-VR-MARINA}, \algname{Byz-VR-MARINA 2.0}, \algname{Byz-DASHA-PAGE} with full gradient computations, i.e., we provide the comparison of the behavior of the methods with biased and unbiased compression. As in the previous subsection, we consider logistic regression problem with non-convex regularization. The results are given in Figure~\ref{fig:NC_biased_vs_unbiased}. In general, we see that the new methods are more robust than \algname{Byz-VR-MARINA}. However, in the conducted experiments, there is no clear ``champion'', e.g., under IPM, LMP and ROP attacks for \texttt{phishing} dataset \algname{Byz-EF21} outperforms \algname{Byz-DASHA-PAGE} and \algname{Byz-VR-MARINA 2.0}, but \algname{Byz-DASHA-PAGE} works noticeably better than \algname{Byz-EF21} under all attacks for \texttt{w8a} dataset.

\begin{figure}[H]
\centering
\begin{subfigure}{.33\textwidth}
  \centering
  \includegraphics[width=1\linewidth]{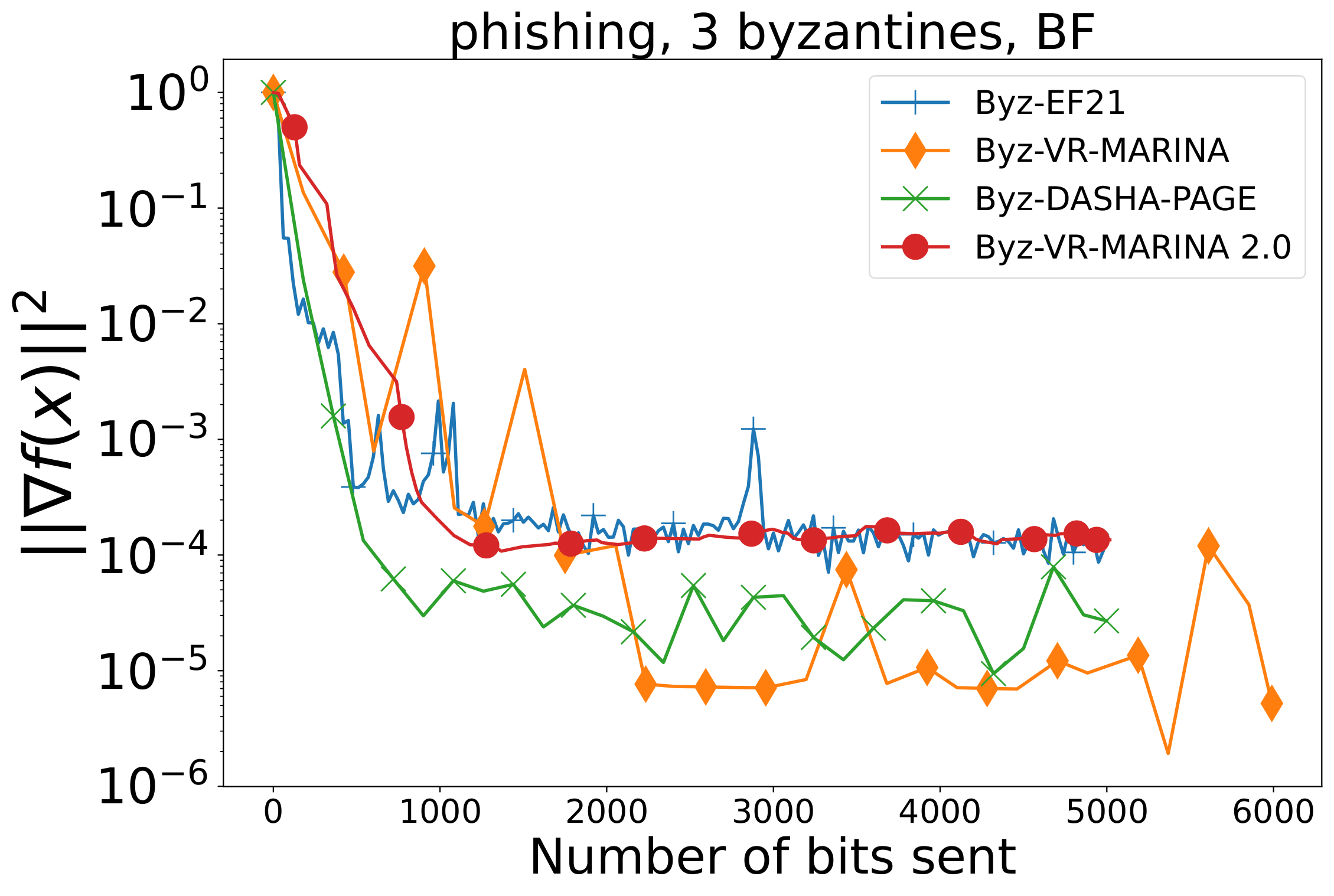}
  \caption{BF attack, \texttt{phishing}}
\end{subfigure}%
\begin{subfigure}{.33\textwidth}
  \centering
  \includegraphics[width=1\linewidth]{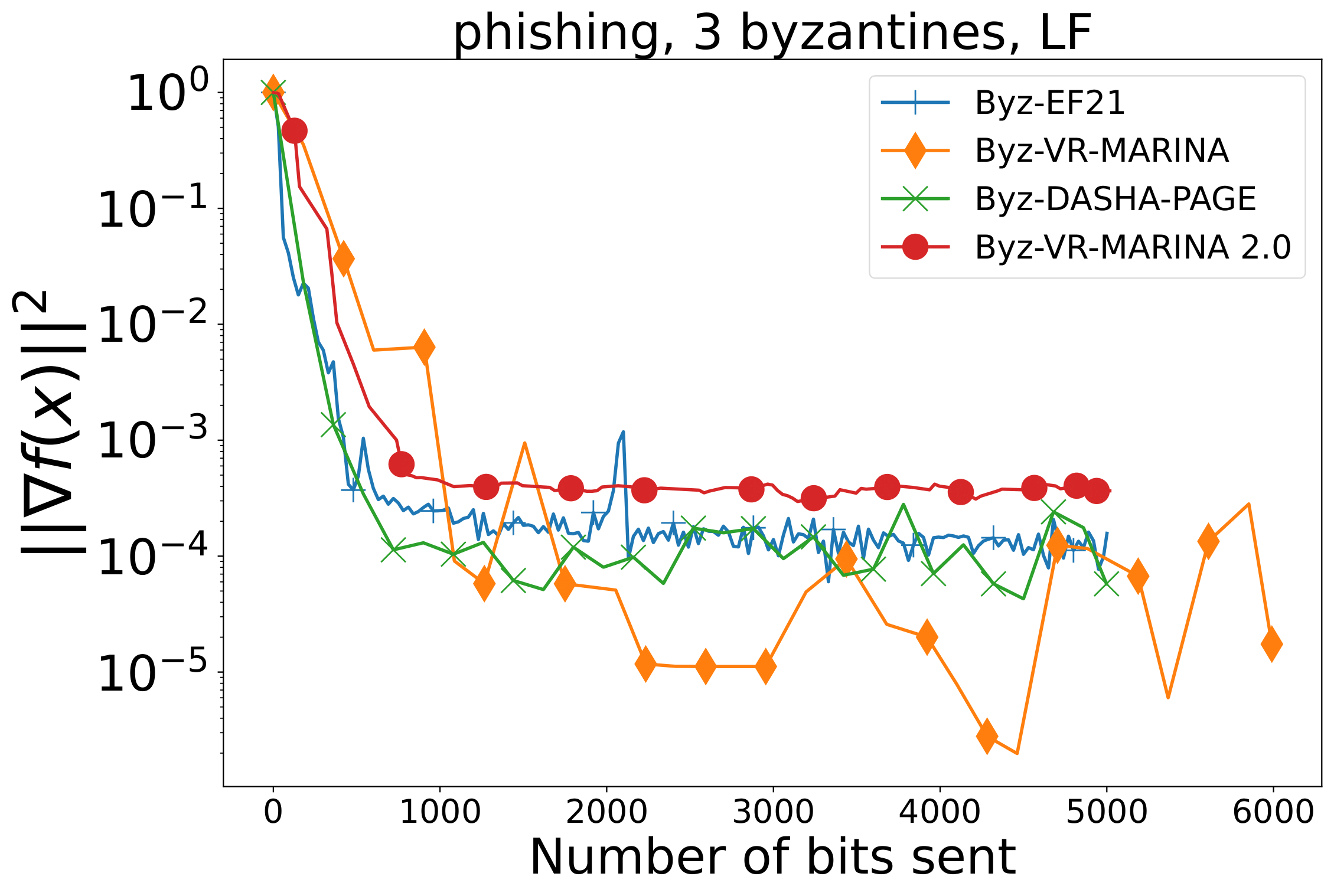}
  \caption{LF attack, \texttt{phishing}}
\end{subfigure}%
\begin{subfigure}{.33\textwidth}
  \centering
  \includegraphics[width=1\linewidth]{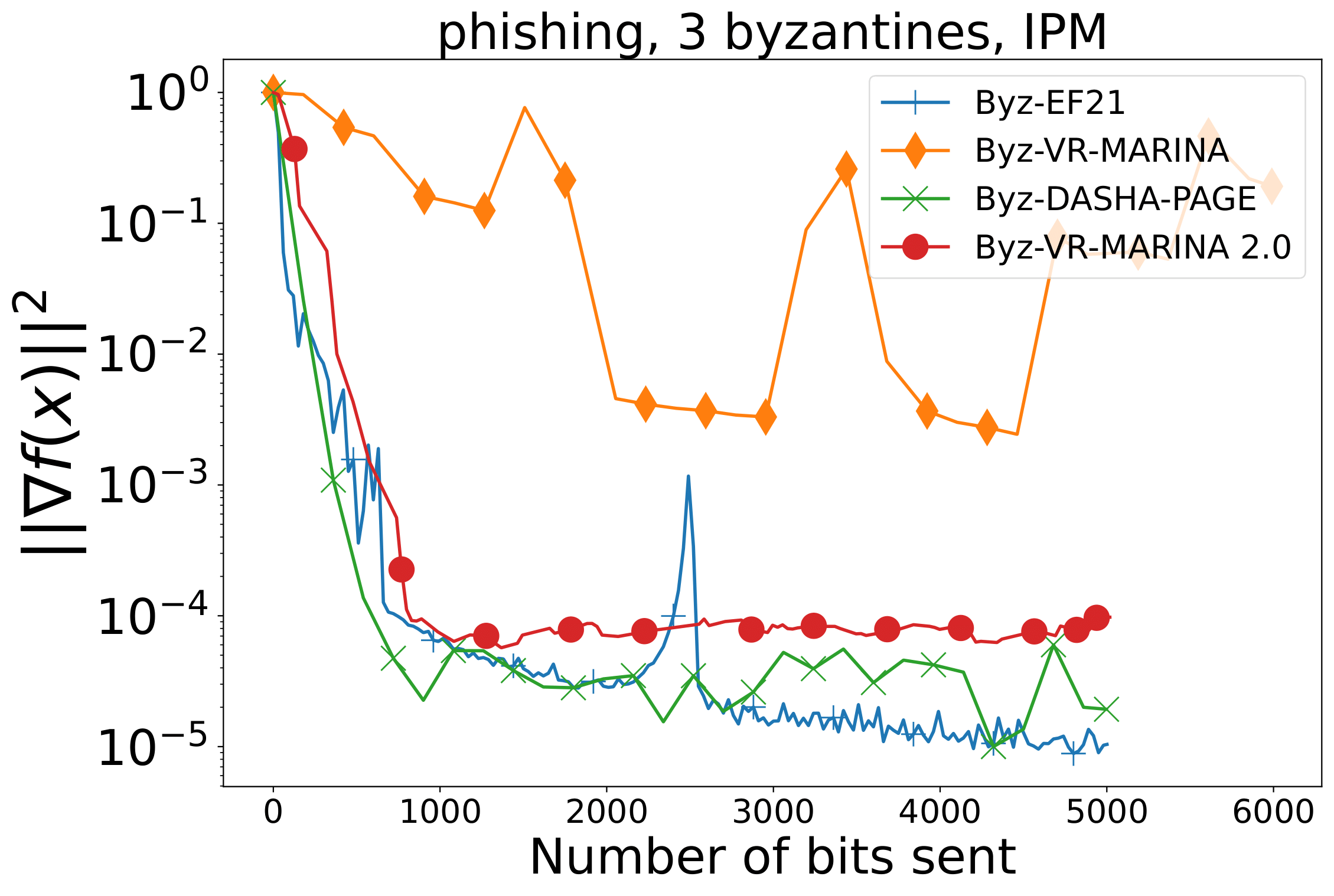}
  \caption{IPM attack, \texttt{phishing}}
\end{subfigure}\\
\begin{subfigure}{.33\textwidth}
  \centering
  \includegraphics[width=1\linewidth]{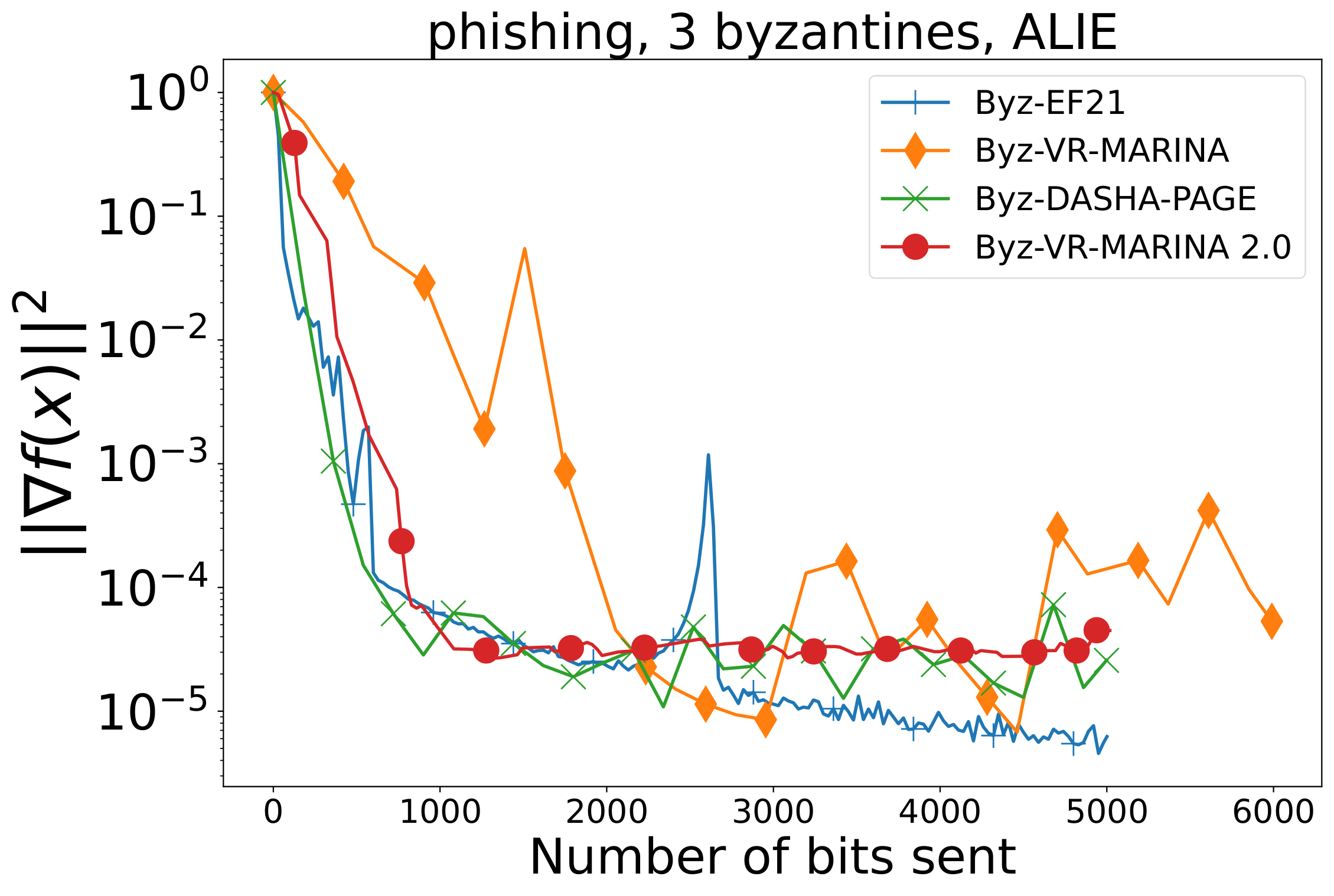}
  \caption{ALIE attack, \texttt{phishing}}
\end{subfigure}%
\begin{subfigure}{.33\textwidth}
  \centering
  \includegraphics[width=1\linewidth]{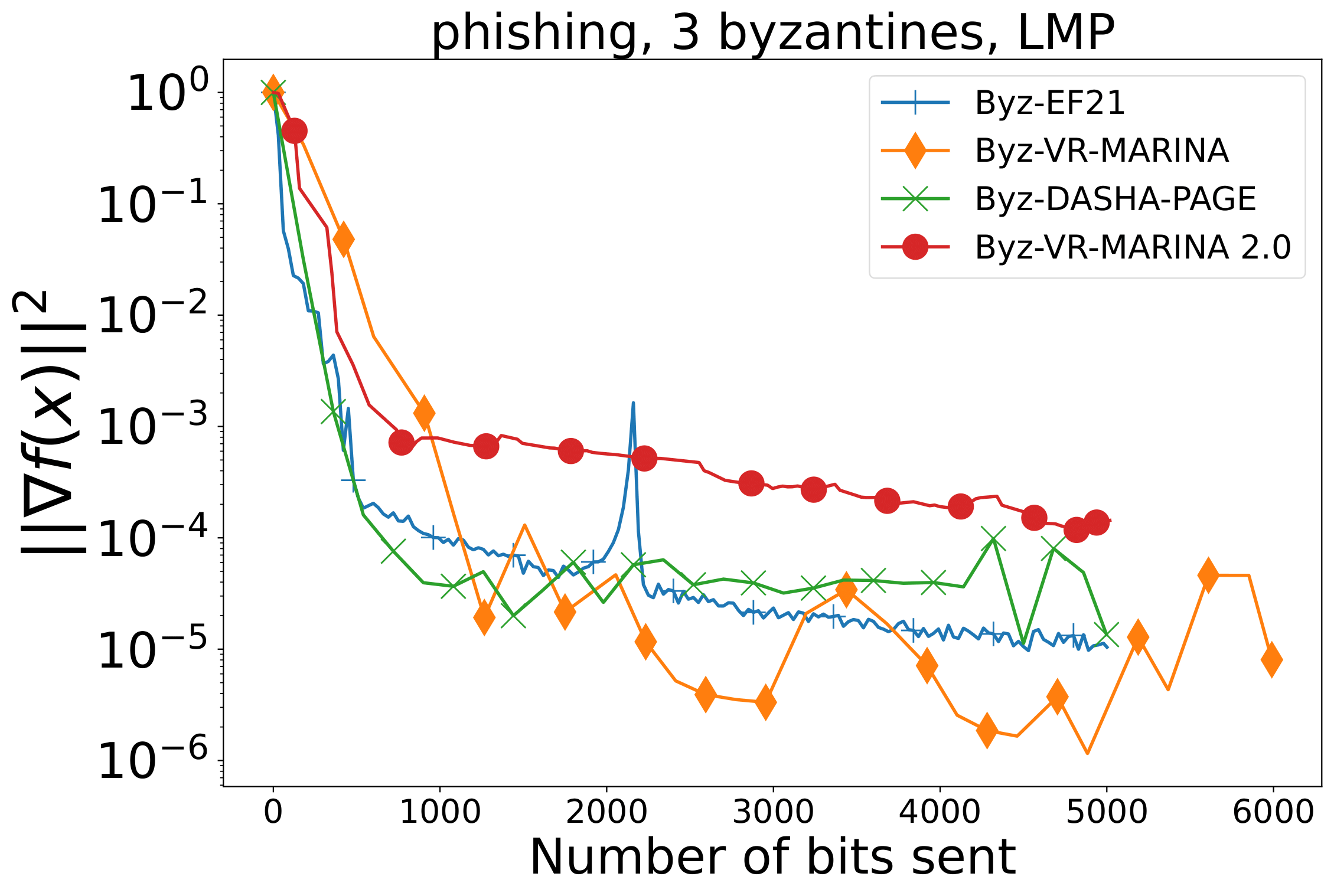}
  \caption{LMP attack, \texttt{phishing}}
\end{subfigure}%
\begin{subfigure}{.33\textwidth}
  \centering
  \includegraphics[width=1\linewidth]{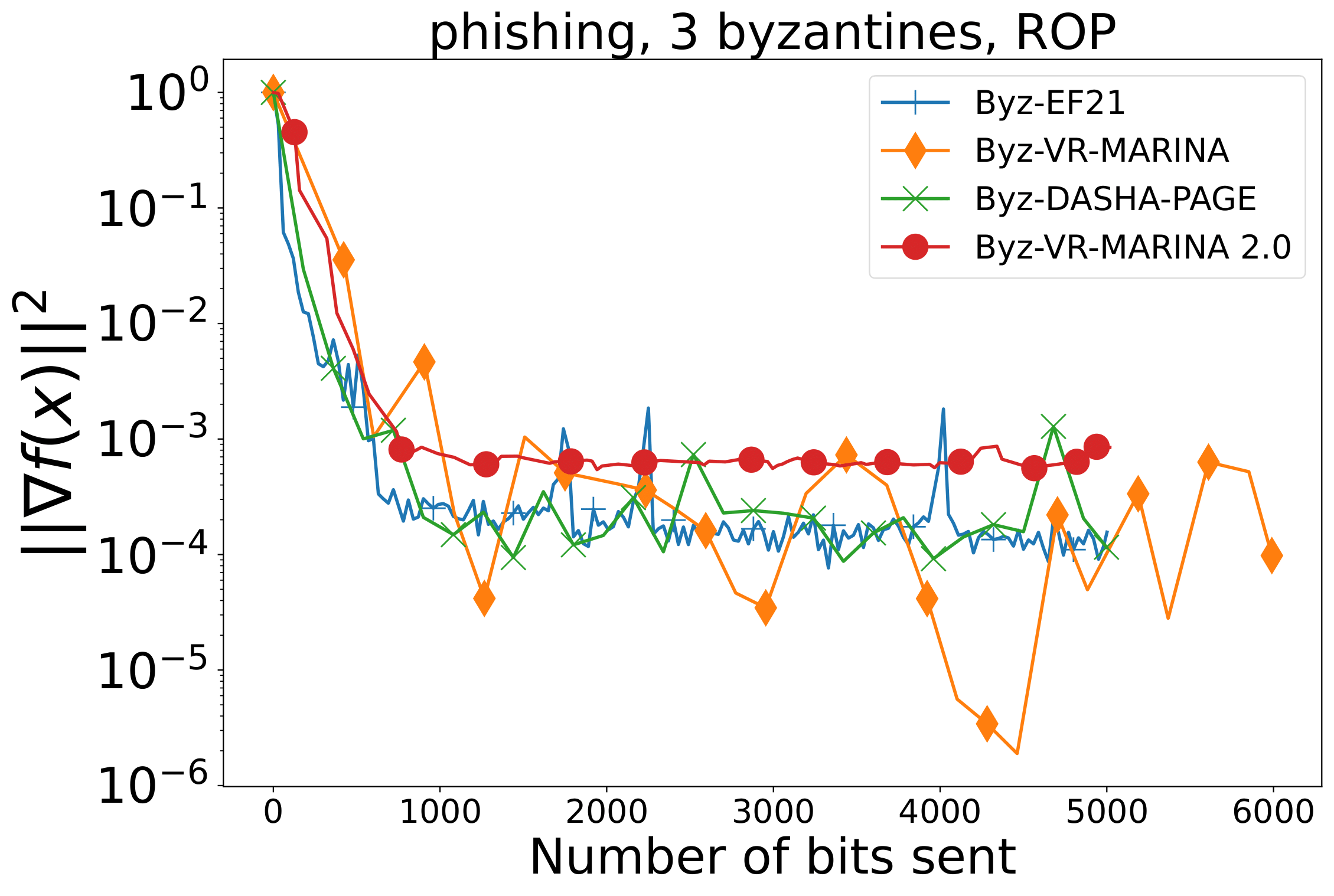}
  \caption{ROP attack, \texttt{phishing}}
\end{subfigure}\\
\begin{subfigure}{.33\textwidth}
  \centering
  \includegraphics[width=1\linewidth]{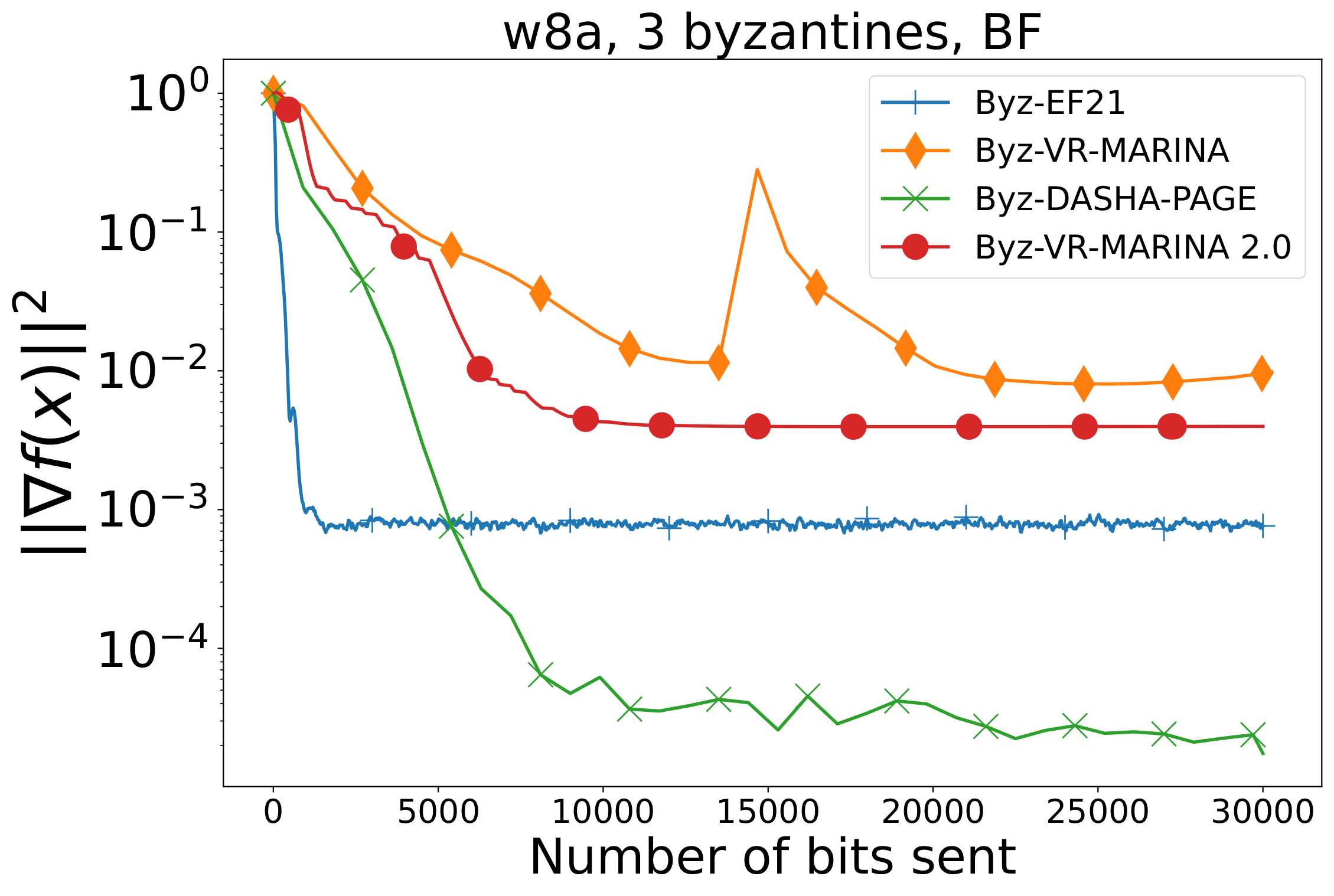}
  \caption{BF attack, \texttt{phishing}}
\end{subfigure}%
\begin{subfigure}{.33\textwidth}
  \centering
  \includegraphics[width=1\linewidth]{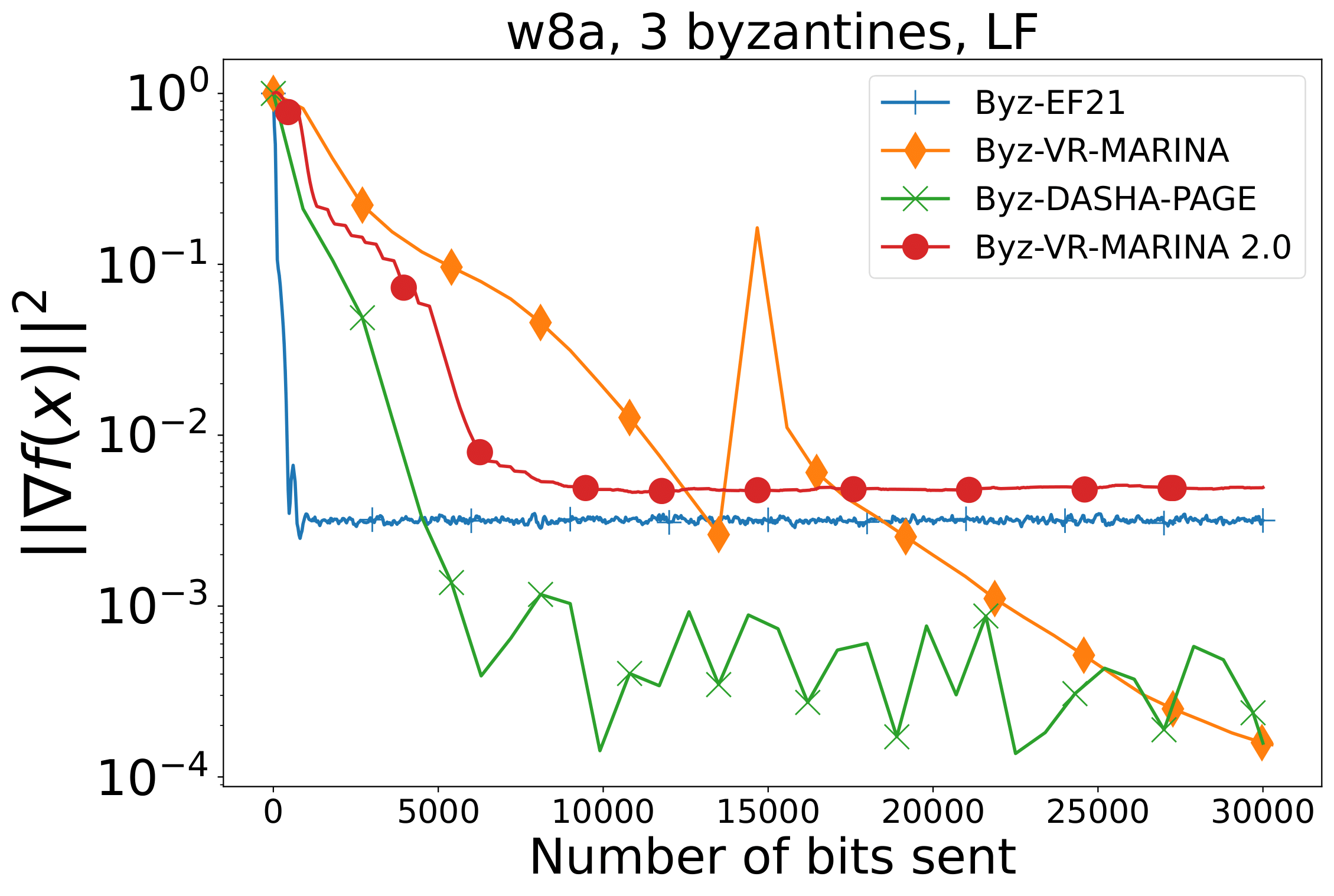}
  \caption{LF attack, \texttt{phishing}}
\end{subfigure}%
\begin{subfigure}{.33\textwidth}
  \centering
  \includegraphics[width=1\linewidth]{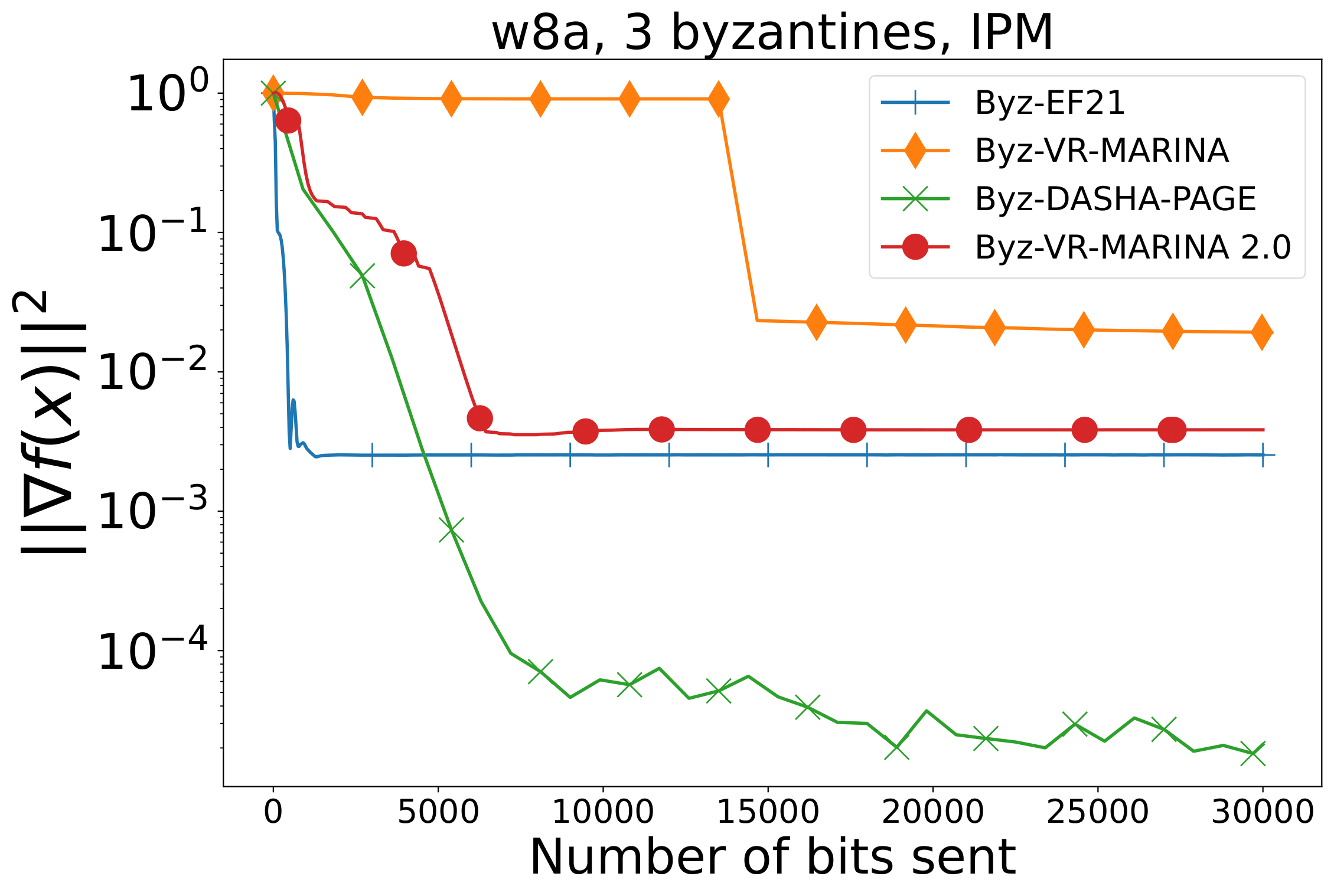}
  \caption{IPM attack, \texttt{phishing}}
\end{subfigure}\\
\begin{subfigure}{.33\textwidth}
  \centering
  \includegraphics[width=1\linewidth]{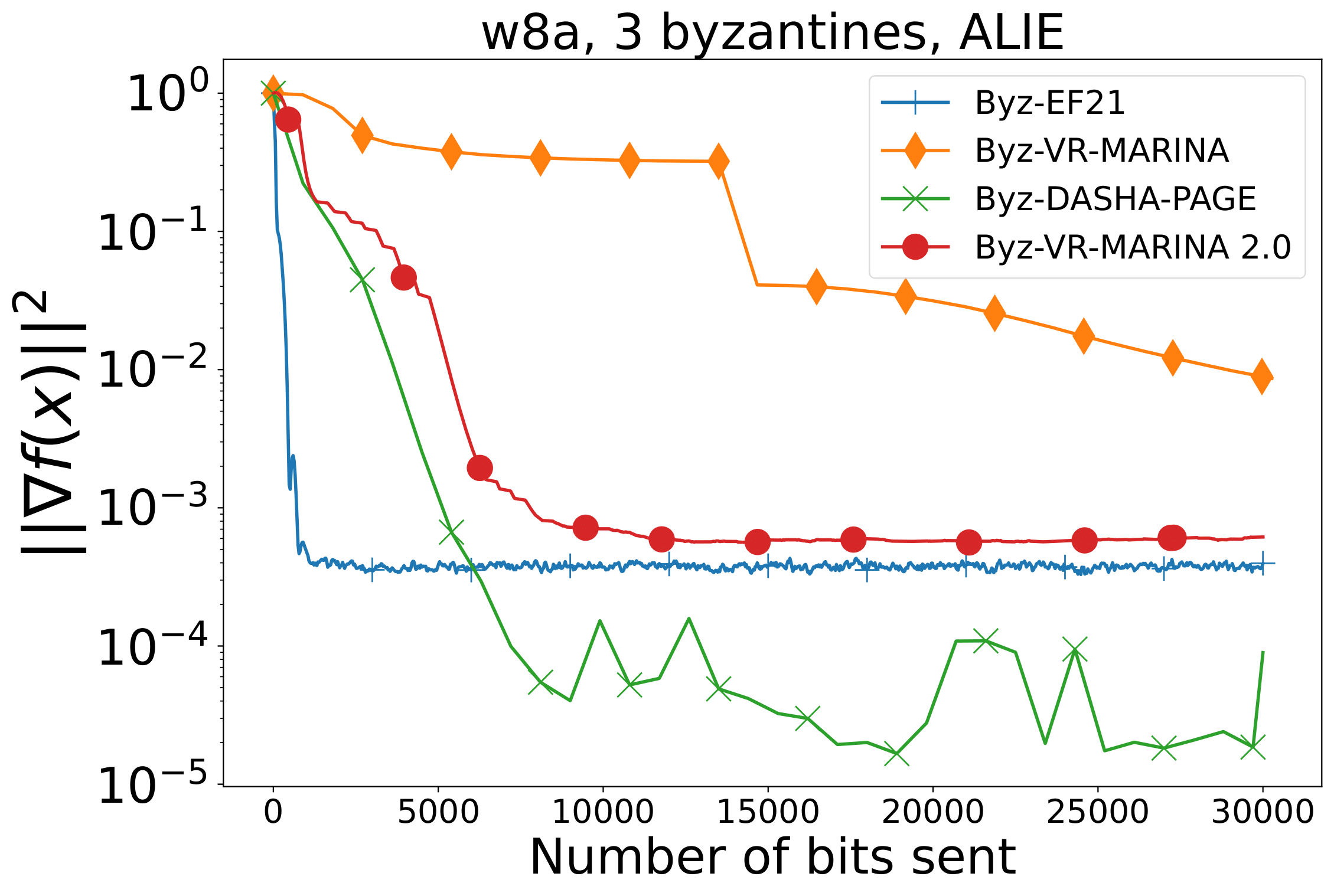}
  \caption{ALIE attack, \texttt{phishing}}
\end{subfigure}%
\begin{subfigure}{.33\textwidth}
  \centering
  \includegraphics[width=1\linewidth]{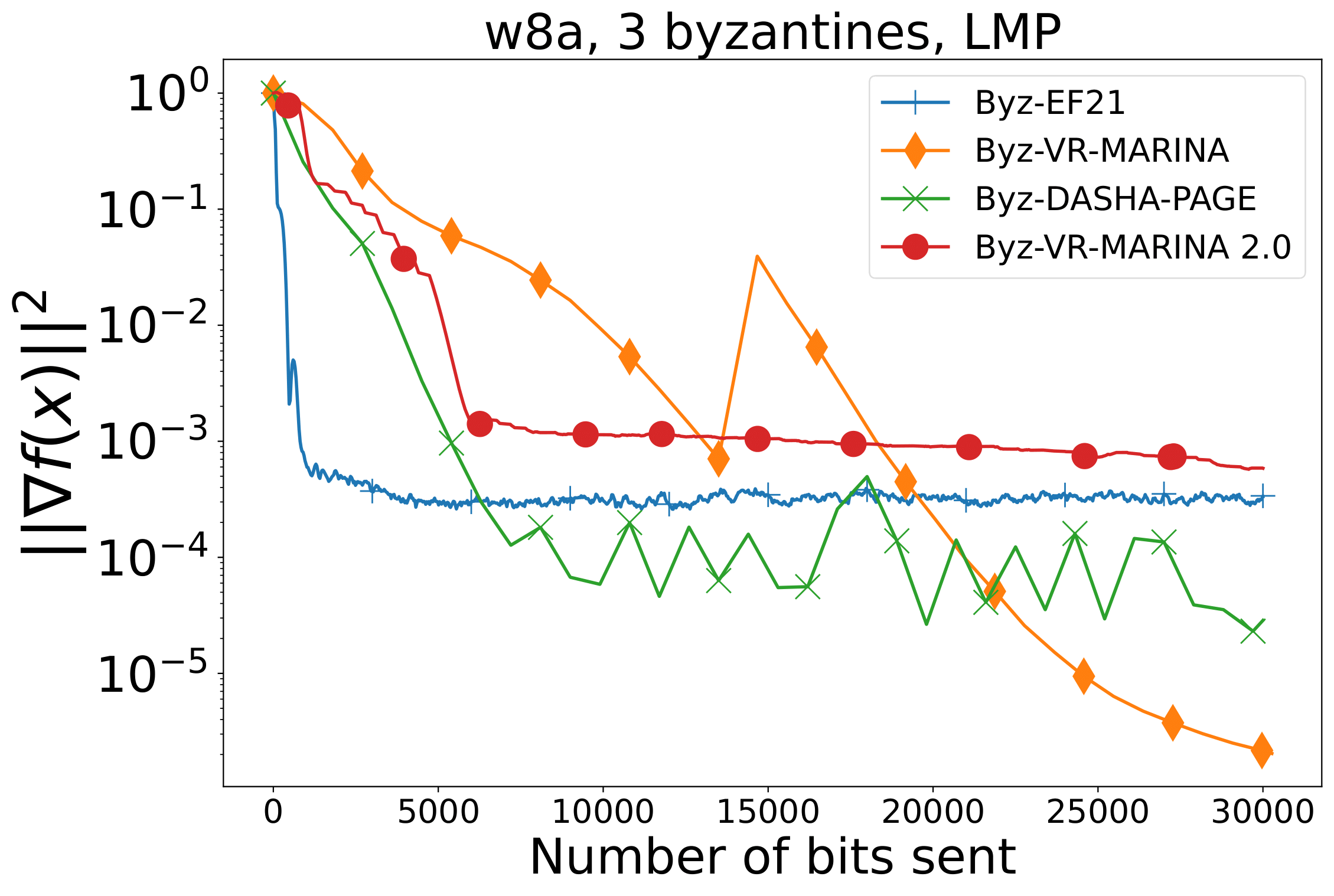}
  \caption{LMP attack, \texttt{phishing}}
\end{subfigure}%
\begin{subfigure}{.33\textwidth}
  \centering
  \includegraphics[width=1\linewidth]{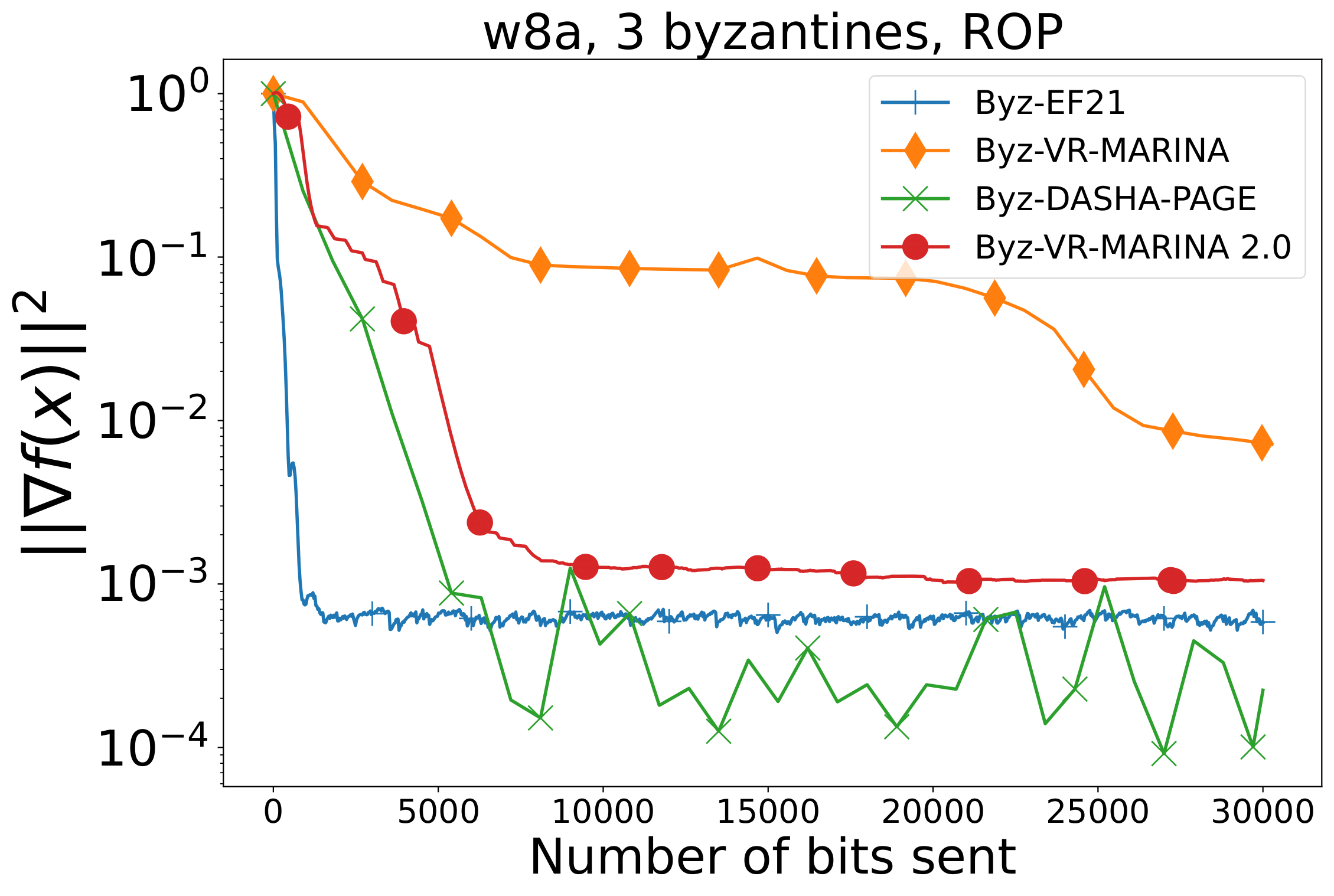}
  \caption{ROP attack, \texttt{phishing}}
\end{subfigure}%
\caption{%\centering
Communication complexity comparison in the heterogeneous non-convex setting on the \texttt{phishing} and \texttt{w8a} datasets for the methods with biased and unbiased compression.}
\label{fig:NC_biased_vs_unbiased}
\end{figure}

%%%%%%%%%%%%%%%%%%%%%%%%%%%%%%%%%%%%%%%%%%%%%%%%%%%%%%%%%%%%

\end{document}